\documentclass[preprint,10pt]{elsarticle}

\usepackage[utf8]{inputenc}				
\usepackage{lmodern} 
\usepackage{microtype}

\usepackage{bm}
\usepackage{amsfonts}
\usepackage{amsmath}
\usepackage{amssymb} 
\usepackage{amsthm}

\usepackage{float}
\usepackage{stmaryrd}
\usepackage{bbold}
\usepackage{setspace}
\usepackage{mathtools}
\usepackage{graphicx}
\usepackage{subcaption}
\usepackage{caption}
\usepackage{sidecap}
\usepackage{placeins}
\usepackage{enumerate}
\usepackage{enumitem}
\usepackage{hyperref}
\usepackage[dvipsnames]{xcolor}
\usepackage{tcolorbox}
\usepackage[a4paper, total={7in, 8in}]{geometry}
\usepackage{arydshln}
\usepackage{multicol}

\newcommand{\dy}{\text{d}\ney}
\newcommand{\dsx}{\text{d}s_{x}}
\newcommand{\dsy}{\text{d}s_{y}}

\newcommand{\reals}{\mathbb{R}}

\newcommand{\norm}[1]{\left\Vert #1\right\Vert}
\newcommand{\grad}{\operatorname{\mathbf{grad}}}
\newcommand{\dv}{\operatorname{div}}
\newcommand{\Curl}{\operatorname{\mathbf{curl}}}
\newcommand{\curl}{\operatorname{curl}}

\newcommand{\jump}[1]{\llbracket {#1}  \rrbracket}
\newcommand{\av}[1]{\left\{\!\!\left\{   #1 \right\}\!\!\right\} }

\newcommand{\neu}{\bm{u}}
\newcommand{\nev}{\bm{v}}
\newcommand{\nex}{\bm{x}}
\newcommand{\ney}{\bm{y}}

\newcommand{\llangle}{\langle\!\langle }
\newcommand{\rrangle}{\rangle\!\rangle}
\newcommand{\LLangle}{\left\langle\!\!\left\langle }
\newcommand{\RRangle}{\right\rangle\!\!\right\rangle}
\newcommand{\gammat}{\gamma_{\mathbf{t}}}
\newcommand{\gammatau}{\gamma_{\bm{\tau}}}
\newcommand{\gamman}{\gamma_{\mathbf{n}}}
\newcommand{\overbar}[1]{\mkern 1.5mu\overline{\mkern-1.5mu#1\mkern-1.5mu}\mkern 1.5mu}

\newcommand{\il}[1]{#1}

\newcommand{\blue}[1]{#1}

\newcommand{\highlightr}[2]{%
	\colorbox{#1!25!white}{\strut $\displaystyle#2$}}

\newcommand{\thighlightr}[2]{%
	\colorbox{#1!25!white}{\strut #2}}

\newcommand{\cdef}[1]{{\color{red}{#1}}}

\newtheorem{theorem}{Theorem}[section]
\newtheorem{lemma}[theorem]{Lemma}
\newtheorem{proposition}[theorem]{Proposition}
\newtheorem{corollary}[theorem]{Corollary}
\newtheorem{remark}[theorem]{Remark}

\newtheorem{problem}[theorem]{Problem}
\newtheorem{assumption}[theorem]{Assumption}
\theoremstyle{definition}
\newtheorem{proofpart}{Part}
\makeatletter 
\@addtoreset{proofpart}{theorem}
\makeatother

\begin{document}

\begin{frontmatter}



\title{Coupled Boundary and Volume Integral Equations for Electromagnetic Scattering}


\author[1]{Ignacio Labarca-Figueroa\corref{cor1}}
\ead{ignacio.labarca-figueroa@uibk.ac.at}
\cortext[cor1]{Corresponding author}
\affiliation[1]{organization={Institute for Theoretical Physics, University of Innsbruck},
            city={Innsbruck},
            country={Austria}}

\author[2]{Ralf Hiptmair}\ead{ralf.hiptmair@sam.math.ethz.ch}
\affiliation[2]{organization={Seminar for Applied Mathematics, ETH Zurich},
         	city={Zurich},
         	country={Switzerland}}

\begin{abstract}
We study frequency domain electromagnetic scattering at a bounded, penetrable, and inhomogeneous obstacle $ \Omega  \subset \reals^3 $.
From the Stratton-Chu integral representation, we derive a new representation formula when constant reference coefficients are given for the interior domain. The resulting integral representation contains the usual layer potentials, but also volume potentials on $\Omega$. Then it is possible to follow a single-trace approach to obtain boundary integral equations perturbed by traces of compact volume integral operators with weakly singular kernels. The coupled boundary and volume integral equations are discretized with a Galerkin approach with usual Curl-conforming and Div-conforming finite elements on the boundary and in the volume. Compression techniques and special quadrature rules for singular integrands are required for an efficient
and accurate method. Numerical experiments provide evidence
that our new formulation enjoys promising properties. 
\end{abstract}




\end{frontmatter}


	
\section{Introduction}\label{sec:Introduction}
\subsection{Maxwell Transmission Problem}\label{sec:Maxwell-transmission-problem}
We are interested in solving the frequency domain electromagnetic wave scattering problem in a medium that is homogeneous outside a bounded region $ \Omega_i  \subset \reals^3$ (see Figure \ref{fig:geom}). We denote the exterior domain $ \Omega_o \coloneqq \reals^3 \setminus \overbar{\Omega}_i. $ Material properties are given by functions $ \varepsilon\in L^{\infty}(\reals^3)$ and $ \mu \in L^{\infty}(\reals^3)$ where
\begin{equation}\label{eq:material_params}
	\varepsilon(\nex) \equiv \varepsilon_0, \quad \mu(\nex) \equiv \mu_0\quad \text{ for } \nex \in \Omega_o,
\end{equation}
and $ \varepsilon_{\text{max}} > \varepsilon(\nex) > \varepsilon_{\text{min}} > 0, \ \mu_{\text{max}} > \mu(\nex) > \mu_{\text{min}} > 0 $ almost everywhere in $ \reals^3. $ \\

\begin{figure}[t]
	\centering
	\includegraphics[width=0.5\linewidth]{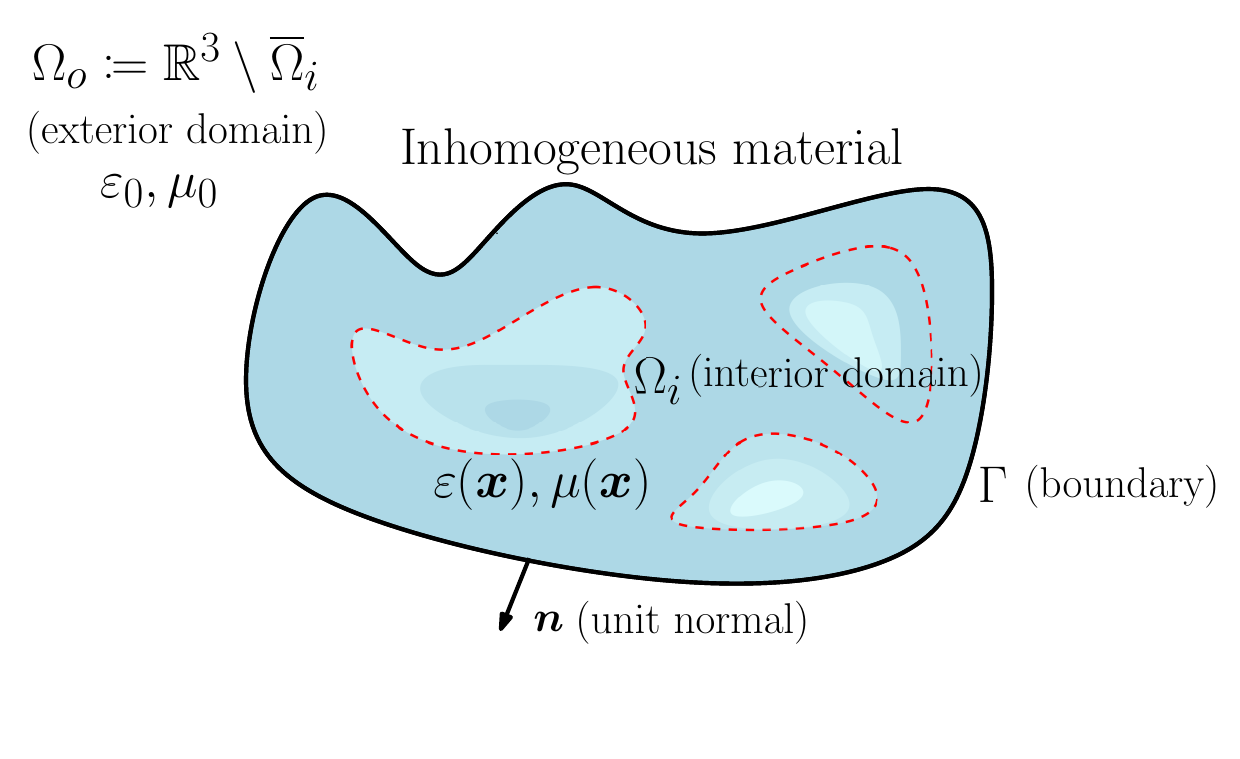}
	\caption{Geometric setting. Inhomogeneous material.}
	\label{fig:geom}
\end{figure}
The equations governing the problem of finding the total electric field $ \neu \coloneqq \neu^s + \neu^{\text{inc}} $ and total magnetic field $ \nev \coloneqq \nev^s + \nev^{\text{inc}}  $  in this inhomogeneous medium are
\begin{equation}\label{eq:inh_helmholtz}
	\Curl \neu - i\omega \mu(\nex)\nev = 0, \quad \Curl \nev + i\omega \varepsilon(\nex)\neu = 0, \quad \text{ for }\nex \in \reals^3,
\end{equation}
where $ \neu^{\text{inc}}, \nev^{\text{inc}}  $ are the incident fields satisfying the vacuum Maxwell's equations in the whole space, 
\begin{equation}\label{eq:uinc_pde}
	\Curl \neu^{\text{inc}} - i\omega \mu_0\nev^{\text{inc}} = 0, \quad \Curl \nev^{\text{inc}} + i\omega \varepsilon_0\neu^{\text{inc}} = 0, \quad \text{ for }\nex \in \reals^3,
\end{equation}
and $ \neu^s, \nev^s  $ satisfy Silver-M\"uller radiation conditions \cite[Chapter~6]{colton2012inverse} 
\begin{equation}\label{eq:radiation}
	\lim\limits_{r\rightarrow\infty}	\nev^s\times \dfrac{\nex}{r} - \neu^s = 0, \text{ uniformly on }r=|\nex|.
\end{equation}
The problem can be formulated as the following transmission problem: 
\begin{tcolorbox}[colback=lightgray!15!white, colframe=black, sharp corners, colframe=lightgray!15!white]
	Find $ \neu, \nev \in \mathbf{H}_{\text{loc}}(\Curl, \reals^3\setminus\Gamma) $ such that
	\begin{align}\label{eq:transmission-problem-maxwell}\begin{aligned}
			&\begin{aligned}
				&\Curl \neu - i\omega \mu_0 \nev &&= 0, &&\Curl \nev + i\omega \varepsilon_0 \neu &&= 0 &&\text{ in }\Omega_o,\\
				&\Curl \neu - i\omega \mu(\nex) \nev &&= 0, && \Curl \nev + i\omega \varepsilon(\nex) \neu &&= 0 &&\text{ in }\Omega_i ,\\
				&\gammat^+ \neu - \gammat^- \neu&&= -\gammat \neu^{\text{inc}}, && \gammat^+ \nev - \gammat^- \nev&&= -\gammat \nev^{\text{inc}}&&\text{ on }\Gamma, \\
			\end{aligned}\\ 
			&\qquad\qquad\lim\limits_{r\rightarrow\infty}	\nev\times \dfrac{\nex}{r} - \neu = 0, \text{ uniformly on }r=|\nex|,
		\end{aligned}
	\end{align}
	where $ \gammat^{\pm}$ denotes the exterior/interior tangential trace operators (see Section \ref{sec:Preliminaries} for details).\\
\end{tcolorbox}

\subsection{VIEs for electromagnetic scattering}\label{sec:VIE}
\blue{It is possible to reformulate the transmission problem \eqref{eq:transmission-problem-maxwell} as a volume integral equation (VIE).} Depending on the material properties, \blue{this can be done in different ways, for instance as the VIE} \cite{botha2006solving, markkanen2020new}:
\begin{tcolorbox}[colback=lightgray!15!white, colframe=black, sharp corners, colframe=lightgray!15!white]
	Find $\neu\in \mathbf{H}(\dv, \Omega_i)\cap \mathbf{H}(\Curl, \Omega_i)$ such that
	\begin{equation}\label{eq:EVIE}
		\neu - \grad\dv\mathbf{N}_{\Omega_i,\kappa_0}(p_e \neu) - \kappa_0^2\mathbf{N}_{\Omega_i,\kappa_0}(p_e\neu) - \Curl\mathbf{N}_{\Omega_i,\kappa_0}(q_m\Curl\neu) = \neu^{\text{inc}},
	\end{equation}
	where $p_e(\nex) \coloneqq 1 - \tfrac{\varepsilon(\nex)}{\varepsilon_0}, \ q_m(\nex) \coloneqq 1 - \tfrac{\mu_0}{\mu(\nex)},$ and $\mathbf{N}_{\Omega_i,\kappa_0}$ is the Newton potential over $\Omega_i$ with wavenumber $\kappa_0$ (introduced in Section \ref{sec:Newton}).
\end{tcolorbox}
Variants of \eqref{eq:EVIE} can be found in \cite{costabel2010volume,costabel2012essential, markkanen2020new,botha2006solving}.
The operators involved in these formulations are not compact in $\mathbf{H}(\Curl, \Omega_i)$ or $\mathbf{H}(\dv, \Omega_i)$. Most of the equations \blue{involve} integral operators with strongly singular kernels. Therefore, Fredholm theory can not be used directly, as the operators underlying the VIEs fail to be compact perturbations of the identity. Spectral properties of the volume integral operators (VIOs) have been studied, with results in the continuous setting \cite{costabel2010volume, costabel2012essential} and numerical experiments for the discrete setting \cite{markkanen2016numerical}. Well-posedness of their discretizations is not available for the existing formulations. Galerkin discretizations, although widely used in literature, are not guaranteed to be stable or converge in appropriate normed spaces. \blue{In \cite{chandler2022coercivity} the authors demonstrated that Galerkin methods for second-kind BIEs in $L^2(\Gamma)$ fail even though they rely on asymptotically dense sequence of subspaces of $L^2(\Gamma)$. This could also affect second-kind VIEs though a rigorous analysis is still open.
}
\subsection{BIEs for piecewise-constant coefficients}
For the particular case of \emph{piecewise-constant material properties}, BIEs can be used to obtain stable formulations for the transmission problem. First and second-kind BIEs can be employed \cite{buffa2003galerkin,claeys2012electromagnetic,claeys2017second, spindler2016second}. In this article we focus on the first-kind single-trace formulation (STF) from \cite[Section~7.1]{buffa2003galerkin}, in the engineering community also known as the Poggio-Miller-Chang-Harrington-Wu-Tsai (PMCHWT) formulation \cite{chang1977surface,poggio1970integral,wu1977scattering}. This formulation can be extended to the setting of composite scatterers with piecewise-constant material properties. The STF BIEs for \eqref{eq:transmission-problem-maxwell} with \emph{piecewise-constant coefficients} have the following structure:
\begin{tcolorbox}[colback=white!15!white, colframe=white]
	Find $\bm{\alpha}\in \mathbf{H}^{-1/2}(\curl_{\Gamma}, \Gamma)$ and $\bm{\beta}\in \mathbf{H}^{-1/2}(\dv_{\Gamma}, \Gamma)$ such that
	\begin{equation}\label{eq:STF-BIE}
		\left(\mathbf{M}^{-1}\mathbf{A}_{\kappa_0}\mathbf{M} + \mathbf{A}_{\kappa_1}\right)\begin{pmatrix}
			\bm{\alpha} \\ \bm{\beta}
		\end{pmatrix} = - \mathbf{M}^{-1}\begin{pmatrix}
			\gammat \neu^{\text{inc}} \\ \gammatau\tilde{\nev}^{\text{inc}}
		\end{pmatrix}
	\end{equation}
	in $\mathbf{H}^{-1/2}(\curl_{\Gamma}, \Gamma) \times \mathbf{H}^{-1/2}(\dv_{\Gamma}, \Gamma),$ where $\mathbf{A}_{\kappa_{\star}}$ is the Maxwell's Calder\'on operator with wavenumber $\kappa_{\star}$ and $\mathbf{M}$ component scaling operator (see Section \ref{sec:BIOs}).
\end{tcolorbox}
For piecewise-constant coefficients, BIEs are arguably the best option as a formulation for the transmission problem. Solving BIE formulations with the boundary-element method (BEM) offers an accurate and efficient approach. Matrix compression techniques such as $\mathcal{H}$ and $\mathcal{H}^2$-matrices \cite{bebendorf2008hierarchical,borm2010efficient} significantly reduce the cost of storing and solving the dense linear systems arising from a BEM discretization.\\ 

\subsection{FEM-BEM coupling}\label{sec:FEM-BEM}
A widely used approach to the discretization of the transmission problem \eqref{eq:transmission-problem-maxwell} relies on the coupling of a volume variational formulation in $\Omega_i$ with boundary integral equations realizing the Dirichlet-to-Neumann map for $\Omega_o$. Subsequent Galerkin finite-element discretization leads to schemes known as FEM-BEM coupling. Different couplings can be obtained depending on the choice of boundary integral equations (BIEs) for the coupling, such as Johnson-N\'ed\'elec \cite{johnson1980coupling}, Bielak-MacCamy \cite{bielak1983exterior} or Costabel-Han approaches \cite{costabel1987symmetric,hou1990new,hiptmair2003coupling}. Formulations robust with respect to the wavenumber have also been studied in \cite{hiptmair2008stabilized}. We use solutions produced by FEM-BEM coupling as reference in Section \ref{sec:numerics}.\\

\subsection{STF-VIEs}
One drawback of the approaches mentioned in Sections \ref{sec:VIE} and \ref{sec:FEM-BEM} is that these methods do not benefit from a piecewise-constant material. Neither classical VIEs nor FEM-BEM coupled formulations reduce to pure BIEs when applied in the special case of piecewise-constant coefficients.
Our interest is to study an extended formulation based on boundary and volume integral operators. The approach is similar to \cite{usner2006generalized}, and the analysis follows closely that in the case of acoustic scattering \cite{labarca2023, labarca2024}, with a few differences that are particular to Maxwell equations. Similar ideas combining BIEs and VIEs can also be found in \cite{munger2023multi,munger2023single,olyslager2023volume}. Starting from the Stratton-Chu integral representation, we derive a new combined integral representation for the electric and magnetic fields. For the case of piecewise-constant coefficients, the formulation reduces to the simple case of first-kind BIEs \eqref{eq:STF-BIE}. \il{As in \cite{labarca2023} we do not pursue second-kind formulations due to the impossibility to remove leading singularities of integral kernels in modified Calder\'on operators by subtraction (see Remark \ref{second-kind}).} The volume integral operators can be shown to be compact, and only supported in the domain of inhomogeneity (i.e. not necessarily the whole domain $\Omega$, see Figure \ref{fig:localized}). \il{This last point is related to a potential computational advantage in our setting: the computational domain in the volume may be reduced to a smaller domain containing the support of the inhomogeneities. This leads to a significantly reduced number of degrees of freedom in the volume, compared to a finite-element setting in which the finite-element spaces always have to cover the whole domain $\Omega_i$. In particular, if $N_{\text{ih}}>0$ is the number of degrees of freedom related to a region where the material properties are inhomogeneous, and $N$ is the total number of degrees of freedom in the volume, then the computational complexity of solving the problem with iterative methods can be reduced to $\mathcal{O}(N_{\text{ih}}^2)$ instead of $\mathcal{O}(N^2)$.\\
	
	 Moreover, it is well-known that $h$--FEM formulations for high-frequency problems suffer from the pollution effect \cite{babuska1997pollution,galkowski2023does, melenk2011wavenumber}. We present some numerical evidence that indicates a better frequency behavior of coupled boundary-volume integral equations in this matter, compared to a FEM-BEM coupling. } \\

\begin{figure}[h!]
	\begin{subfigure}[b]{0.45\textwidth}
		\includegraphics[width=1.0\linewidth]{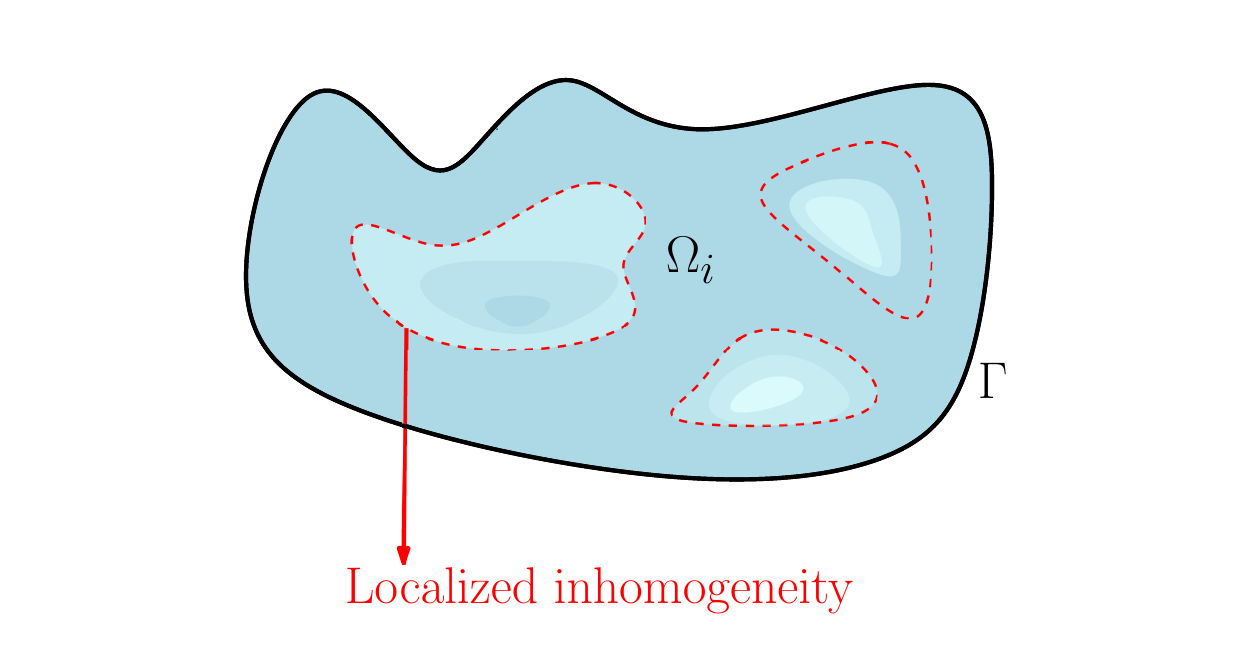}
		\caption{Region to be meshed for VIEs and FEM.}
		\label{fig:DomainVIEs}	
	\end{subfigure}\hfill
	\begin{subfigure}[b]{0.45\textwidth}
		\includegraphics[width=1.0\linewidth]{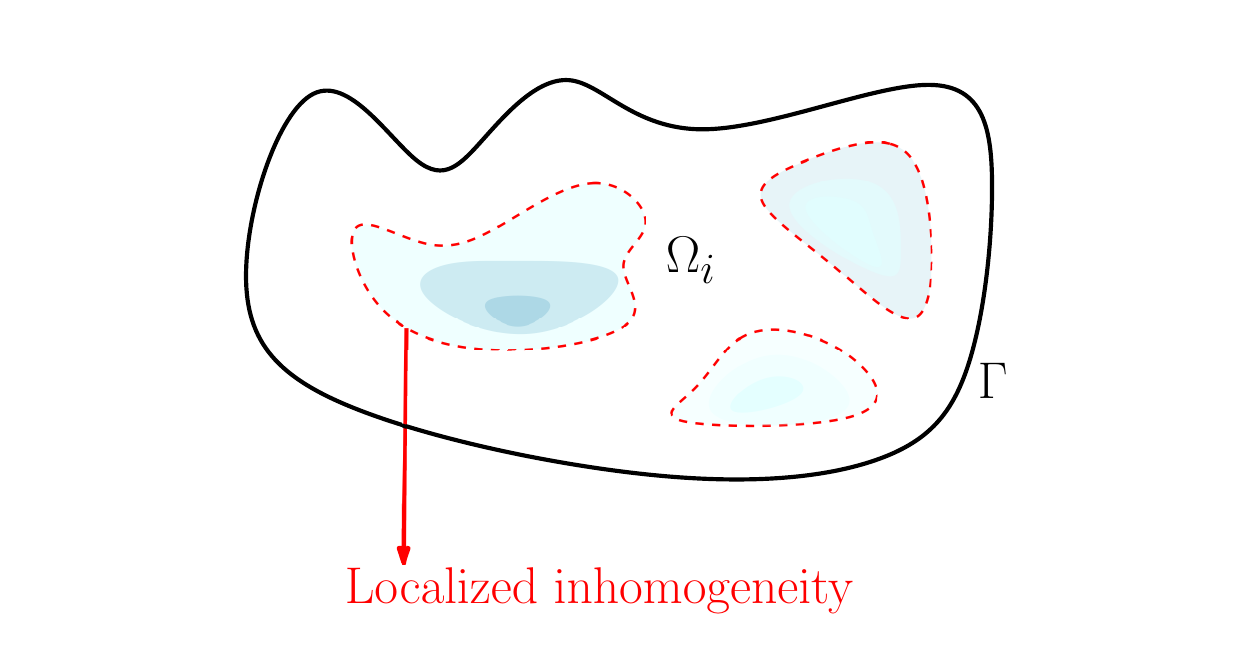}	
		\caption{Region to be meshed in our formulation.}
		\label{fig:DomainBVIEs}
	\end{subfigure}
	\caption{Domains for transmission problems with spatially varying coefficients.}
	\label{fig:localized}
\end{figure}

The setting of our work is as follows:
\begin{tcolorbox}[colback=lightgray!15!white, colframe=black, sharp corners, colframe=lightgray!15!white]
	\begin{assumption}\label{assumption}
		Throughout we make the following assumptions 
		\begin{enumerate}
			\item $ \Omega_i$ is a bounded Lipschitz domain with boundary $ \Gamma. $
			\item The parameters are smooth inside $\Omega_i$: $ \varepsilon, \mu \in C^1(\overbar{\Omega}_i)\cap C^2(\Omega_i). $
			\item There are positive constants $\varepsilon_{\min},\mu_{\min}, \varepsilon_{\max}, \mu_{\max} $ such that
			\begin{align*}
				&\varepsilon_{\min} \leq \varepsilon(\nex) \leq \varepsilon_{\max}, 
				&\mu_{\min} \leq \mu(\nex) \leq \mu_{\max},
			\end{align*}
			for all $ \nex \in \Omega_i. $
			\item Reference coefficients $ \varepsilon_1,\mu_1\in\reals $ are chosen such that $\mu_1 > 0, \varepsilon_1>0. $
		\end{enumerate}
	\end{assumption}
\end{tcolorbox}

In constrast with the acoustic scattering approach, for Maxwell problems we need different techniques. Problems are no longer coercive, but $T$-coercive \cite{ciarlet2012}. Discrete stability now depends on $h$-uniform inf-sup conditions, equivalent to $T_h$-coercivity. First order formulations play a central role, due to the symmetry between electric and magnetic fields. Finally, we observed an interesting problem when discretizing volume integral equations: discrete stability of duality pairings can not be taken for granted as in the scalar case.\\
The required stability estimates are not readily available as in the case of $H^1(\Omega)$ and its dual space $\widetilde{H}^{-1}(\Omega)$.\\

\subsection{Outline and main results}

In Section \ref{sec:Preliminaries} we introduce the preliminaries for the functional setting in which we study our equations. We present the derivation of the representation formula in Section \ref{sec:integralrep}. Our new representation formula is written in Section \ref{sec:STF}, \eqref{eq:new-rep}, and we state the variational formulation in Problem \ref{prob:varSTF}, Section \ref{sec:analysis}.\\
In Section \ref{sec:analysis} we study the continuous problem using standard techniques: Fredholm theory and T-Coercivity. In Theorem \ref{th:well-posed} we establish the well-posedness of Problem \ref{prob:varSTF}.\\ 
Results about the Galerkin discretization are presented in Section \ref{sec:Galerkin}. Numerical experiments that validate our formulation are shown in Section \ref{sec:numerics}.

\subsection*{List of symbols}

	\begin{align*}
		&\textbf{Symbol}   &&	\textbf{Description}										&&		\textbf{Section}				\\
		&\varepsilon, \mu 		&& \text{Material coefficients varying in space} && \text{Section \ref{sec:Introduction}}, \eqref{eq:material_params}\\
		&C^{\infty}     && \text{Spaces of smooth functions} && \text{Section \ref{sec:Preliminaries}}\\
		&C^{\infty}_0   && \text{Smooth functions vanishing on the boundary} && \text{Section \ref{sec:Preliminaries}}\\
		&C^{\infty}_{\text{comp}}    && \text{Compactly supported smooth functions} && \text{Section \ref{sec:Preliminaries}}\\
		&H^s,\ H^s_{\text{loc}},\ H^s_{\text{comp}}   && \text{Scalar Sobolev spaces of order }s && \text{Section \ref{sec:Preliminaries}}\\
		&\mathbf{H}^s,\ \mathbf{H}^s_{\text{loc}}, \ \mathbf{H}^s_{\text{comp}}&& \text{Vector Sobolev spaces of order }s && \text{Section \ref{sec:Preliminaries}}\\
		&H^s(\Gamma) , \mathbf{H}^{s}(\Gamma)  && \text{Scalar/vector Sobolev space of order }s \text{ on } \Gamma  && \text{Section \ref{sec:Preliminaries}}\\
		&H^{-1/2}(\Gamma),\ \mathbf{H}^{-1/2}(\Gamma)   && \text{Dual spaces of }H^{1/2}(\Gamma)\text{ and }\mathbf{H}^{1/2}(\Gamma) && \text{Section \ref{sec:Preliminaries}}\\
		&\gamma, \partial_n   && \text{Dirichlet/Neumann/Normal trace operators} && \text{Section \ref{sec:Preliminaries}}\\
		&\gammat, \gammatau, \gamman   && \text{Tangential and normal trace operators} && \text{Section \ref{sec:Preliminaries}}\\
		&\mathbf{H}^{-1/2}(\dv_{\Gamma}, \Gamma)  && \text{Maxwell Trace space }\gammatau(\mathbf{H}(\Curl, \Omega)) && \text{Section \ref{sec:Preliminaries}}\\
		&\mathbf{H}^{-1/2}(\curl_{\Gamma}, \Gamma)  && \text{Maxwell Trace space }\gammat(\mathbf{H}(\Curl, \Omega))  && \text{Section \ref{sec:Preliminaries}}\\
		&G_{j}   && \text{Fundamental solution with wavenumber } \kappa_j && \text{Section \ref{sec:Newton}}, \eqref{eq:fundamental_solution}\\
		&\mathsf{N}_{j},\  \mathbf{N}_j   && \text{Scalar and vector Newton potential} && \text{Section \ref{sec:Newton}}, \eqref{eq:NewtonR3}\\
		&\mathsf{N}_{\Omega, j}, \ \mathbf{N}_{\Omega, j}   && \text{Scalar and vector Newton potential (local)} && \text{Section \ref{sec:Newton}}, \eqref{eq:mapping_newton_omega},\eqref{eq:mappingVecNewton}\\
		&\bm{\mathcal{T}}_j, \ \bm{\mathcal{D}}_j   && \text{Maxwell layer potentials with wavenumber }\kappa_j && \text{Section \ref{sec:Newton}},\eqref{eq:MaxwellSL},\eqref{eq:MaxwellDL} \\
		&\bm{\mathcal{T}}_j^{\tilde{\varepsilon}, \tilde{\mu}}  && \text{Weighted Maxwell single layer potential} && \text{Section \ref{sec:Newton}},\eqref{eq:weightedT} \\
		&\mathbf{V}_j, \ \mathbf{K}_j, \ \mathbf{K}'_j, \ \mathbf{W}_j   && \text{BIOs with wavenumber }\kappa_j && \text{Section \ref{sec:BIOs}}\\
		&\mathbf{A}_j   && \text{Calder\'on operator} && \text{Section \ref{sec:BIOs}}, \eqref{eq:Calderon}\\
		&\varepsilon_1, \mu_1   && \text{Constant reference coefficients} && \text{Section \ref{sec:integralrep}}, \eqref{eq:transmission_problem_pwsmooth}\\
		&p_e, p_m  && \text{Contrast functions with reference coefficients} && \text{Section \ref{sec:integralrep}}, \eqref{eq:transmission_problem_pwsmooth}\\
		&\tilde{\varepsilon}, \tilde{\mu}  && \text{Scaled material coefficients} && \text{Section \ref{sec:STF}}, \eqref{eq:scaled-coeffs}\\
		&\mathbf{V}_j^{\tilde{\varepsilon}, \tilde{\mu}} , \ \mathbf{W}_j^{\tilde{\mu}, \tilde{\varepsilon}}    && \text{Weighted BIOs with wavenumber }\kappa_j && \text{Section \ref{sec:STF}},\eqref{eq:weightedBIOs}\\
		&\mathbf{A}_j^{\tilde{\varepsilon}, \tilde{\mu}}   && \text{Weighted Calder\'on operator} && \text{Section \ref{sec:STF}}, \eqref{eq:weightedCalderon}\\
		&\mathbf{\Lambda}^e,\mathbf{\Lambda}^m,\mathbf{\Xi}^e, \mathbf{\Xi}^m   && \text{Volume integral operators (VIOs)} && \text{Section \ref{sec:STF}}, \eqref{eq:VolAK}\\
		&\mathbf{J}^e, \mathbf{J}^m  && \text{Operators related to traces of VIOs} && \text{Section \ref{sec:STF}}, \eqref{eq:tracesVolE},\eqref{eq:tracesVolM}\\
		&\mathbf{M}  && \text{Diagonal multiplier} && \text{Section \ref{sec:STF}}, \eqref{eq:STF-first}
	\end{align*}
	\newpage
	\section{Derivation of VIEs}
	\subsection{Preliminaries: Function spaces and trace operators}\label{sec:Preliminaries}
	
	Let $ \Omega_i  \subset \reals^3 $ be a Lipschitz domain, $ \Gamma \coloneqq \partial \Omega_i  $ its Lipschitz boundary with outward unit normal $ \bm{n} $. We rely on standard Sobolev spaces $ H^s(\Omega_i) $ of order $ s > 0 $. We also denote as $ \widetilde{H}^{-s}(\Omega_i ) $ the dual space of $ H^s(\Omega_i ) $ \cite[Section~3]{mclean2000strongly}. Spaces of compactly supported (resp. locally integrable) functions will be denoted with a sub-index {\itshape{comp}} (resp. {\itshape{loc}}), as in $H^s_{\text{comp}}(\Omega_i)$. Sobolev spaces on the boundary $ \Gamma $ are denoted as $ H^{s+1/2}(\Gamma). $ They arise naturally as boundary restrictions of elements of $ H^{s+1}(\Omega_i ) $ by the interior Dirichlet trace operator
	\begin{align*}
		&\gamma^- : H^{s+1}(\Omega_i)\rightarrow H^{s+1/2}(\Gamma), \quad 0\leq s \leq \tfrac{1}{2},\\
		&\cdef{\gamma^- u} \coloneqq u|_{\Gamma }, \quad \text{ for }u\in C^{\infty}(\overbar{\Omega}_i), 
	\end{align*}
	which is a bounded operator \cite[Theorem~3.37]{mclean2000strongly}. Note that we use boldface symbols to indicate vector-valued functions and function spaces of vector fields. We define the interior normal (component) trace operator $ \gamman $ \cite[Theorem~3.24]{monk2003finite}
	\begin{align*}
		& \gamman^-  : \mathbf{H}(\dv, \Omega_i) \rightarrow H^{-1/2}(\Gamma),\\ 
		&\cdef{\gamman^- 	\neu} \coloneqq \neu|_{\Gamma} \cdot \bm{n}, \quad \text{ for }\neu\in [C^{\infty}(\overbar{\Omega}_i )]^3,  
	\end{align*}
	where the space $ \mathbf{H}(\dv, \Omega_i ) $ is defined as
	\begin{equation*}
		\cdef{\mathbf{H}(\dv, \Omega_i ) }\coloneqq \left\{\bm{u}\in [L^2(\Omega_i )]^3 \ : \ \dv\bm{u}\in L^2(\Omega_i )  \right\},
	\end{equation*}
	the interior tangential (component) trace operator $ \gammat $ \cite[Theorem~4.1]{buffa2002traces}
	\begin{align*}
		&\gammat^- : \mathbf{H}(\Curl, \Omega_i ) \rightarrow \mathbf{H}^{-1/2}(\text{curl}_{\Gamma}, \Gamma),\\ 
		&\cdef{\gammat^-\bm{u}} \coloneqq  \bm{n} \times (\bm{u}|_{\Gamma} \times \bm{n}) , \quad \text{ for }\bm{u}\in [C^{\infty}(\overbar{\Omega}_i )]^3,  
	\end{align*}
	and the rotated tangential (component) trace operator $ \gammatau $ \cite[Theorem~4.1]{buffa2002traces}
	\begin{align*}
		&\gammatau^- :  \mathbf{H}(\Curl, \Omega_i) \rightarrow \mathbf{H}^{-1/2}(\dv_{\Gamma}, \Gamma),\\ 
		&\cdef{\gammatau^-\bm{u}} \coloneqq  \bm{u}|_{\Gamma} \times \bm{n} , \quad \text{ for }\bm{u}\in [C^{\infty}(\overbar{\Omega}_i)]^3, 
	\end{align*}
	where the occurring spaces are defined as
	\begin{align*}
		\cdef{\mathbf{H}(\Curl, \Omega_i)} &\coloneqq \left\{\bm{u}\in \mathbf{L}^2(\Omega_i)\ : \ \Curl \neu\in \mathbf{L}^2(\Omega_i)  \right\}, \\
		\cdef{\mathbf{H}^1(\Curl, \Omega_i)} &\coloneqq \left\{\bm{u}\in \mathbf{H}^1(\Omega_i)\ : \ \Curl \neu\in \mathbf{H}^1(\Omega_i)  \right\}, \\
		\cdef{\mathbf{H}^{-1/2}(\text{curl}_{\Gamma}, \Gamma)} &\coloneqq \left\{ \bm{\mu} \in \mathbf{H}_{\parallel}^{-1/2}(\Gamma) \ : \ \curl_{\Gamma} \bm{\mu} \in H^{-1/2}(\Gamma)\right\}, \text{\cite[Theorem~4.1]{buffa2002traces}}\\
		\cdef{	\mathbf{H}^{-1/2}(\dv_{\Gamma}, \Gamma)} &\coloneqq \left\{ \bm{\mu} \in \mathbf{H}_{\times}^{-1/2}(\Gamma) \ : \ \dv_{\Gamma} \bm{\mu} \in H^{-1/2}(\Gamma)\right\}, \text{\cite[Theorem~4.1]{buffa2002traces}}\\
		\cdef{\mathbf{H}_{\times}^{1/2}(\Gamma)} &\coloneqq \gammatau(\mathbf{H}^1(\Omega_i)), \text{\cite[Section~2]{buffa2002traces}},\\
		\cdef{\mathbf{H}_{\parallel}^{1/2}(\Gamma)} &\coloneqq \gammat(\mathbf{H}^1(\Omega_i)), \text{\cite[Section~2]{buffa2002traces}},
	\end{align*}
	and $\cdef{\mathbf{H}^{-1/2}_{\times}(\Gamma) }\coloneqq \left[\mathbf{H}^{1/2}_{\times}(\Gamma)\right]', \ \cdef{\mathbf{H}^{-1/2}_{\parallel}(\Gamma) }\coloneqq \left[\mathbf{H}^{1/2}_{\parallel}(\Gamma)\right]'$.\\
	Differential operators on surfaces of Lipschitz domains are defined according to \cite[Section~4]{buffa2002traces}.
	We also need the isomorphism $ \mathsf{R} : \mathbf{H}^{-1/2}(\curl_{\Gamma}, \Gamma) \rightarrow  \mathbf{H}^{-1/2}(\dv_{\Gamma}, \Gamma)$ given by $ \cdef{\mathsf{R}\bm{\mu} }\coloneqq \bm{n}\times \bm{\mu}. $ In particular \cite[Section~2]{buffa2002traces}, for $ \bm{u}\in \mathbf{H}(\Curl, \Omega_i)$ we have 
	\begin{equation}
		\mathsf{R}(\gammat \bm{u}) = - \gammatau \bm{u}, \quad \mathsf{R}(\gammatau \bm{u}) = \gammat \bm{u}.
	\end{equation}For $ u\in H(\Delta, \Omega_i), $ where
	\begin{equation*}
		\cdef{H(\Delta, \Omega_i)} \coloneqq \left\{ u\in H^1(\Omega_i)  \ : \ \Delta u\in L^2(\Omega_i)  \right\},
	\end{equation*}
	we define the  Neumann trace operator $ \partial_n $ as \cite[Theorem~2.8.3]{sauter2010boundary}
	\begin{align*}
		&\partial_n^- : H(\Delta, \Omega_i) \rightarrow H^{-1/2}(\Gamma),\\ 
		&\cdef{\partial_n^-u} = \grad u|_{\Gamma} \cdot \bm{n}, \quad \text{ for }u\in C^{\infty}(\overbar{\Omega}_i).
	\end{align*}
	Replacing $ \Omega_i  $ by $ \Omega_o \coloneqq \reals^d \setminus \overbar{\Omega}_i $ in the previous definitions, we obtain exterior trace operators: $ \gamma^+, \gamman^+, \gammat^+, \gammatau^+ $ and $ \partial_n^+, $ keeping the normal vector $\bm{n}$.\\
	We define jump and average trace operators for elements of $ H^1(\reals^d\setminus\Gamma),\ \mathbf{H}(\dv, \reals^d\setminus\Gamma) $ and $ H(\Delta, \reals^d\setminus\Gamma) $:
	\begin{align*}
		\cdef{	\jump{\gamma} }\coloneqq \{\gamma^+ - \gamma^-\}, \qquad\cdef{ \av{\gamma}} \coloneqq \dfrac{1}{2}\{\gamma^+ + \gamma^-\},
	\end{align*}
	and similarly for other trace operators.
	We denote the bilinear inner product in $ L^2(\Omega_i) $ as $ \cdef{\langle u, v\rangle_{\Omega_i}}. $
	It can be extended to a duality pairing between $ \widetilde{H}^{-1}(\Omega_i) $ and $ H^1(\Omega_i). $
	Similarly, we define the bilinear dual product for $ H^{1/2}(\Gamma) $ and its dual $ H^{-1/2}(\Gamma) $, and denote it as $ \langle \cdot, \cdot \rangle_{\Gamma}. $ We denote $ \cdef{\langle \bm{\mu}, \bm{\zeta}\rangle_{\tau, \Gamma} }$ the bilinear duality pairing between $ \mathbf{H}^{-1/2}(\curl_{\Gamma}, \Gamma) $ and $ \mathbf{H}^{-1/2}(\dv_{\Gamma}, \Gamma). $

	\subsection{Fundamental Solutions and Newton Potential}\label{sec:Newton}
	The fundamental solution for the Helmholtz operator with wavenumber $ \kappa\in \mathbb{R} $ is given by  $ G_{\kappa}\in L_{\text{loc}}^1(\reals^3)$ \cite[Section~5.4]{steinbach2007numerical}:
	\begin{equation}\label{eq:fundamental_solution}
		\cdef{	G_{\kappa}(\nex, \ney)} \coloneqq 
		\dfrac{\exp(i\kappa|\nex - \ney| )}{4\pi |\nex - \ney|}, \quad \nex,\ney\in \reals^3, \nex\neq\ney.
	\end{equation}
	The Newton potential $ \mathsf{N}_{\kappa} : C^{\infty}_c(\reals^3) \rightarrow C^{\infty}(\reals^3) $ is the mapping defined by \cite[Section~3.1.1]{sauter2010boundary}
	\begin{equation}\label{eq:NewtonR3}
		\cdef{\mathsf{N}_{\kappa} f(\nex)} \coloneqq \displaystyle \int\limits_{\reals^3}G_{\kappa}(\nex, \ney)f(\ney)\text{d}\ney.
	\end{equation}
	The Newton potential can be extended to the following two continuous operators
	\begin{equation}\label{eq:mapping_newton}
		\begin{array}{lll}
			\mathsf{N}_{\kappa} &:& H^{-1}_{\text{comp}}(\reals^3)\rightarrow H^1_{\text{loc}}(\reals^3),\\
			\mathsf{N}_{\kappa} &:& L^2_{\text{comp}}(\reals^3)\rightarrow H^2_{\text{loc}}(\reals^3)
		\end{array}
	\end{equation}
	and more generally, 
	$\mathsf{N}_{\kappa}: H^{s}_{\text{comp}}(\reals^3)\rightarrow H^{s+2}_{\text{loc}}(\reals^3)$
	is continuous for $ s \in \reals $ \cite[Theorem.~3.12]{sauter2010boundary}.\\
	Similarly, by extension by zero followed by restriction to $ \Omega_i , $ it is possible to consider the Newton potential in a bounded domain $ \Omega_i: $
	\begin{equation}\label{eq:mapping_newton_omega}
		\begin{array}{lll}
			\cdef{\mathsf{N}_{\Omega_i,\kappa} }&:& L^2(\Omega_i)\rightarrow H^2(\Omega_i),\\
			\mathsf{N}_{\Omega_i,\kappa} &:& \widetilde{H}^{-1}(\Omega_i)\rightarrow H^1(\Omega_i).
		\end{array}
	\end{equation}
	We define the scalar single layer potential as \cite[Theorem~1]{costabel1988boundary}
	\begin{align}
		\cdef{\mathsf{S}_{\kappa} }\coloneqq \mathsf{N}_{\kappa} \circ \gamma ' &\ : \ H^{-1/2}(\Gamma) \rightarrow H^1_{\text{loc}}(\reals^3 \setminus \Gamma),
	\end{align}
	which for smooth enough densities $\psi \in L^{\infty}(\Gamma)$ has the following integral representations for $\nex\notin\Gamma $
	\begin{align}
		(\mathsf{S}_{\kappa}\psi)(\nex) &= \int\limits_{\Gamma} G_{\kappa}(\nex,\ney) \psi(\ney)\dsy.
	\end{align}

	The following theorem \cite[Theorem~8.1]{colton2012inverse} is essential for the derivation of volume integral equations for scattering problems.
	
	\begin{theorem}\label{eq:newton}
		The Newton potential defines a solution operator for the Helmholtz equation on $ \reals^3 $, i.e. for $ f \in L^2_{\mathrm{comp}}(\reals^3)$ compactly supported in $ \Omega_i$, $ u \coloneqq \mathsf{N}_{\kappa}f $ satisfies
		\begin{equation}
			-\Delta u - \kappa^2 u = f \quad \text{ in }\reals^3
		\end{equation}
		and the Sommerfeld radiation conditions.
	\end{theorem}
	Both the Newton potential and the single layer potential will also be used with vectorial arguments, for which the following mapping properties hold.
	\begin{proposition}\label{eq:vectorNewton-SL}The Newton potential can be extended to vectorial arguments component-wise. We denote it as $\mathbf{N}_{\Omega_i,\kappa}$, and it defines a continuous linear operator
		\begin{equation}\label{eq:mappingVecNewton}
			\mathbf{N}_{\Omega_i,\kappa} : \mathbf{L}^2(\Omega_i)\rightarrow \mathbf{H}^2(\Omega_i),
		\end{equation}
		and it has the integral representation 
		\begin{equation*}
			\cdef{\mathbf{N}_{\Omega_i,\kappa}(\bm{f})}\coloneqq \displaystyle\int\limits_{\Omega_i} G_{\kappa}(\nex-\ney)\bm{f}(\ney)\dy, 
		\end{equation*}	
		for all $\bm{f}\in \mathbf{L}^2(\Omega_i).$\\
		The single layer potential can also be extended to vectorial arguments component-wise. We denote it $\cdef{\mathbf{S}_{\kappa}}$, and it defines a continuous linear operator
		\begin{equation*}
			\mathbf{S}_{\kappa} :\mathbf{H}^{-1/2+s}(\Gamma) \rightarrow \mathbf{H}^{1+s}_{\text{loc}}(\reals^3), \quad -\tfrac{1}{2} < s < \tfrac{1}{2}.
		\end{equation*}
	\end{proposition}
	\begin{proof}
		We know that the scalar Newton potential satisfies
		\begin{equation}\label{eq:scalarNewton}
			\mathsf{N}_{\Omega_i,\kappa} : L^2(\Omega_i)\rightarrow H^2(\Omega_i).
		\end{equation}
		In particular, for $\neu\in \mathbf{L}^2(\Omega_i),$ using \eqref{eq:scalarNewton} component-wise leads to
		\begin{equation}
			\mathbf{N}_{\Omega_i,\kappa} : \mathbf{L}^2(\Omega_i)\rightarrow \mathbf{H}^2(\Omega_i).
		\end{equation}
		In a similar way, we know (see \cite[Theorem~3.1.16]{sauter2010boundary}) that the scalar single-layer potential satisfies
		\begin{equation}\label{eq:scalarSL}
			\mathsf{S}_{\kappa} : H^{-1/2+s}(\Gamma)\rightarrow H^{1+s}_{\text{loc}}(\reals^3), \quad -\tfrac{1}{2}<s<\tfrac{1}{2}.
		\end{equation}
		For any $\bm{\psi}\in \mathbf{H}^{-1/2+s}, s\in (-\tfrac{1}{2}, \tfrac{1}{2}]$, using \eqref{eq:scalarSL} component-wise, we obtain
		\begin{equation}
			\mathbf{S}_{\kappa} : \mathbf{H}^{-1/2+s}\rightarrow \mathbf{H}^{1+s}_{\text{loc}}(\reals^3), \quad -\tfrac{1}{2}<s< \tfrac{1}{2}.
		\end{equation}
	\end{proof}
	\begin{corollary}
		The Newton potential defines a continuous linear operator
		\begin{equation*}
			\mathbf{N}_{\Omega_i,\kappa} \ : \ \mathbf{H}(\Curl, \Omega_i)\rightarrow \mathbf{H}^2(\Omega_i).
		\end{equation*}
	\end{corollary}
	\begin{corollary}
		The vector-valued single-layer potential $\mathbf{S}_{\kappa}$ defines a continuous linear operator
		\begin{equation*}
			\mathbf{S}_{\kappa} \ : \ \mathbf{H}^{-1/2}(\dv_{\Gamma}, \Gamma) \rightarrow \mathbf{H}(\Curl, \Omega_i).
		\end{equation*}
	\end{corollary}
	\begin{proof}
		From Proposition \ref{eq:vectorNewton-SL}, since $\mathbf{H}^1_{\text{loc}}(\reals^3)\subset \mathbf{H}_{\text{loc}}(\Curl, \reals^3)$, and $\mathbf{H}^{-1/2}(\dv_{\Gamma},\Gamma)\subset\mathbf{H}^{-1/2}_{\times}(\Gamma)\subset\mathbf{H}^{-1/2}(\Gamma)=[H^{-1/2}(\Gamma)]^3$, we obtain
		\begin{equation}
			\mathbf{S}_{\kappa} : \mathbf{H}^{-1/2}(\dv_{\Gamma},\Gamma)\rightarrow\mathbf{H}_{\text{loc}}(\Curl, \reals^3).
		\end{equation}
	\end{proof}
	\subsection{Stratton-Chu Representation Formula}\label{sec:Representation}
	
	We show an integral representation for arbitrary vector fields, which will be useful for the study of Maxwell solutions \cite[Theorem~9.1]{monk2003finite}
	\begin{tcolorbox}[colback=lightgray!15!white, colframe=lightgray!15!white, sharp corners]
	\begin{theorem}[Stratton-Chu Integral Representation]\label{stratton-chu-general}
		Let $ \neu, \nev \in C^2(\Omega_i),\  \varepsilon_{\star}, \mu_{\star} > 0,$ and $ \kappa_{\star} = \omega\sqrt{\mu_{\star}\varepsilon_{\star}}. $ Then the following integral representations hold
		\begin{align}\nonumber
			\neu =&\ \Curl(\mathbf{N}_{\Omega_i,\kappa_{\star}}(\Curl \neu - i\omega \mu_{\star} \nev)) + i\omega\mu_{\star} \mathbf{N}_{\Omega_i,\kappa_{\star}}(\Curl \nev + i\omega \varepsilon_{\star} \neu) & \\
			&- \grad (\mathsf{N}_{\Omega_i, \kappa_\star} (\dv\neu))+ \Curl (\mathbf{S}_{\kappa_\star}(\gammatau \neu)) + \grad( \mathsf{S}_{\kappa_\star}(\gamman \neu)) + i\omega \mu_{\star} \mathbf{S}_{\kappa_\star}(\gammatau \nev),&\nonumber \\
			&&\\ \nonumber
			\nev =&\ \Curl (\mathbf{N}_{\Omega_i, \kappa_{\star}}(\Curl \nev + i\omega \varepsilon_{\star} \neu) )- i\omega\varepsilon_{\star} \mathbf{N}_{\Omega_i,\kappa_{\star}}(\Curl \neu - i\omega \mu_{\star} \nev) &\\
			&- \grad (\mathsf{N}_{\Omega_i, \kappa_{\star}} (\dv\nev))+ \Curl (\mathbf{S}_{\kappa_{\star}}(\gammatau \nev) )+ \grad (\mathsf{S}_{\kappa_{\star}}(\gamman \nev)) - i\omega \varepsilon_{\star} \mathbf{S}_{\kappa_{\star}}(\gammatau \neu).&\nonumber
		\end{align} 
	\end{theorem}
\end{tcolorbox}
	We introduce the transmission problem with piecewise-constant coefficients $\varepsilon_0,\mu_0 > 0$ in $\Omega_o$, $\ \varepsilon_1,\mu_1 >0 $ in $\Omega_i.$
	\begin{equation}\label{eq:transmission_problem_pwc}
		\left\{ \begin{array}{l}
			\begin{array}{rclrcll}
				\Curl \neu - i\omega \mu_0 \nev &=& 0, & \Curl \nev + i\omega \varepsilon_0 \neu &=& 0 &\text{ in }\Omega_o,\\
				\Curl \neu - i\omega \mu_1 \nev &=& \bm{f}_1, & \Curl\nev + i\omega \varepsilon_1 \neu &=& \bm{f}_2 &\text{ in }\Omega_i ,\\
				\gammat^+ \neu - \gammat^- \neu&=& -\gammat \neu^{\text{inc}}, &	\gammat^+ \nev - \gammat^- \nev&=& -\gammat \nev^{\text{inc}} &\text{ on }\Gamma, \\
			\end{array} \\
			\qquad\qquad\lim\limits_{r\rightarrow\infty}	\nev\times \dfrac{\nex}{r} - \neu = 0, \text{ uniformly on }r=|\nex|
		\end{array}\right.
	\end{equation}
	where $\bm{f}_1$ and $ \bm{f}_2 $ are in $\mathbf{H}(\dv, \Omega_i).$
	For Maxwell solutions, the integral representation takes a different form. 
	If $ \neu, \nev $ are solutions of \eqref{eq:transmission_problem_pwc}, it is possible to express $\neu$ and $\nev$ in terms of $ \gamman\neu, \gammat \neu, \gammatau \nev$. Note that from \eqref{eq:transmission_problem_pwc} we have
	\begin{align*}
		\begin{aligned}
			\Curl \neu - i\omega\mu_1\nev = \bm{f}_1 \text{ in }\Omega\quad&\Rightarrow \quad\nev = \tfrac{1}{i\omega\mu_1}\left(\Curl\neu - \bm{f}_1\right) &&\text{ in }\Omega_i,\\
			\Curl \nev + i\omega\varepsilon_1\neu = \bm{f}_2 \text{ in }\Omega\quad&\Rightarrow \quad\neu = \tfrac{1}{i\omega\varepsilon_1}\left(-\Curl\nev + \bm{f}_2\right) &&\text{ in }\Omega_i,
		\end{aligned}
	\end{align*}
	and therefore, using the property \cite[Section~4]{buffa2002traces}
	\begin{equation}
		\gamman^{\pm}(\Curl \bm{F}) = \dv_{\Gamma}(\gammatau^{\pm} \bm{F}), \quad \text{ for all }\bm{F}\in \mathbf{H}(\Curl, \reals^3\setminus \Gamma),
	\end{equation}
	we obtain the following identities in $H^{-1/2}(\Gamma)$
	\begin{subequations}
		\begin{align}\label{eq:gammanU}
			\gamman^- \neu &= \bm{n}\cdot \neu |_{\Gamma} = \dfrac{1}{i\omega \varepsilon_1}\left(-\bm{n}\cdot \Curl \nev|_{\Gamma} + \bm{n}\cdot \bm{f}_2|_{\Gamma}\right) = \dfrac{1}{i\omega\varepsilon_1}\left( - \dv_{\Gamma} (\gammatau^- \nev) + \gamman^- \bm{f}_2 \right), \\
			&\nonumber\\ \label{eq:gammanV}
			\gamman^- \nev &= \bm{n}\cdot \nev |_{\Gamma} = \dfrac{1}{i\omega \mu_1}\left(\bm{n}\cdot \Curl \neu|_{\Gamma}- \bm{n}\cdot \bm{f}_1|_{\Gamma}\right) = -\dfrac{1}{i\omega\mu_1} \left(\curl_{\Gamma} (\gammat^- \neu) - \gamman^-\bm{f}_1\right).
		\end{align}
	\end{subequations}
	From Theorem \ref{stratton-chu-general}, in $ \Omega_i $ we have ($\kappa_1\coloneqq\omega\sqrt{\varepsilon_1\mu_1}$) for the solution $(\neu,\nev)$ of \eqref{eq:transmission_problem_pwc}
	\begin{subequations}\label{eq:representation_formula_interior}
		\begin{align}
			&\begin{aligned}
				\neu =&\ \Curl (\mathbf{N}_{\Omega_i, \kappa_1}(\bm{f}_1)) + i\omega\mu_1 \mathbf{N}_{\Omega_i, \kappa_1}(\bm{f}_2) - \grad (\mathsf{N}_{\Omega_i, \kappa_1} (\dv\neu))&\\
				&-\Curl (\mathbf{S}_{\kappa_1}(\mathsf{R}\gammat^- \neu)) +  i\omega\mu_1 \left(\dfrac{1}{\kappa_1^2}\grad (\mathsf{S}_{\kappa_1}(\dv_{\Gamma}(\gammatau^- \nev))) +  \mathbf{S}_{\kappa_1}(\gammatau^- \nev)\right),&\\
				&+ \dfrac{1}{i\omega\varepsilon_1}\grad(\mathsf{S}_{\kappa_1}(\gamman^-\bm{f}_2)),
			\end{aligned}\\
			&\nonumber\\
			&	\begin{aligned}
				\nev =&\ \Curl (\mathbf{N}_{\Omega_i, \kappa_1}(\bm{f}_2)) - i\omega\varepsilon_1 \mathbf{N}_{\Omega_i, \kappa_1}(\bm{f}_1) - \grad (\mathsf{N}_{\Omega_i,\kappa_1} (\dv\nev))&\\
				&+\Curl (\mathbf{S}_{\kappa_1}(\gammatau^- \nev)) + i\omega\varepsilon_1 \left( \dfrac{1}{\kappa_1^2}\grad (\mathsf{S}_{\kappa_1}(\curl_{\Gamma}(\gammat^- \neu))) +  \mathbf{S}_{\kappa_1}(\mathsf{R}\gammat^- \neu)\right)&\\
				&+\dfrac{1}{i\omega\mu_1}\grad(\mathsf{S}_{\kappa_1}(\gamman^-\bm{f}_1)).
			\end{aligned}
		\end{align} 
	\end{subequations}
	Owing to the vanishing source terms, in $ \Omega_o $ we find the representation
	\begin{subequations}\label{eq:representation-formula-exterior}
		\begin{align}\label{eq:U0}
			\neu =&\ \Curl (\mathbf{S}_{\kappa_0}(\mathsf{R}\gammat^+ \neu)) -i\omega\mu_0\left( \dfrac{1}{\kappa_0^2}\grad (\mathsf{S}_{\kappa_0}(\dv_{\Gamma}(\gammatau^+ \nev)))  +\mathbf{S}_{\kappa_0}(\gammatau^+ \nev)\right),& \\
			&&\nonumber \\ \label{eq:V0}
			\nev =& - \Curl (\mathbf{S}_{\kappa_0}(\gammatau^+ \nev)) - i\omega\varepsilon_0\left(\dfrac{1}{\kappa_0^2}\grad (\mathsf{S}_{\kappa_0}(\curl_{\Gamma}(\gammat^+ \neu))) + \mathbf{S}_{\kappa_0}(\mathsf{R}\gammat^+ \neu)\right).& 
		\end{align}
	\end{subequations}
		We define the Maxwell single layer potential as
	\begin{align}\label{eq:MaxwellSL}
		\cdef{	\bm{\mathcal{T}}_{\kappa}\bm{\beta} }\coloneqq \tfrac{1}{\kappa^2}\grad\circ\ \mathsf{S}_{\kappa}\circ \dv_{\Gamma} + \mathbf{S}_{\kappa},\quad \text{ for all }\bm{\beta}\in \mathbf{H}^{-1/2}(\dv_{\Gamma}, \Gamma),
	\end{align}
	and the Maxwell double layer potential as
	\begin{equation}\label{eq:MaxwellDL}
		\cdef{	\bm{\mathcal{D}}_{\kappa} \bm{\alpha}} \coloneqq \Curl \mathbf{S}_{\kappa} (\mathsf{R} \bm{\alpha}), \quad \text{ for all }\bm{\alpha}\in \mathbf{H}^{-1/2}(\curl_{\Gamma}, \Gamma).
	\end{equation}
	
	\begin{proposition}[Maxwell Layer Potentials {\cite[Theorem~5]{buffa2003galerkin}}]\label{SLandDL}
		The Maxwell single layer potential and the Maxwell double layer potential are continuous linear operators
		\begin{subequations}
		\begin{align}
			\bm{\mathcal{T}}_{\kappa} \ &:\ \mathbf{H}^{-1/2}(\dv_{\Gamma}, \Gamma) \rightarrow \mathbf{H}(\Curl^2, \Omega_i \cup \Omega_o)\cap \mathbf{H}(\dv0, \Omega_i\cup\Omega_o),\\
			\bm{\mathcal{D}}_{\kappa} \ &:\ \mathbf{H}^{-1/2}(\curl_{\Gamma}, \Gamma) \rightarrow \mathbf{H}(\Curl^2, \Omega_i \cup \Omega_o)\cap \mathbf{H}(\dv0, \Omega_i\cup \Omega_o),
		\end{align}
	\end{subequations}
		where 
		\begin{align*}
			\cdef{\mathbf{H}(\Curl^2, \Omega_i\cup\Omega_o)}&\coloneqq \{\neu\in \mathbf{H}(\Curl, \Omega_i\cup\Omega_o) \ : \ \Curl^2\neu \in \mathbf{L}^2(\Omega_i\cup\Omega_o)\}, \\
			\cdef{\mathbf{H}(\dv0, \Omega_i\cup\Omega_o)}&\coloneqq \{\neu\in\mathbf{L}^2(\Omega_i\cup\Omega_o)\ : \ \dv\neu = 0\}.
		\end{align*}
	\end{proposition}
	Maxwell layer potentials define solutions for the Maxwell's equations complying with the Silver-M\"uller radiation conditions. Note that
	\begin{subequations}
		\begin{align}\label{eq:curlSL}
			&\Curl \bm{\mathcal{T}}_{\kappa} = \Curl \mathbf{S}_{\kappa} =  -\bm{\mathcal{D}}_{\kappa}\mathsf{R}, \\ \label{eq:curlDL}
			&\Curl \bm{\mathcal{D}}_{\kappa} = \Curl ^ 2 (\mathbf{S}_{\kappa}\mathsf{R}) = (\grad\dv + \kappa^2) \mathbf{S}_{\kappa}\mathsf{R}.
		\end{align}
	\end{subequations}
	The following identity is useful in our computations.
	\begin{lemma}[{\cite[Lemma~5]{buffa2003galerkin}}]\label{eq:divSL}
		For $\bm{\beta}\in \mathbf{H}^{-1/2}(\dv_{\Gamma}, \Gamma)$ we have $\dv(\mathbf{S}_{\kappa}(\bm{\beta}) )= \mathsf{S}_{\kappa}(\dv_{\Gamma}\bm{\beta})$ in $\mathbf{L}^2(\reals^3).$
	\end{lemma}
	From Lemma \ref{eq:divSL} and \eqref{eq:curlDL} we obtain
	\begin{equation}
		\Curl\bm{\mathcal{D}}_{\kappa} = (\grad \circ \mathsf{S}_{\kappa} \circ \dv_{\Gamma} + \kappa^2\mathbf{S}_{\kappa})\mathsf{R} = \kappa^2\bm{\mathcal{T}}_{\kappa}\mathsf{R}.
	\end{equation}
	The Maxwell layer potentials also satisfy jump relations across a boundary $\Gamma.$ This will be useful when deriving boundary integral equations from integral representations.
	\begin{proposition}[Jump relations {\cite[Theorem~7]{buffa2003galerkin}}]
		Tangential traces of Maxwell layer potentials are well defined and satisfy
		\begin{align}\label{eq:jump-relations}
			\jump{\gammat}\bm{\mathcal{T}}_{\kappa} = 0, \quad 
			\jump{\gammat}\bm{\mathcal{D}}_{\kappa} = -\mathbf{I},\quad \text{ in }\mathbf{H}^{-1/2}(\curl_{\Gamma}, \Gamma),
		\end{align}
		where $\mathbf{I}$ is the identity operator in $\mathbf{H}^{-1/2}(\curl_{\Gamma}, \Gamma).$
	\end{proposition}
	\subsection{Boundary integral operators}\label{sec:BIOs}
	Boundary integral operators can be defined by averaging traces of Maxwell layer potentials \eqref{eq:MaxwellSL} and \eqref{eq:MaxwellDL}.\\
	First, we define the Maxwell single-layer boundary integral operators (or electric field integral operators) \cite[Section~5]{buffa2003galerkin},
	\begin{align}
		\cdef{\mathbf{V}_{\kappa}} 
		&: \mathbf{H}^{-1/2}(\dv_{\Gamma}, \Gamma) \rightarrow \mathbf{H}^{-1/2}(\curl_{\Gamma}, \Gamma),\\ \nonumber
		&\coloneqq \av{\gammat} \bm{\mathcal{T}}_{\kappa}  = \dfrac{1}{\kappa^2}  \grad_{\Gamma} \circ \mathsf{V}_{\kappa} \circ \dv_{\Gamma} + \mathbf{V}^{\mathbf{t}}_{\kappa} \\ 
		&\nonumber\\
		\cdef{	\bm{W}_{\kappa}} 
		&: \mathbf{H}^{-1/2}(\curl_{\Gamma}, \Gamma) \rightarrow \mathbf{H}^{-1/2}(\dv_{\Gamma}, \Gamma),\\ \nonumber
		&\coloneqq -\av{\gammatau} \Curl(\bm{\mathcal{D}}_{\kappa}) = -\av{\gammatau}\kappa^2\bm{\mathcal{T}}_{\kappa}\mathsf{R}= \Curl_{\Gamma} \circ \mathsf{V}_{\kappa} \circ \curl_{\Gamma} + \kappa^2 \mathbf{V}^{\bm{\tau}}_{\kappa}
	\end{align}
	where 
	\begin{equation}
		\mathsf{V}_{\kappa} \coloneqq \av{\gamma} \mathsf{S}_{\kappa},\quad  \mathbf{V}^{\mathbf{t}}_{\kappa} \coloneqq \av{\gammat} \mathbf{S}_{\kappa}, \quad \mathbf{V}^{\bm{\tau}}_{\kappa} \coloneqq \av{\gammatau} (\mathbf{S}_{\kappa}\mathsf{R}).
	\end{equation}
	\begin{proposition}[Ellipticity of single layer operator {\cite[Lemma~8]{buffa2003galerkin}}]
		The operators $\mathsf{V}_0,\ \mathbf{V}_0^{\mathbf{t}}$ and $\mathbf{V}_0^{\bm{\tau}}$ are continuous and satisfy
		\begin{subequations}
			\begin{align}
				\langle \mathsf{V}_0 \psi, \psi \rangle_{\Gamma} &\geq C \norm{\psi}_{H^{-1/2}(\Gamma)}^2 \quad \text{ for all }\psi\in H^{-1/2}(\Gamma),\\
				\langle \mathbf{V}^{\mathbf{t}}_0\mathbf{\bm{\beta}}, \bm{\beta}\rangle_{\Gamma} &\geq C\norm{\bm{\beta}}^2_{\mathbf{H}^{-1/2}(\Gamma)} \quad \text{ for all }\bm{\beta}\in \mathbf{H}^{-1/2}_{\times}(\Gamma),\\
				\langle \mathbf{V}^{\bm{\tau}}_0\mathbf{\bm{\alpha}}, \bm{\alpha}\rangle_{\Gamma} &\geq C\norm{\bm{\alpha}}^2_{\mathbf{H}^{-1/2}(\Gamma)} \quad \text{ for all }\bm{\alpha}\in \mathbf{H}^{-1/2}_{\parallel}(\Gamma),
			\end{align}
		\end{subequations}
		with constants $C>0$ only depending on $\Gamma$ \footnote{We write $C$ for a positive generic constant. The value of $C$ may be different at different occurrences. }.
	\end{proposition}
	We also define the Maxwell double-layer boundary integral operators (or magnetic field integral operators),
	\begin{align}\begin{aligned}
			\cdef{\mathbf{K}_{\kappa}} &\coloneqq \av{\gammat} \bm{\mathcal{D}}_{\kappa}  &:& \ \mathbf{H}^{-1/2}(\curl_{\Gamma}, \Gamma) \rightarrow \mathbf{H}^{-1/2}(\curl_{\Gamma}, \Gamma),&
			&\\
			\cdef{\mathbf{K}'_{\kappa}} &\coloneqq \av{\gammatau} (\bm{\mathcal{D}}_{\kappa}\mathsf{R})  &:& \ \mathbf{H}^{-1/2}(\dv_{\Gamma}, \Gamma) \rightarrow \mathbf{H}^{-1/2}(\dv_{\Gamma}, \Gamma),&
		\end{aligned}
	\end{align}
	We collect all of them in the Calder\'on operator
	\begin{equation}\label{eq:Calderon}
		\cdef{\mathbf{A}_{\kappa}} \coloneqq \begin{pmatrix}
			-\mathbf{K}_{\kappa} & \mathbf{V}_{\kappa} \\
			\mathbf{W}_{\kappa} & \mathbf{K}'_{\kappa}
		\end{pmatrix} : \bm{\mathcal{H}}(\Gamma) 
		\rightarrow  \bm{\mathcal{H}}(\Gamma),
	\end{equation}
	where we denote
	\begin{align}\label{eq:HGamma}
		\cdef{\bm{\mathcal{H}}(\Gamma) }\coloneqq \mathbf{H}^{-1/2}(\curl_{\Gamma}, \Gamma)\times \mathbf{H}^{-1/2}(\dv_{\Gamma}, \Gamma).
	\end{align}

	\subsection{Calder\'on Identities}
	From the jump relations \eqref{eq:jump-relations} and definitions of BIOs, traces of layer potentials can be written in the form of Calder\'on identities. We start from the representation formula in \eqref{eq:representation_formula_interior} for the transmission problem \eqref{eq:transmission_problem_pwc}, that is the case of \emph{piecewise-constant coefficients}, and assume that there are no sources $\bm{f}_1$ and $\bm{f}_2$ in $\Omega_i.$ Then, it follows from the definitions and jump relations
	
	\begin{subequations}\label{eq:calderonMax-interior}
		\begin{align}
			\gammat^- \neu &= \left(\tfrac{1}{2}\mathbf{I} - \mathbf{K}_{\kappa_1}\right)(\gammat^- \neu) +i\omega\mu_{1} \mathbf{V}_{\kappa_1}(\gammatau^- \nev),\\
			\gammatau^- \nev &= \left(\tfrac{1}{2}\mathbf{I} + \mathbf{K}'_{\kappa_1}\right)(\gammatau^- \nev) + \dfrac{1}{i\omega\mu_1}\mathbf{W}_{\kappa_1}(\gammat^- \neu).
		\end{align}
	\end{subequations}
	Similarly, for \eqref{eq:U0} and \eqref{eq:V0} we have
	\begin{subequations}\label{eq:calderonMax-exterior}
		\begin{align}
			\gammat^+ \neu &= \left(\tfrac{1}{2}\mathbf{I} + \mathbf{K}_{\kappa_0}\right)(\gammat^+ \neu) - i\omega\mu_0\mathbf{V}_{\kappa_0}(\gammatau^+ \nev),\\
			\gammatau^+ \nev &= \left(\tfrac{1}{2}\mathbf{I} - \mathbf{K}'_{\kappa_0}\right)(\gammatau^+ \nev) - \dfrac{1}{i\omega\mu_0}\mathbf{W}_{\kappa_0}(\gammat^+ \neu).
		\end{align} 
	\end{subequations}
	We denote
	\begin{equation}
		\cdef{\tilde{\nev}}\coloneqq \left\{ \begin{aligned}
			i\omega\mu_0\nev, \quad \text{ for }\nex\in \Omega_o,\\
			i\omega\mu_1 \nev, \quad \text{ for }\nex\in\Omega_i,
		\end{aligned}\right.
	\end{equation}
	and write \eqref{eq:calderonMax-interior} and \eqref{eq:calderonMax-exterior} as
	\begin{subequations}\label{eq:calderonMax}
		\begin{align}\label{eq:calderonMax1}
			&\left(\tfrac{1}{2}\mathbf{I} - \mathbf{A}_{\kappa_1}\right)\begin{pmatrix}
				\gammat^-\neu \\ \gammatau^-\tilde{\nev}
			\end{pmatrix} =0, \\  \nonumber&\\ \label{eq:calderonMax0}
			&\left(\tfrac{1}{2}\mathbf{I} + \mathbf{A}_{\kappa_0}\right)\begin{pmatrix}
				\gammat^+\neu \\ \gammatau^+\tilde{\nev}
			\end{pmatrix} =0.
		\end{align}
	\end{subequations}
	\subsection{Boundary Integral Formulations for Transmission Problems}
	The focus is still on the case of piecewise-constant coefficients. From the Calder\'on identities \eqref{eq:calderonMax} it is possible to obtain a formulation for transmission problems. So far, we know that $\neu$ and $\nev$ are Maxwell solutions, and we have written expressions for their interior and exterior traces. It remains to impose transmission conditions
	\begin{equation}\label{eq:maxTrans}
		\begin{pmatrix}
			\gammat^+ \neu \\ \gammatau^+\tilde{\nev }
		\end{pmatrix} - \begin{pmatrix}
			1 & 0 \\ 0 & \tfrac{\mu_0}{\mu_1}
		\end{pmatrix} \begin{pmatrix}
			\gammat^- \neu \\ \gammatau^-\tilde{\nev} 
		\end{pmatrix} = \begin{pmatrix}
			\gammat \neu^{\text{inc}} \\ \gammatau\tilde{\nev}^{\text{inc}}
		\end{pmatrix} \quad \text{ on }\Gamma.
	\end{equation}
	We denote 
	\begin{equation}
		\cdef{\mathbf{M}} \coloneqq \begin{pmatrix}
			1 & 0 \\ 0 & \tfrac{\mu_0}{\mu_1}
		\end{pmatrix}
	\end{equation}
	and combine \eqref{eq:maxTrans} with \eqref{eq:calderonMax0} to obtain
	\begin{equation}\label{eq:preMaxSTF}
		\left(\tfrac{1}{2}\mathbf{I} + \mathbf{A}_{\kappa_0} \right)\mathbf{M}\begin{pmatrix}
			\gammat^- \neu \\ \gammatau^-\tilde{\nev} 
		\end{pmatrix} = - \left(\tfrac{1}{2}\mathbf{I} + \mathbf{A}_{\kappa_0} \right)\begin{pmatrix}
			\gammat \neu^{\text{inc}} \\ \gammatau\tilde{\nev}^{\text{inc}} 
		\end{pmatrix} = - \begin{pmatrix}
			\gammat \neu^{\text{inc}} \\ \gammatau\tilde{\nev}^{\text{inc}}
		\end{pmatrix},
	\end{equation}
	where we used that $\neu^{\text{inc}}$ and $\nev^{\text{inc}}$ are interior Maxwell solutions with wavenumber $\kappa_0.$ From \eqref{eq:preMaxSTF} we get
	\begin{equation}\label{eq:preMaxSTF0}
		\left(\tfrac{1}{2}\mathbf{I} + \mathbf{M}^{-1}\mathbf{A}_{\kappa_0}\mathbf{M}\right)\begin{pmatrix}
			\gammat^- \neu \\ \gammatau^-\tilde{\nev }
		\end{pmatrix} = -\mathbf{M}^{-1}\begin{pmatrix}
			\gammat \neu^{\text{inc}} \\ \gammatau\tilde{\nev}^{\text{inc}}
		\end{pmatrix}.
	\end{equation}
	Now, subtracting \eqref{eq:calderonMax1} from \eqref{eq:preMaxSTF0} we obtain the first-kind single-trace formulation \cite[Section~7]{buffa2003galerkin}.
	\begin{tcolorbox}[colback=lightgray!15!white, sharp corners, colframe=lightgray!15!white]
		Find $\bm{\alpha}\in \bm{H}^{-1/2}(\curl_{\Gamma}, \Gamma)$ and $\bm{\beta}\in \bm{H}^{-1/2}(\dv_{\Gamma}, \Gamma)$ such that
		\begin{equation}\label{eq:MaxSTF1}
			\left(\mathbf{M}^{-1}\mathbf{A}_{\kappa_0}\mathbf{M} + \mathbf{A}_{\kappa_1}\right)\begin{pmatrix}
				\bm{\alpha} \\ \bm{\beta}
			\end{pmatrix} = - \mathbf{M}^{-1}\begin{pmatrix}
				\gammat \neu^{\text{inc}} \\ \gammatau\tilde{\nev}^{\text{inc}}
			\end{pmatrix}
		\end{equation}
		in $\bm{H}^{-1/2}(\curl_{\Gamma}, \Gamma) \times \bm{H}^{-1/2}(\dv_{\Gamma}, \Gamma).$
	\end{tcolorbox}
	
	\subsection{Boundary-Volume Integral Representation}\label{sec:integralrep}
	Now we return to the situation where the interior coefficients may not be constant anymore, i.e. may vary in space. 
	We write the transmission problem \eqref{eq:transmission-problem-maxwell} as follows 
	\begin{tcolorbox}[colback=lightgray!15!white, colframe=black, sharp corners, colframe=lightgray!15!white]
		\begin{equation}\label{eq:transmission_problem_pwsmooth}
			\left\{ \begin{array}{l}
				\begin{array}{rclrcll}
					\Curl \neu - i\omega \mu_0 \nev &=& 0, & \Curl \nev + i\omega \varepsilon_0 \neu &=& 0 &\text{ in }\Omega_o,\\
					\Curl \neu - i\omega \mu_1 \nev &=& \bm{f}_1, & \Curl \nev + i\omega \varepsilon_1 \neu &=& \bm{f}_2 &\text{ in }\Omega_i,\\
					\gammat^+ \neu - \gammat^- \neu&=& -\gammat \neu^{\text{inc}}, &	\gammat^+ \nev - \gammat^- \nev&=& -\gammat \nev^{\text{inc}} &\text{ on }\Gamma, \\
				\end{array} \\
				\qquad\qquad\lim\limits_{r\rightarrow\infty}	\nev\times \dfrac{\nex}{r} - \neu = 0, \text{ uniformly on }r=|\nex|,
			\end{array}\right.
		\end{equation}
		where 
		\begin{align*}
			\cdef{\bm{f}_1(\nex)} &\coloneqq -i\omega \mu_1 p_m(\nex)\nev(\nex), &\cdef{\bm{f}_2(\nex)}&\coloneqq i\omega\varepsilon_1 p_e(\nex)\neu(\nex),\\
			\cdef{p_m(\nex)} &\coloneqq 1 - \dfrac{\mu(\nex)}{\mu_1}, &\cdef{p_e(\nex)} &\coloneqq 1 - \dfrac{\varepsilon(\nex)}{\varepsilon_1},
		\end{align*}
		for $\nex\in\Omega_i$, and $ \cdef{\varepsilon_1, \mu_1} \in \mathbb{R}_+$ are conveniently chosen parameters.
	\end{tcolorbox}
	\begin{remark}
		Note that for smooth parameters $\varepsilon$ and $\mu$ (see Assumption \ref{assumption}) and for electric/magnetic fields $\neu, \nev \in \mathbf{H}(\Curl, \Omega_i)\cap\mathbf{H}(\dv, \Omega_i)$, we obtain that $\bm{f}_1|_{\Omega_i}$ and $\bm{f}_2|_{\Omega_i}$ are in $\mathbf{H}(\dv, \Omega_i).$ Therefore, we are in the setting of Section \ref{sec:Representation}.
	\end{remark}
	The representation formula \eqref{eq:representation_formula_interior} now reads:
	In $\Omega_i,$
	\begin{subequations}
		\begin{align}\label{eq:representation_formula_interior2}&\begin{aligned}
				\neu =& -i\omega \mu_1 \Curl (\mathbf{N}_{\Omega_i, \kappa_1}(p_m \nev)) - \kappa_1^2 \mathbf{N}_{\Omega_i, \kappa_1}(p_e\neu) - \grad (\mathsf{N}_{\Omega_i, \kappa_1} (\dv\neu))&\\
				&- \Curl (\mathsf{S}_{\kappa_1}(\mathsf{R}\gammat^- \neu)) - \grad (\mathsf{S}_{\kappa_1}(\gamman^-\neu)) +  i\omega\mu_1 \mathbf{S}_{\kappa_1}\left(\gammatau^- \nev\right),&
			\end{aligned}\\
			\label{eq:representation_formula_interior2v}&\begin{aligned}
				\nev =&\  i\omega \varepsilon_1 \Curl (\mathbf{N}_{\Omega_i, \kappa_1}(p_e \neu)) - \kappa_1^2 \mathbf{N}_{\Omega_i, \kappa_1}(p_m\nev) - \grad (\mathsf{N}_{\Omega_i, \kappa_1} (\dv\nev))&\\
				&+ \Curl \mathbf{S}_{\kappa_1}(\gammatau^- \neu) - \grad (\mathsf{S}_{\kappa_1}(\gamman^-\nev)) +  i\omega\varepsilon_1 \mathbf{S}_{\kappa_1}\left(\mathsf{R}\gammat^- \neu\right).&
			\end{aligned}
		\end{align}
	\end{subequations}
	The operator $ \Curl (\mathbf{N}_{\Omega_i,\kappa_1}(p_{\star}\  \cdot)) \ :\ \mathbf{H}(\Curl, \Omega_i) \rightarrow \mathbf{H}(\Curl, \Omega_i) $ ($\star=\{e, m\}$), is only bounded, not compact. This can be seen by an integration by parts result on Newton potentials. For $\mathbf{F}\in \mathbf{L}^2(\Omega_i),$
	\begin{equation}\label{eq:ibpNewton}
		\Curl (\mathbf{N}_{\Omega_i, \kappa_{\star}}(\bm{F})) = \mathbf{N}_{\Omega_i, \kappa_{\star}}(\Curl \bm{F}) + \mathbf{S}_{\kappa_{\star}}(\gammatau \bm{F}).
	\end{equation}
	It follows that
	\begin{equation}
		\Curl\mathbf{N}_{\Omega_i, \kappa_1}(p\neu) = \mathbf{N}_{\Omega_i, \kappa_1}(\Curl(p\neu)) + \mathbf{S}_{\kappa_1}(p\gammatau \neu),
	\end{equation}
	where the vector single-layer potential $\mathbf{S}_{\kappa_1}$ is only a bounded operator in $\mathbf{H}(\Curl, \Omega_i)$.\\
	We will repeatedly make use of the product rule
	\begin{align*}
		\Curl\left( f\bm{F} \right) = \grad f \times \bm{F} + f\Curl \bm{F}, \quad f\in C^1(\Omega_i), \bm{F}\in [C^1(\Omega_i)]^3.
	\end{align*}
	Solutions of \eqref{eq:transmission_problem_pwsmooth} also satisfy 
	\begin{subequations}
		\begin{align}
			\dv(\varepsilon\neu) &= \grad \varepsilon \cdot \neu + \varepsilon \dv(\neu) = 0 \quad \Rightarrow \quad\dv(\neu) = - \bm{\tau}_e \cdot \neu, &&\text{ in }\Omega_i,\\
			\dv(\mu\nev) &= \grad \mu \cdot \nev+ \mu \dv(\nev) = 0 \quad\Rightarrow\quad \dv(\nev) =  - \bm{\tau}_m \cdot \nev, &&\text{ in }\Omega_i,
		\end{align}
	\end{subequations}
	where we defined
	\begin{equation}\label{eq:taue-taum}
		\cdef{\bm{\tau}_e} \coloneqq \dfrac{\grad\varepsilon}{\varepsilon}, \quad 
		\cdef{\bm{\tau}_m }\coloneqq \dfrac{\grad\mu}{\mu}.
	\end{equation}
	
	\section{STF-VIEs}\label{sec:STF}

		The representation formula from \eqref{eq:representation_formula_interior2} and \eqref{eq:representation_formula_interior2v} now reads
		\begin{align}\label{eq:rep_formula_general}
			\neu =& -i\omega \mu_1 \Curl (\mathbf{N}_{\Omega_i, \kappa_1}(p_m \nev)) - \kappa_1^2 \mathbf{N}_{\Omega_i, \kappa_1}(p_e\neu) + \grad (\mathsf{N}_{\Omega_i, \kappa_1} (\bm{\tau}_e \cdot \neu))&\\
			&- \Curl (\mathbf{S}_{\kappa_1}(\mathsf{R}\gammat^- \neu)) - \grad (\mathsf{S}_{\kappa_1}(\gamman^-\neu) )+  i\omega\mu_1 \mathbf{S}_{\kappa_1}\left(\gammatau^- \nev\right),&\nonumber\\
			&&\nonumber\\ \label{eq:rep_formula_generalV}
			\nev =&\  i\omega \varepsilon_1 \Curl (\mathbf{N}_{\Omega_i, \kappa_1}(p_e \neu)) - \kappa_1^2 \mathbf{N}_{\Omega_i, \kappa_1}(p_m\nev) + \grad (\mathsf{N}_{\Omega_i, \kappa_1} (\bm{\tau}_m\cdot \nev))&\\
			&+ \Curl (\mathbf{S}_{\kappa_1}(\gammatau^- \neu)) - \grad (\mathsf{S}_{\kappa_1}(\gamman^-\nev)) +  i\omega\varepsilon_1 \mathbf{S}_{\kappa_1}\left(\mathsf{R}\gammat^- \neu\right).&\nonumber
		\end{align}
		The integration by parts result from \eqref{eq:ibpNewton} leads to
		\begin{subequations}
			\begin{align}\label{eq:ibpNpm}
				\Curl (\mathbf{N}_{\Omega_i, \kappa_1}(p_m \nev)) &= \mathbf{N}_{\Omega_i, \kappa_1}(\Curl (p_m \nev)) + \mathbf{S}_{\kappa_1}(p_m \gammatau^- \nev),\\ \label{eq:ibpNpe}
				\Curl (\mathbf{N}_{\Omega_i, \kappa_1}(p_e \neu)) &= \mathbf{N}_{\Omega_i, \kappa_1}(\Curl (p_e \neu)) - \mathbf{S}_{\kappa_1}(p_e \mathsf{R}\gammat^- \neu).
			\end{align}
		\end{subequations}
		From \eqref{eq:gammanU} and \eqref{eq:gammanV} we obtain
		\begin{subequations}
			\begin{align}
				\label{eq:gammanU2}
				\gamman^- \neu &=  -\dfrac{1}{i\omega\varepsilon_1\tilde{\varepsilon}} \ \dv_{\Gamma} (\gammatau^- \nev) \quad \text{ in } H^{-1/2}(\Gamma), \\ \label{eq:gammanV2}
				\gamman^- \nev &= \dfrac{1}{i\omega\mu_1\tilde{\mu}} \ \curl_{\Gamma} (\gammat^- \neu) \quad \text{ in } H^{-1/2}(\Gamma),
			\end{align}
		\end{subequations}
		where 
		\begin{equation}\label{eq:scaled-coeffs}
			\cdef{\tilde{\varepsilon}(\nex)} \coloneqq \dfrac{\varepsilon(\nex)}{\varepsilon_1}, \qquad \cdef{\tilde{\mu}(\nex)} \coloneqq \dfrac{\mu(\nex)}{\mu_1},   
		\end{equation}for all $ \nex\in \Omega_i. $\\
		From now on, we denote $
		\tilde{\nev} \coloneqq i\omega\mu_1 \nev $.
		Combining expressions \eqref{eq:ibpNpm}--\eqref{eq:gammanV2} into \eqref{eq:rep_formula_general} and \eqref{eq:rep_formula_generalV} we obtain \emph{a new representation formula} for the fields $\neu, \nev$ solving \eqref{eq:transmission_problem_pwsmooth}:
		\begin{tcolorbox}[colback=lightgray!15!white, sharp corners, colframe=lightgray!15!white]
			\begin{subequations}\label{eq:new-rep}
				\begin{align}\label{eq:repU}
					\neu =& \ \mathbf{\Xi}^m \tilde{\nev} + \mathbf{\Lambda}^e \neu - \bm{\mathcal{D}}_{\kappa_1}(\gammat^- \neu) + \bm{\mathcal{T}}_{\kappa_1}^{\tilde{\varepsilon}, \tilde{\mu}}(\gammatau^-\tilde{\nev}),&&\text{ in }\Omega_i,\\
					\label{eq:repV}
					\tilde{\nev} =&\ \mathbf{\Xi}^e \neu+ \mathbf{\Lambda}^m \tilde{\nev} - \bm{\mathcal{D}}_{\kappa_1}(\mathsf{R}\gammatau^- \tilde{\nev}) -\kappa_1^2 \bm{\mathcal{T}}_{\kappa_1}^{\tilde{\mu}, \tilde{\varepsilon}}( \mathsf{R}\gammat^- \neu), &&\text{ in }\Omega_i,
				\end{align}
			\end{subequations}
			where we defined the \emph{volume integral operators}
			\begin{subequations}\label{eq:VolAK}
				\begin{align}\label{eq:VolK}
					\cdef{\mathbf{\Xi}^m \nev} &\coloneqq   -\mathbf{N}_{\Omega_i, \kappa_1}(\Curl(p_m \nev)),\\ \label{eq:VolA}
					\cdef{\mathbf{\Xi}^e \neu} &\coloneqq -\kappa_1^2 \mathbf{N}_{\Omega_i, \kappa_1}(\Curl (p_e \neu)),\\
					\cdef{\mathbf{\Lambda}^m \nev} &\coloneqq -\kappa_1^2 \mathbf{N}_{\Omega_i,\kappa_1}(p_m\nev) + \grad \mathsf{N}_{\Omega_i,\kappa_1} (\bm{\tau}_m\cdot \nev),\\
					\cdef{\mathbf{\Lambda}^e \neu} &\coloneqq -\kappa_1^2 \mathbf{N}_{\Omega_i, \kappa_1}(p_e\neu) + \grad \mathsf{N}_{\Omega_i, \kappa_1} (\bm{\tau}_e \cdot \neu),
				\end{align}
			\end{subequations}
			and the \emph{layer potentials}
			\begin{subequations}
				\begin{align}\label{eq:weightedT}
					&\cdef{\bm{\mathcal{T}}_{\kappa_1}^{a, b} (\bm{\beta})} \coloneqq \tfrac{1}{\kappa_1^2}\grad (\mathsf{S}_{\kappa_1}(\tfrac{1}{a} \dv_{\Gamma}(\bm{\beta}))) + \mathbf{S}_{\kappa_1}(b \bm{\beta}), &  \bm{\beta} \in \mathbf{H}^{-1/2}(\dv_{\Gamma}, \Gamma),\\ 
					&\bm{\mathcal{D}}_{\kappa_1}(\bm{\alpha}) \coloneqq \Curl \mathbf{S}_{\kappa_1}(\mathsf{R} \bm{\alpha}), & \bm{\alpha} \in \mathbf{H}^{-1/2}(\curl_{\Gamma}, \Gamma).
				\end{align}
			\end{subequations}
		\end{tcolorbox}
		
		\begin{remark}
			We use $\tilde{\nev}$ as an unknown instead of $\nev$ in order to avoid scalings in our operators. 
		\end{remark}
		We take the trace $ \gammat^- $ on \eqref{eq:repU} and $ \gammatau^- $ on \eqref{eq:repV}. By the jump relations \eqref{eq:jump-relations}, we obtain
		\begin{subequations}
			\begin{align}\label{eq:gammat}
				\gammat^- \neu &= \gammat^- (\mathbf{\Xi}^m \tilde{\nev})  + \gammat^- (\mathbf{\Lambda}^e \neu )+ 
				\left(\tfrac{1}{2}\mathbf{I} - \mathbf{K}_{\kappa_1}\right)(\gammat^- \neu) + \mathbf{V}_{\kappa_1}^{\tilde{\varepsilon}, \tilde{\mu}}(\gammatau^- \tilde{\nev}), \\ \label{eq:gammatau}
				\gammatau^- \tilde{\nev} &= \gammatau^- (\mathbf{\Xi}^e \neu)  + \gammatau^- (\mathbf{\Lambda}^m \tilde{\nev} )+ \mathbf{W}_{\kappa_1}^{\tilde{\mu}, \tilde{\varepsilon}}(\gammat^- \neu) + \left(\tfrac{1}{2}\mathbf{I} + \mathbf{K}'_{\kappa_1}\right)(i\omega\mu_1\gammatau^- \tilde{\nev}),
			\end{align}
		\end{subequations}
		where we define weighted boundary integral operators as
		\begin{subequations}\label{eq:weightedBIOs}
			\begin{align}
				\cdef{\mathbf{V}_{\kappa_1}^{\tilde{\varepsilon}, \tilde{\mu}}(\bm{\beta})} &\coloneqq \tfrac{1}{\kappa_1^2} \grad_{\Gamma} \mathsf{V}_{\kappa_1}(\tfrac{1}{\tilde{\varepsilon}} \dv_{\Gamma}(\bm{\beta})) + \mathbf{V}_{\kappa_1}^{\mathbf{t}}(\tilde{\mu}\bm{\beta}), && \bm{\beta}\in \mathbf{H}^{-1/2}(\dv_{\Gamma}, \Gamma),\\
				\cdef{\mathbf{W}_{\kappa_1}^{\tilde{\mu}, \tilde{\varepsilon}}(\bm{\alpha}) }&\coloneqq  \quad \Curl_{\Gamma} \mathsf{V}_{\kappa_1}(\tfrac{1}{\tilde{\mu}} \curl_{\Gamma}(\bm{\alpha})) + \kappa_1^2\mathbf{V}_{\kappa_1}^{\tau}(\tilde{\varepsilon}\bm{\alpha}), &&\bm{\alpha}\in \mathbf{H}^{-1/2}(\curl_{\Gamma}, \Gamma).
			\end{align}
		\end{subequations}
		We denote $ \cdef{\bm{\alpha}^-} \coloneqq \gammat^- \neu, \ \cdef{\bm{\beta}^-}\coloneqq \gammatau^-\tilde{\nev}$ and rewrite \eqref{eq:gammat} and \eqref{eq:gammatau} as
		\begin{subequations}
			\begin{align}\label{eq:alpha}
				\bm{\alpha}^- &= \gammat^- (\mathbf{\Xi}^m \tilde{\nev})  + \gammat^- (\mathbf{\Lambda}^e \neu )+ 
				\left(\tfrac{1}{2}\mathbf{I} - \mathbf{K}_{\kappa_1}\right)(\bm{\alpha}^-) + \mathbf{V}_{\kappa_1}^{\tilde{\varepsilon}, \tilde{\mu}}( \bm{\beta}^-), \\ \label{eq:beta}
				\bm{\beta}^- &= \gammatau^- (\mathbf{\Xi}^e \neu)  + \gammatau^- (\mathbf{\Lambda}^m \tilde{\nev} )+ \mathbf{W}_{\kappa_1}^{\tilde{\mu}, \tilde{\varepsilon}}( \bm{\alpha}^-) + \left(\tfrac{1}{2}\mathbf{I} + \mathbf{K}'_{\kappa_1}\right)(\bm{\beta}^-).
			\end{align}
		\end{subequations}
		We denote $ \cdef{\bm{\alpha}^+} \coloneqq \gammat^+ \neu, \ \cdef{\bm{\beta}^+}\coloneqq i\omega\mu_0\gammatau^+\nev$. From \eqref{eq:U0},  \eqref{eq:V0} and the jump relations \eqref{eq:jump-relations} we have
		\begin{subequations}
			\begin{align}\label{eq:U00}
				\bm{\alpha}^+ &= \left(\tfrac{1}{2}\mathbf{I} + \mathbf{K}_{\kappa_0}\right)(\bm{\alpha}^+) - \mathbf{V}_{\kappa_0}(\bm{\beta}^+), \\ \label{eq:V00}
				\bm{\beta}^+ &= -\mathbf{W}_{\kappa_0}(\bm{\alpha}^+) + \left(\tfrac{1}{2}\mathbf{I} - \mathbf{K}'_{\kappa_0}\right)(\bm{\beta}^+).
			\end{align}
		\end{subequations}
		From \eqref{eq:alpha}--\eqref{eq:V00} we can write
		\begin{subequations}
			\begin{align}\label{eq:calderon1}
				&\left(\tfrac{1}{2}\mathsf{I} - \mathbf{A}_{\kappa_1}^{\tilde{\varepsilon}, \tilde{\mu}} \right)\begin{pmatrix}
					\bm{\alpha}^- \\ \bm{\beta}^-
				\end{pmatrix} - \mathbf{J}^e\neu - \mathbf{J}^m \tilde{\nev} = 0,\\ \label{eq:calderon0}
				&\left(\tfrac{1}{2}\mathsf{I} + \mathbf{A}_{\kappa_0} \right)\begin{pmatrix}
					\bm{\alpha}^+ \\ \bm{\beta}^+
				\end{pmatrix} = 0,
			\end{align}
		\end{subequations}
		where we defined 
		\begin{subequations}
			\begin{align}\label{eq:weightedCalderon}
				\cdef{\mathbf{A}_{\kappa_1}^{\tilde{\varepsilon}, \tilde{\mu}}} &\coloneqq \begin{pmatrix}
					-\mathbf{K}_{\kappa_1} & \mathbf{V}_{\kappa_1}^{\tilde{\varepsilon}, \tilde{\mu}}\\
					\mathbf{W}_{\kappa_1}^{\tilde{\mu}, \tilde{\varepsilon}} & \mathbf{K}'_{\kappa_1}
				\end{pmatrix} \ : \ \bm{\mathcal{H}}(\Gamma)\rightarrow\bm{\mathcal{H}}(\Gamma), \\ \label{eq:tracesVolE}
				\cdef{\mathbf{J}^e}  &\coloneqq \begin{pmatrix}
					\gammat^- \mathbf{\Lambda}^e \\ \gammatau^-\mathbf{\Xi}^e
				\end{pmatrix} \ : \ \mathbf{H}(\Curl, \Omega_i)\rightarrow \bm{\mathcal{H}}(\Gamma), \\ \label{eq:tracesVolM}
				\cdef{\mathbf{J}^m } &\coloneqq \begin{pmatrix}
					\gammat^- \mathbf{\Xi}^m \\ \gammatau^-\mathbf{\Lambda}^m
				\end{pmatrix}\ : \ \mathbf{H}(\Curl, \Omega_i)\rightarrow \bm{\mathcal{H}}(\Gamma).
			\end{align}
		\end{subequations}
		Combining \eqref{eq:calderon1} and \eqref{eq:calderon0} with the transmission conditions
		\begin{equation}
			\begin{pmatrix}
				1 & 0 \\ 0 & \tfrac{1}{\mu_0}
			\end{pmatrix}\begin{pmatrix}
				\bm{\alpha}^+ \\ \bm{\beta}^+
			\end{pmatrix} - \begin{pmatrix}
				1 & 0 \\ 0 & \tfrac{1}{\mu_1}
			\end{pmatrix}\begin{pmatrix}
				\bm{\alpha}^- \\ \bm{\beta}^-
			\end{pmatrix} = \begin{pmatrix}
				\gammat \neu^{\text{inc}}  \\ \tfrac{1}{\mu_0}\gammatau\tilde{\nev}^{\text{inc}} 
			\end{pmatrix}, \quad \tilde{\nev}^{\text{inc}}\coloneqq i\omega\mu_0\nev^{\text{inc}}.
		\end{equation}
		Retaining $\cdef{\bm{\alpha}}\coloneqq \bm{\alpha}^-$ and $\cdef{\bm{\beta}}\coloneqq\bm{\beta}^-$ as unknown traces on $\Gamma$, we obtain the single-trace boundary-volume integral equation
		\begin{equation}
			\left(\mathbf{M}^{-1}\mathbf{A}_{\kappa_0}\mathbf{M} + \mathbf{A}_{\kappa_1}^{\tilde{\varepsilon}, \tilde{\mu}} \right) \begin{pmatrix}
				\bm{\alpha} \\ \bm{\beta}
			\end{pmatrix}  + \mathbf{J}^e \neu  + \mathbf{J}^m \tilde{\nev} = -\mathbf{M}^{-1}\begin{pmatrix}
				\gammat \neu^{\text{inc}}\\ \gammatau \tilde{\nev}^{\text{inc}} 
			\end{pmatrix} ,\label{eq:STF-first}\\
		\end{equation}
		where $ \cdef{\mathbf{M}} \coloneqq \begin{pmatrix}
			1 & 0 \\ 0 & \tfrac{\mu_0}{\mu_1}
		\end{pmatrix} $.\\
	\il{	\begin{remark}\label{second-kind}
			Note that trying to derive a second-kind integral formulation from \eqref{eq:calderon1} and \eqref{eq:calderon0} would fail in the singularity subtraction that occurs when taking differences of Calder\'on operators. This is due to the modified Calder\'on operator in \eqref{eq:calderon1}, containing the inhomogeneous coefficients on the boundary. For coefficients that are constant on the boundary, the formulation would follow as in the constant-coefficient case.
		\end{remark}
	}	
		We summarize the resulting coupled BIEs-VIEs in
		\begin{tcolorbox}[colback=lightgray!15!white, sharp corners, colframe=lightgray!15!white]
			\begin{problem}[STF-VIE]\label{eq:Problem}
				Find fields $ \neu, \nev \in \mathbf{H}(\Curl, \Omega_i)$ and traces $ (\bm{\alpha}, \bm{\beta}) \in \bm{\mathcal{H}}(\Gamma)$ such that
				\begin{equation}\tag{$\clubsuit$}\label{eq:STF-VIE}
					\begin{aligned}
						&\eqref{eq:STF-first} \\ &\eqref{eq:repU} \\ &\eqref{eq:repV}
					\end{aligned} \ \Rightarrow \ \left( 
					\begin{array}{c:c:c}
						\mathbf{M}^{-1}\mathbf{A}_{\kappa_0}\mathbf{M} + \mathbf{A}_{\kappa_1}^{\tilde{\varepsilon}, \tilde{\mu}} & \mathbf{J}^e & \mathbf{J}^m\\ \hdashline 
						\begin{matrix}
							\ \ \bm{\mathcal{D}}_{\kappa_1} & \ -\bm{\mathcal{T}}_{\kappa_1}^{\tilde{\varepsilon}, \tilde{\mu}} \end{matrix} & \mathsf{I}  -\mathbf{\Lambda}^e  & -\mathbf{\Xi}^m\\ \hdashline
						\begin{matrix}
							\kappa_1^2\bm{\mathcal{T}}_{\kappa_1}^{\tilde{\mu}, \tilde{\varepsilon}}\mathsf{R}&  \bm{\mathcal{D}}_{\kappa_1}\mathsf{R}
						\end{matrix} & - \mathbf{\Xi}^e  & \mathsf{I}  -\mathbf{\Lambda}^m
					\end{array}
					\right) \begin{pmatrix}
						\bm{\alpha} \\ \bm{\beta} \\\hdashline  \neu \\ \hdashline \nev 
					\end{pmatrix} = -\begin{pmatrix}
						\bm{g}_1\\ \bm{g}_2 \\\hdashline 0 \\\hdashline 0
					\end{pmatrix}
				\end{equation}
				holds in $\bm{\mathcal{H}}(\Gamma) \times \mathbf{H}(\Curl, \Omega_i)\times \mathbf{H}(\Curl, \Omega_i),$ where we defined the right-hand-side 
				\begin{align*}
					\begin{aligned}
						\bm{g}_1 \coloneqq \gammat\neu^{\text{inc}}, \quad
						\bm{g}_2 \coloneqq \tfrac{\mu_1}{\mu_0}\gammatau\tilde{\nev}^{\text{inc}}.
					\end{aligned}
				\end{align*}
			\end{problem}
		\end{tcolorbox}

		\section{Analysis of STF-VIEs}\label{sec:analysis}
		\subsection{Variational Formulation}\label{sec:variational}
		We now present a variational formulation for the coupled system \eqref{eq:Problem}. We denote by $\langle \cdot, \cdot\rangle_{\Omega_i}$ the duality pairing between $\mathbf{H}(\Curl, \Omega_i)'$ and $\mathbf{H}(\Curl, \Omega_i)$. Recall that $\langle \cdot, \cdot \rangle_{\bm{\tau},\Gamma}$ denotes the duality pairing between $\mathbf{H}^{-1/2}(\curl_{\Gamma}, \Gamma)$ and $\mathbf{H}^{-1/2}(\dv_{\Gamma}, \Gamma)$. In the trace space $\bm{\mathcal{H}}(\Gamma)$ we define
		\begin{align}\label{eq:maxwell-prod-duality}
			\cdef{\llangle \cdot, \cdot \rrangle }\ : \ \bm{\mathcal{H}}(\Gamma)\times\bm{\mathcal{H}}(\Gamma)\rightarrow \mathbb{C}, \quad \llangle \underline{\bm{\varphi}}, \underline{\bm{\zeta}} \rrangle \coloneqq \langle \bm{\varphi}_1, \bm{\zeta}_0\rangle_{\bm{\tau},\Gamma} + \langle \bm{\zeta}_1, \bm{\varphi}_0\rangle_{\bm{\tau},\Gamma},
		\end{align}
		for all $\underline{\bm{\varphi}} = (\bm{\varphi_0}, \bm{\varphi}_1)\in \bm{\mathcal{H}}(\Gamma), \ \underline{\bm{\zeta}} = (\bm{\zeta}_0, \bm{\zeta}_1)\in\bm{\mathcal{H}}(\Gamma)$.
		
		\begin{tcolorbox}[colback=lightgray!15!white, sharp corners, colframe=lightgray!15!white]
			\begin{problem}[Variational Formulation for STF-VIE]\label{prob:varSTF} Given $\underline{\bm{g}}\in \bm{\mathcal{H}}(\Gamma),$ we seek $(\bm{\alpha}, \bm{\beta})\in \bm{\mathcal{H}}(\Gamma)$ and $(\neu, \nev) \in \cdef{\bm{\mathcal{H}}(\Omega_i)}\coloneqq [\mathbf{H}(\Curl, \Omega_i)]^2$ such that the variational formulation 
				\begin{equation*}
					\begin{array}{lclcc}
						\mathsf{a}((\bm{\alpha}, \bm{\beta}), (\bm{\bm{\zeta}}, \bm{\xi})) &+& \mathsf{b}((\neu, \nev), (\bm{\zeta}, \bm{\xi})) &=& \llangle \underline{\bm{g}}, (\bm{\zeta}, \bm{\xi})\rrangle,\\
						\mathsf{c}((\bm{\alpha}, \bm{\beta}), (\bm{w}, \bm{q})) &+& \mathsf{d}((\neu, \nev), (\bm{w}, \bm{q})) &=& 0,
					\end{array}
				\end{equation*}
				holds for all $(\bm{\bm{\zeta}}, \bm{\xi}) \in \bm{\mathcal{H}}(\Gamma)$ and $ (\bm{w}, \bm{q}) \in \bm{\mathcal{H}}(\Omega_i)',$ where we denote the bilinear forms
				\begin{align}\label{eq:sesqui-linear-a}
					\cdef{\mathsf{a}((\bm{\alpha}, \bm{\beta}), (\bm{\bm{\zeta}}, \bm{\xi}))} &\coloneqq \llangle (\mathbf{M}^{-1}\mathbf{A}_{\kappa_0}\mathbf{M} + \mathbf{A}_{\kappa_1}^{\tilde{\varepsilon}, \tilde{\mu}})(\bm{\alpha}, \bm{\beta}), (\bm{\bm{\zeta}}, \bm{\xi})\rrangle, \\ \label{eq:sesqui-linear-b}
					\cdef{\mathsf{b}((\neu, \nev),(\bm{\zeta}, \bm{\xi}))} &\coloneqq \llangle \mathbf{J}^e\neu, (\bm{\zeta}, \bm{\xi})\rrangle + \llangle \mathbf{J}^m\nev, (\bm{\zeta}, \bm{\xi})\rrangle,	\\\label{eq:sesqui-linear-c}
					\cdef{\mathsf{c}((\bm{\alpha}, \bm{\beta}), (\bm{w}, \bm{q}))} &\coloneqq \langle (\bm{\mathcal{D}}_{\kappa_1} \bm{\alpha} - \bm{\mathcal{T}}^{\tilde{\varepsilon}, \tilde{\mu}}_{\kappa_1}\bm{\beta}), \bm{w}\rangle_{\Omega_i} + \langle (\bm{\mathcal{D}}_{\kappa_1} (\mathsf{R}\bm{\beta}) + \kappa_1^2\bm{\mathcal{T}}^{\tilde{\mu}, \tilde{\varepsilon}}_{\kappa_1}(\mathsf{R}\bm{\alpha})), \bm{q}\rangle_{\Omega_i},\\\label{eq:sesqui-linear-d}
					\cdef{\mathsf{d}((\neu, \nev), (\bm{w}, \bm{q})) }&\coloneqq \langle \neu - \mathbf{\Lambda}^e\neu - \mathbf{\Xi}^m\nev, \bm{w}\rangle_{\Omega_i} + \langle \nev - \mathbf{\Lambda}^m\nev - \mathbf{\Xi}^e\neu, \bm{q}\rangle_{\Omega_i}
				\end{align}
			\end{problem}
		\end{tcolorbox}
		\begin{proposition}
			The bilinear forms defined in \eqref{eq:sesqui-linear-a}--\eqref{eq:sesqui-linear-d}
			\begin{align}
				&\mathsf{a} \ : \ \bm{\mathcal{H}}(\Gamma) \times \bm{\mathcal{H}}(\Gamma) \rightarrow \mathbb{C},\\
				&\mathsf{b} \ : \ \bm{\mathcal{H}}(\Omega_i)\times \bm{\mathcal{H}}(\Gamma) \rightarrow \mathbb{C},\\
				&\mathsf{c} \ : \ \bm{\mathcal{H}}(\Gamma) \times \bm{\mathcal{H}}(\Omega_i)' \rightarrow \mathbb{C},\\
				&\mathsf{d} \ : \ \bm{\mathcal{H}}(\Omega_i)\times \bm{\mathcal{H}}(\Omega_i)' \rightarrow \mathbb{C},
			\end{align}
			are all continuous.
		\end{proposition}
		\begin{proof}
			The result follows from
			\begin{itemize}
				\item Continuity of the vector-valued Newton potential (Proposition \ref{eq:vectorNewton-SL}).
				\item Continuity of the Maxwell layer potentials (Proposition \ref{SLandDL}).
				\item Continuity of trace operators \cite[Theorem~4.1]{buffa2002traces}.
			\end{itemize}
		\end{proof}
		
		\subsection{Coercivity of weak STF-VIEs}\label{sec:coercive}
		
		Based on the results shown in Section \ref{sec:Newton}, we establish mapping properties for the operators defined in Section \ref{sec:STF}. 
		\begin{proposition}
			The Newton potentials $ \mathsf{N}_{\Omega_i, \star} $ and $\mathbf{N}_{\Omega_i, \star}$.
			\begin{subequations}
				\begin{align}\label{eq:mappingA}
					\mathbf{N}_{\Omega_i, \star} &: \mathbf{H}(\Curl, \Omega_i)\rightarrow \mathbf{H}^1(\Curl, \Omega_i),\\ \label{eq:mappingB}
					\mathbf{N}_{\Omega_i, \star}\Curl&: \mathbf{H}(\Curl, \Omega_i)\rightarrow \mathbf{H}^1(\Curl, \Omega_i),\\ \label{eq:mappingC}
					\grad \mathsf{N}_{\Omega_i, \star}&: L^2(\Omega_i)\rightarrow \mathbf{H}^1(\Curl, \Omega_i).
				\end{align}
			\end{subequations}
			are continuous.
		\end{proposition}
		\begin{proof}
			We know that 
			\begin{equation}\label{eq:NewtonH2}
				\mathbf{N}_{\Omega_i, \star} : \mathbf{L}^2(\Omega_i)\rightarrow \mathbf{H}^2(\Omega_i).
			\end{equation}
			We also know the continuous embeddings $\mathbf{H}(\Curl,\Omega_i)\subset \mathbf{L}^2(\Omega_i)$ and $\mathbf{H}^2(\Omega_i)\subset \mathbf{H}^1(\Curl,\Omega_i).$ Therefore, we conclude \eqref{eq:mappingA}. \\
			We can show \eqref{eq:mappingB} from \eqref{eq:NewtonH2}, since $\Curl : \mathbf{H}(\Curl, \Omega_i)\rightarrow \mathbf{L}^2(\Omega_i)$ is bounded.\\
			To show \eqref{eq:mappingC}, we consider $f\in L^2(\Omega_i)$. We know from \eqref{eq:mapping_newton} that $\mathsf{N}_{\Omega_i,\star}f \in H^2(\Omega_i),$ and therefore $\grad \mathsf{N}_{\Omega_i,\star}f \in \mathbf{H}^1(\Omega_i).$ In addition, we know that $
			\Curl (\grad\mathsf{N}_{\Omega_i,\star}f) \equiv 0.$ Since both $\mathsf{N}_{\Omega_i,\star}f$ and its $\Curl$ belong to $\mathbf{H}^1(\Omega_i),$ we conclude that $\mathsf{N}_{\Omega_i,\star} : L^2(\Omega_i)\rightarrow \mathbf{H}^1(\Curl, \Omega_i).$
		\end{proof}
		
		\begin{corollary}\label{compactness}
			Under Assumption \ref{assumption}, the operators 
			\begin{align*}
				&\mathbf{\Lambda}^e : \mathbf{H}(\Curl, \Omega_i)\rightarrow \mathbf{H}(\Curl, \Omega_i), &\mathbf{\Lambda}^m : \mathbf{H}(\Curl, \Omega_i)\rightarrow \mathbf{H}(\Curl, \Omega_i),\\
				&\mathbf{\Xi}^e : \mathbf{H}(\Curl, \Omega_i)\rightarrow \mathbf{H}(\Curl, \Omega_i), &\mathbf{\Xi}^m : \mathbf{H}(\Curl, \Omega_i)\rightarrow \mathbf{H}(\Curl, \Omega_i),
			\end{align*}
			defined as in \eqref{eq:VolK}--\eqref{eq:VolA} are compact. In particular, the operators
			\begin{align*}
				&\mathbf{J}^e : \mathbf{H}(\Curl, \Omega_i)\rightarrow \bm{\mathcal{H}}(\Gamma), &\mathbf{J}^m : \mathbf{H}(\Curl, \Omega_i)\rightarrow \bm{\mathcal{H}}(\Gamma), 
			\end{align*}
			defined in \eqref{eq:tracesVolE} and \eqref{eq:tracesVolM} are compact.
		\end{corollary}
		\begin{proof}
			From \eqref{eq:VolK} and \eqref{eq:VolA}, we observe that $\mathbf{\Lambda}^e, \mathbf{\Lambda}^m, \mathbf{\Xi}^e$ and $\mathbf{\Xi}^m$ are linear combinations of operators of the form \eqref{eq:mappingA}-\eqref{eq:mappingC}. Multipliers are all bounded and smooth, therefore they map elements of $\mathbf{H}(\Curl, \Omega_i)$ to $\mathbf{H}(\Curl, \Omega_i)$, and to $L^2(\Omega_i).$ \\
			The result follows by Rellich's embedding theorem \cite[Theorem~2.5.5]{sauter2010boundary}, which states compact inclusion from $H^1(\Omega_i)$ into $L^2(\Omega_i)$ (and therefore from $\mathbf{H}^1(\Omega_i)$ into $\mathbf{L}^2(\Omega_i)$).
		\end{proof}
		
		The left-hand side of the STF-VIE \eqref{eq:STF-VIE} can be decomposed into several operators as suggested by the operator matrix notation in \eqref{eq:STF-VIE}. An abstract analysis on such block operators is given in Appendix \ref{BlockOps}. In particular, we need to establish the coercivity/inf-sup stability of the diagonal operators $$\mathbf{M}^{-1}\mathbf{A}_{\kappa_0}\mathbf{M} + \mathbf{A}_{\kappa_1}^{\tilde{\varepsilon},\tilde{\mu}} \ : \ \bm{\mathcal{H}}(\Gamma)\rightarrow\bm{\mathcal{H}}(\Gamma),$$ and
		\begin{equation*}
			\begin{pmatrix}
				\mathbf{I} - \mathbf{\Lambda}^e & -\mathbf{\Xi}^m\\
				-\mathbf{\Xi}^e & \mathbf{I}-\mathbf{\Lambda}^m
			\end{pmatrix} \ : \ \bm{\mathcal{H}}(\Omega_i) \rightarrow \bm{\mathcal{H}}(\Omega_i).
		\end{equation*}
		After establishing stability and uniqueness of solutions, from Proposition \ref{th:well_posed} we will be able to infer well-posedness of the continuous problem.\\
		
		The first step is to show that a generalized G\r{a}rding inequality (T-coercivity) holds for the (weighted) Maxwell Calder\'on operator $ \mathbf{A}^{\tilde{\varepsilon}, \tilde{\mu}}_{\kappa_1} $ from \eqref{eq:weightedCalderon}. We start with the following result for (weighted) scalar and vector-valued single layer operators.
		\begin{lemma}\label{lemma:weighted} Let $ \chi \in C^{1}(\overbar{\Omega}_i) $ be such that $$ 0 < \chi_{\text{min}} < \chi(\nex) < \chi_{\text{max}} $$ for all $ \nex\in \Omega_i. $ Let $\kappa_1 > 0 $ and $ \mathsf{V}_{\kappa_1} : H^{-1/2}(\Gamma) \rightarrow H^{1/2}(\Gamma) $ be the scalar single layer boundary integral operator with wavenumber $ \kappa_1 $. Then, there exist a compact operator $ \Theta_{\chi} : H^{-1/2}(\Gamma) \rightarrow H^{1/2}(\Gamma)  $ and $ c_{\chi} > 0 $ such that
			\begin{equation}
				\mathrm{Re} \left\{ 	\langle \mathsf{V}_{\kappa_1} (\chi \varphi), \overline{\varphi}\rangle_{\Gamma}  + \langle \Theta_{\chi}\varphi, \overline{\varphi}\rangle_{\Gamma} \right\} \geq c_{\chi} \norm{\varphi}_{H^{-1/2}(\Gamma)}^2
			\end{equation}
			holds for all $ \varphi \in H^{-1/2}(\Gamma). $\\
			The result can also be extended to the vectorial case. There exist a compact operator $\mathbf{\Theta}_{\chi}:\mathbf{H}^{-1/2}(\dv_{\Gamma},\Gamma)\rightarrow \mathbf{H}^{-1/2}(\curl_{\Gamma},\Gamma)$ and a constant $C_{\chi}>0$ such that 
			\begin{equation}
				\mathrm{Re}\left\{ \langle \mathbf{V}_{\kappa_1}^{\mathbf{t}}(\chi\bm{\beta}), \overline{\bm{\beta}}\rangle_{\bm{\tau}, \Gamma} +  \langle \mathbf{\Theta}_{\chi}\bm{\beta}, \overline{\bm{\beta}}\rangle_{\bm{\tau}, \Gamma}\right\} \geq C_{\chi}\norm{\bm{\beta}}^2_{\mathbf{H}^{-1/2}(\Gamma)}
			\end{equation}
			holds for all $\bm{\beta}\in \mathbf{H}^{-1/2}(\dv_{\Gamma}0,\Gamma).$
		\end{lemma}
		\begin{proof}
			Note that $ \chi^{1/2} $ is well defined and $ \chi^{1/2}\in C^{1}(\overbar{\Omega}_i). $ Then
			\begin{align*}
				\langle \mathsf{V}_{\kappa_1}(\chi \varphi), \overline{\varphi} \rangle_{\Gamma} &= \langle \mathsf{V}_{\kappa_1}(\chi^{1/2} \chi^{1/2} \varphi), \overline{\varphi} \rangle_{\Gamma} \\
				&= \langle \chi^{1/2}\mathsf{V}_{\kappa_1}(\chi^{1/2} \varphi), \overline{\varphi} \rangle_{\Gamma} + 
				\langle (\mathsf{V}_{\kappa_1} \chi - \chi^{1/2}\mathsf{V}_{\kappa_1}\chi^{1/2}) \varphi, \overline{\varphi} \rangle_{\Gamma} \\ \label{eq:proof_coercivity}
				&= \langle \mathsf{V}_{\kappa_1}(\chi^{1/2} \varphi),\overline{ (\chi^{1/2}\varphi)} \rangle_{\Gamma} + 
				\langle (\mathsf{V}_{\kappa_1} \chi - \chi^{1/2}\mathsf{V}_{\kappa_1}\chi^{1/2}) \varphi, \overline{\varphi} \rangle_{\Gamma}.
			\end{align*}
			We define $ \Theta_{\chi} \coloneqq \mathsf{V}_{\kappa_1} \chi - \chi^{1/2}\mathsf{V}_{\kappa_1}\chi^{1/2}. $ This is a compact operator due to the cancellation of singularity at $ \nex = \ney $:
			\begin{equation*}
				G(\nex, \ney) \chi(\ney) - G(\nex, \ney)\chi^{1/2}(\nex)\chi^{1/2}(\ney) = G(\nex, \ney) \chi^{1/2}(\ney)(\chi^{1/2}(\ney) - \chi^{1/2}(\nex)).
			\end{equation*}
			We also know that there exists a compact operator $ \Theta_{\mathsf{V}_{\kappa_1}} :  H^{-1/2}(\Gamma) \rightarrow H^{1/2}(\Gamma)  $ such that
			\begin{equation*}
				\mathrm{Re} \left\{ 	\langle \mathsf{V}_{\kappa_1} \psi, \overline{\psi} \rangle_{\Gamma}  + \langle \Theta_{\mathsf{V}_{\kappa_1}}\psi, \overline{\psi}\rangle_{\Gamma} \right\} \geq c_{\mathsf{V}_{\kappa_1}} \norm{\psi}_{H^{-1/2}(\Gamma)}^2
			\end{equation*}
			for every $ \psi\in H^{-1/2}(\Gamma) $. Note that $ \chi^{1/2}\varphi \in H^{-1/2}(\Gamma) $ for $ \chi^{1/2}$ smooth and $ \varphi \in H^{-1/2}(\Gamma)$. We define $ \Theta \coloneqq \chi^{1/2}\Theta_{\mathsf{V}_{\kappa_1}}\chi^{1/2} - \Theta_{\chi} $ and conclude
			\begin{equation*}
				\mathrm{Re}\left\{	\langle \mathsf{V}_{\kappa_1}(\chi \varphi), \overline{\varphi} \rangle_{\Gamma} + 	\langle \Theta \varphi, \overline{\varphi} \rangle_{\Gamma}\right\} \geq c_{\mathsf{V}_{\kappa_1}}\norm{\chi^{1/2}\varphi}^2_{H^{-1/2}(\Gamma)} \geq c_{\chi} \norm{\varphi}^2_{H^{-1/2}(\Gamma)},
			\end{equation*}
			where $ c_{\chi} $ depends on $ c_{\mathsf{V}_{\kappa_1}} $ and $ \chi,$ and we used Lemma \ref{lemma:eq_norms} in the last inequality. \\
			The result for $\mathbf{V}_{\kappa_1}^{\mathbf{t}}$ can be shown by following the same approach and using Lemma \ref{lemma:eq_norms}.
		\end{proof}
		In the spirit of results valid for the Maxwell Calder\'on operator \cite[Theorem~9]{buffa2003galerkin}, we can state the following proposition.
		\begin{proposition}[Generalized G\r{a}rding inequality for $\mathbf{A}_{\kappa_1}^{\tilde{\varepsilon},\tilde{\mu}}$]\label{GardingWeightedCald}
			Let $ \varepsilon, \mu \in C^1(\overbar{\Omega}_i)$ and define arbitrary positive coefficients $ \varepsilon_1, \mu_1, \kappa_1 $. Let $\mathbf{A}_{\kappa_1}^{\tilde{\varepsilon},\tilde{\mu}}$ be defined as in \eqref{eq:weightedCalderon}, where $\tilde{\varepsilon}=\tfrac{\varepsilon}{\varepsilon_1}$ and $\tilde{\mu}=\tfrac{\mu}{\mu_1}$. Then, there is a compact operator $ \mathbf{\Theta}_{\mathbf{A}} : \bm{\mathcal{H}}(\Gamma)\rightarrow \bm{\mathcal{H}}(\Gamma) $, an isomorphism $ \mathbf{X}_{\Gamma} : \bm{\mathcal{H}}(\Gamma)\rightarrow \bm{\mathcal{H}}(\Gamma)$ and a constant $ C_{\mathbf{A}} > 0 $ depending on $ \varepsilon,\mu, \varepsilon_1, \mu_1, \kappa_1, \Omega_i, $ such that
			\begin{equation}\label{eq:weightedMaxCald}
				\mathrm{Re}\left\{ \LLangle \mathbf{A}_{\kappa_1}^{\tilde{\varepsilon},\tilde{\mu}}\begin{pmatrix}
					\bm{\alpha} \\ \bm{\beta}
				\end{pmatrix}, \mathbf{X}_{\Gamma}  \begin{pmatrix}
					\overline{\bm{\alpha}} \\ \overline{\bm{\beta}}
				\end{pmatrix} \RRangle_{\Gamma} + \LLangle \mathbf{\Theta}_{\mathbf{A}}\begin{pmatrix}
					\bm{\alpha} \\ \bm{\beta}
				\end{pmatrix}, \begin{pmatrix}
					\overline{\bm{\alpha}} \\ \overline{\bm{\beta}}
				\end{pmatrix} \RRangle_{\Gamma}   \right\} \geq C_{\mathbf{A}}\norm{(\bm{\alpha}, \bm{\beta})}^2_{\bm{\mathcal{H}}(\Gamma)}.
			\end{equation}
		\end{proposition}
		\begin{proof}
			The proof is largely based on the one in \cite[Theorem~9]{buffa2003galerkin}, with the difference that the operator $ \mathbf{A}_{\kappa_1}^{\tilde{\varepsilon}, \tilde{\mu}} $ is weighted by the strictly positive $C^1$--smooth multipliers $ \tilde{\varepsilon} $ and $ \tilde{\mu}$.\\
			We use the regular decomposition theorem \cite[Lemma~2]{buffa2003galerkin}: $\bm{\alpha}\in \mathbf{H}^{-1/2}(\curl_{\Gamma},\Gamma)$ can be written as 
			\begin{equation}\label{eq:decomp-alpha}
				\bm{\alpha} = \bm{\alpha}_{\perp} + \bm{\alpha}_0, \quad \bm{\alpha}_{\perp}\in \mathbf{H}^{1/2}_{\parallel}(\Gamma), \ \bm{\alpha}_0\in \mathbf{H}^{-1/2}(\curl_{\Gamma}0, \Gamma),
			\end{equation}
			where $$\norm{\bm{\alpha}_{\perp}}_{\mathbf{H}^{1/2}_{\parallel}(\Gamma)} \leq C\norm{\curl_{\Gamma}\bm{\alpha}}_{H^{-1/2}(\Gamma)}.$$
			Similarly, $\bm{\beta}\in \mathbf{H}^{-1/2}(\dv_{\Gamma}, \Gamma)$ can be written as
			\begin{equation}\label{eq:decomp-beta}
				\bm{\beta} = \bm{\beta}_{\perp} + \bm{\beta}_0, \quad \bm{\beta}_{\perp}\in \mathbf{H}^{1/2}_{\times}(\Gamma), \ \bm{\beta}_0\in \mathbf{H}^{-1/2}(\dv_{\Gamma}0, \Gamma),
			\end{equation}
			where $$\norm{\bm{\beta}_{\perp}}_{\mathbf{H}^{1/2}_{\times}(\Gamma)} \leq C\norm{\dv_{\Gamma}\bm{\beta}}_{H^{-1/2}(\Gamma)}.$$
			We define 
			\begin{align}
				\mathbf{X}_{\Gamma}\begin{pmatrix}
					\bm{\alpha} \\ \bm{\beta}
				\end{pmatrix} \coloneqq \begin{pmatrix}
					\ \ 	\bm{\alpha}_{\perp} - \bm{\alpha}_0 \\ \bm{\beta}_{\perp}-\bm{\beta}_0
				\end{pmatrix}, \quad \text{ for all }\bm{\alpha}\in\mathbf{H}^{-1/2}(\curl_{\Gamma}, \Gamma), \ \bm{\beta}\in \mathbf{H}^{-1/2}(\dv_{\Gamma}, \Gamma).
			\end{align}
			Now we write
			\begin{align}\label{eq:Aepsmu1}\begin{aligned}
					\LLangle\mathbf{A}_{\kappa_1}^{\tilde{\varepsilon}, \tilde{\mu}}\begin{pmatrix}
						\bm{\alpha} \\ \bm{\beta}
					\end{pmatrix}, \mathbf{X}_{\Gamma}\begin{pmatrix}
						\overline{\bm{\alpha}} \\ \overline{\bm{\beta}}
					\end{pmatrix}\RRangle_{\Gamma} &{}=		\LLangle\mathbf{A}_{\kappa_1}^{\tilde{\varepsilon}, \tilde{\mu}}\begin{pmatrix}
						\bm{\alpha}_{\perp} \\ \bm{\beta}_0
					\end{pmatrix}, \begin{pmatrix}
						\ \ \overline{\bm{\alpha}}_{\perp} \\ -\overline{\bm{\beta}}_0
					\end{pmatrix}\RRangle_{\Gamma} \\
					&{}+ \LLangle\mathbf{A}_{\kappa_1}^{\tilde{\varepsilon}, \tilde{\mu}}\begin{pmatrix}
						\bm{\alpha}_0 \\ \bm{\beta}_{\perp}
					\end{pmatrix}, \begin{pmatrix}
						\ \ \overline{\bm{\alpha}}_{\perp} \\ -\overline{\bm{\beta}}_0
					\end{pmatrix}\RRangle_{\Gamma} \\
					&{}+\LLangle\mathbf{A}_{\kappa_1}^{\tilde{\varepsilon}, \tilde{\mu}}\begin{pmatrix}
						\bm{\alpha}_{\perp} \\ \bm{\beta}_0
					\end{pmatrix}, \begin{pmatrix}
						-\overline{\bm{\alpha}}_{0} \\ \ \ \overline{\bm{\beta}}_{\perp}
					\end{pmatrix}\RRangle_{\Gamma} \\
					&{}+ \LLangle\mathbf{A}_{\kappa_1}^{\tilde{\varepsilon}, \tilde{\mu}}\begin{pmatrix}
						\bm{\alpha}_0 \\ \bm{\beta}_{\perp}
					\end{pmatrix}, \begin{pmatrix}
						-\overline{\bm{\alpha}}_{0} \\ \ \ \overline{\bm{\beta}}_{\perp}
					\end{pmatrix}\RRangle_{\Gamma} 
				\end{aligned}
			\end{align}
			We study the first term in the right-hand side of \eqref{eq:Aepsmu1},
			\begin{align}
				\LLangle\mathbf{A}_{\kappa_1}^{\tilde{\varepsilon}, \tilde{\mu}}\begin{pmatrix}
					\bm{\alpha}_{\perp} \\ \bm{\beta}_0
				\end{pmatrix}, \begin{pmatrix}
					\ \ \overline{\bm{\alpha}}_{\perp} \\ -\overline{\bm{\beta}}_0
				\end{pmatrix}\RRangle_{\Gamma} =& - \langle \mathbf{K}_{\kappa_1}\bm{\alpha}_{\perp}, \overline{\bm{\beta}}_0\rangle_{\Gamma} - \langle\mathbf{V}_{\kappa_1}^{\tilde{\varepsilon},\tilde{\mu}}\bm{\beta}_0, \overline{\bm{\beta}}_0\rangle_{\Gamma} \\ \nonumber
				&+\langle \mathbf{W}_{\kappa_1}^{\tilde{\mu}, \tilde{\varepsilon}}\bm{\alpha}_{\perp}, \overline{\bm{\alpha}}_{\perp}\rangle_{\Gamma} + \langle \mathbf{K}'_{\kappa_1}\bm{\beta}_0, \overline{\bm{\alpha}}_{\perp}\rangle_{\Gamma},
			\end{align}
			where
			\begin{subequations}
				\begin{align}\label{eq:V1emu}
					\langle\mathbf{V}_{\kappa_1}^{\tilde{\varepsilon},\tilde{\mu}}\bm{\beta}_0, \overline{\bm{\beta}}_0\rangle_{\Gamma} &= \langle \mathbf{V}^{\mathbf{t}}_{\kappa_1}(\tilde{\mu}\bm{\beta}_0), \overline{\bm{\beta}}_0\rangle_{\Gamma},\\ \label{eq:W1mue}
					\langle\mathbf{W}_{\kappa_1}^{\tilde{\mu},\tilde{\varepsilon}}\bm{\alpha}_{\perp}, \overline{\bm{\alpha}}_{\perp}\rangle_{\Gamma} &= \langle \mathsf{V}_{\kappa_1}(\tfrac{1}{\tilde{\mu}}\curl_{\Gamma}\bm{\alpha}_{\perp}), \curl_{\Gamma}\overline{\bm{\alpha}}_{\perp}\rangle_{\Gamma} + \kappa_1^2 \langle \mathbf{V}_{\kappa_1}^{\bm{\tau}}(\tilde{\varepsilon}\bm{\alpha}_{\perp}), \overline{\bm{\alpha}}_{\perp}\rangle_{\Gamma}
				\end{align}
			\end{subequations}
			The second term in \eqref{eq:W1mue} is a compact perturbation since $\bm{\alpha}_{\perp}\in \mathbf{H}^{1/2}_{\times}(\Gamma)$ and $\mathbf{H}^{1/2}_{\times}(\Gamma)$ is compactly embedded in $\mathbf{H}^{-1/2}_{\times}(\Gamma)$ \cite[Corollary~1]{buffa2003galerkin}, while for the first term in \eqref{eq:W1mue} and \eqref{eq:V1emu} we have a coercivity result that follows from Lemma \ref{lemma:weighted} \:
			\begin{subequations}
				\begin{align}
					\mathrm{Re}\{ \langle \mathbf{V}^{\mathbf{t}}_{\kappa_1}(\tilde{\mu}\bm{\beta}_0), \overline{\bm{\beta}}_0\rangle_{\Gamma} + \mathsf{t}(\bm{\beta}_0, \overline{\bm{\beta}}_0) \} &\geq c_1 \norm{\bm{\beta}_0}_{\mathbf{H}^{-1/2}(\Gamma)}^2, \\
					\mathrm{Re}\{ \langle \mathsf{V}_{\kappa_1}(\tfrac{1}{\tilde{\varepsilon}}\curl_{\Gamma}\bm{\alpha}_{\perp}), \curl_{\Gamma}\overline{\bm{\alpha}}_{\perp}\rangle_{\Gamma} + \mathsf{t}(\bm{\alpha}_{\perp}, \overline{\bm{\alpha}}_{\perp}) \} &\geq c_2 \norm{\curl_{\Gamma}\bm{\alpha}_{\perp}}_{H^{-1/2}(\Gamma)}^2.
				\end{align}
			\end{subequations}
			On the other hand, from \cite[Lemma~6]{buffa2003galerkin}, a symmetry between $\mathbf{K}_{\kappa_1}$ and $\mathbf{K}_{\kappa_1}'$ with respect to the duality pairing $\langle \cdot, \cdot \rangle_{\Gamma}$ implies \cite[Theorem~9]{buffa2003galerkin}
			\begin{equation}
				-\langle\mathbf{K}_{\kappa_1}\bm{\alpha}_{\perp}, \overline{\bm{\beta}}_0\rangle_{\Gamma} + \langle\mathbf{K}_{\kappa_1}'\bm{\beta}_0, \overline{\bm{\alpha}}_{\perp}\rangle_{\Gamma} = -2i\text{Im}\left\{ \langle  \mathbf{K}_{\kappa_1}\bm{\alpha}_{\perp}, \overline{\bm{\beta}}_0\rangle_{\Gamma} \right\}.
			\end{equation}
			Then, we can establish that there exist a compact perturbation  $\mathbf{\Theta}_{\mathbf{A},1}$ and a constant $c_{\mathbf{A},1}>0$ such that
			\begin{align}\label{eq:Part1}
				\mathrm{Re}\left\{ \LLangle \mathbf{A}_{\kappa_1}^{\tilde{\varepsilon}, \tilde{\mu}}\begin{pmatrix}
					\bm{\alpha}_{\perp} \\ \bm{\beta}_0
				\end{pmatrix}, \begin{pmatrix}
					\ \ \overline{\bm{\alpha}}_{\perp} \\ -\overline{\bm{\beta}}_0
				\end{pmatrix}\RRangle_{\Gamma} + \LLangle \mathbf{\Theta}_{\mathbf{A},1}\begin{pmatrix}
					\bm{\alpha} \\ \bm{\beta}
				\end{pmatrix}, \begin{pmatrix}
					\overline{\bm{\alpha}} \\ \overline{\bm{\beta}}
				\end{pmatrix}\RRangle_{\Gamma} \right\}  \geq& c_{\mathbf{A},1}\left(\norm{\curl_{\Gamma}\bm{\alpha}_{\perp}}_{H^{-1/2}(\Gamma)}^2 \right.\\ \nonumber
				& \left.+ \norm{\bm{\beta}_0}_{\mathbf{H}^{-1/2}(\Gamma)}^2\right).
			\end{align} 
			In a similar way, we study the third and fourth terms in \eqref{eq:Aepsmu1} and show
			\begin{align}\label{eq:Part2}
				\mathrm{Re}\left\{ \LLangle \mathbf{A}_{\kappa_1}^{\tilde{\varepsilon}, \tilde{\mu}}\begin{pmatrix}
					\bm{\alpha}_{0} \\ \bm{\beta}_{\perp}
				\end{pmatrix}, \begin{pmatrix}
					\ \ \overline{\bm{\alpha}}_{0} \\ \overline{\bm{\beta}}_{\perp}
				\end{pmatrix}\RRangle_{\Gamma} + \LLangle \mathbf{\Theta}_{\mathbf{A},2}\begin{pmatrix}
					\bm{\alpha} \\ \bm{\beta}
				\end{pmatrix}, \begin{pmatrix}
					\overline{\bm{\alpha}} \\ \overline{\bm{\beta}}
				\end{pmatrix}\RRangle_{\Gamma} \right\}  \geq& c_{\mathbf{A},2}\left(\norm{\dv_{\Gamma}\bm{\beta}_{\perp}}_{H^{-1/2}(\Gamma)}^2 \right.\\ \nonumber
				& \left.+ \norm{\bm{\alpha}_0}_{\mathbf{H}^{-1/2}(\Gamma)}^2\right).
			\end{align} 
			Combining \eqref{eq:Part1} and \eqref{eq:Part2}, and by the stability of the decomposition in \eqref{eq:decomp-alpha} and \eqref{eq:decomp-beta}, we conclude that \eqref{eq:weightedMaxCald} holds.
		\end{proof}
		\begin{corollary}\label{th:Tcoercive-weighted-STF}
			Let us define $ \mathbb{A} \coloneqq  \mathbf{M}^{-1}\mathbf{A}_{\kappa_0}\mathbf{M} + \mathbf{A}_{\kappa_1}^{\tilde{\varepsilon}, \tilde{\mu}}$.  Then, under Assumption \ref{assumption},  there is a compact operator $ \mathbf{\Theta}_{\mathbb{A}} : \bm{\mathcal{H}}(\Gamma)\rightarrow \bm{\mathcal{H}}(\Gamma) $, an isomorphism $ \mathbb{X}_{\Gamma} : \bm{\mathcal{H}}(\Gamma)\rightarrow \bm{\mathcal{H}}(\Gamma)$ and a constant $ C_{\mathbb{A}} > 0 $ depending on $ \varepsilon,\mu, \varepsilon_1, \mu_1, \varepsilon_0, \mu_0, \Omega_i, $ such that
			\begin{equation}
				\mathrm{Re}\left\{ \left\langle \mathbb{A} \begin{pmatrix}
					\bm{\alpha} \\ \bm{\beta}
				\end{pmatrix}, \mathbb{X}_{\Gamma}  \begin{pmatrix}
					\overline{\bm{\alpha}} \\ \overline{\bm{\beta}}
				\end{pmatrix} \right\rangle_{\Gamma} + \left\langle \mathbf{\Theta}_{\mathbb{A}}\begin{pmatrix}
					\bm{\alpha} \\ \bm{\beta}
				\end{pmatrix}, \begin{pmatrix}
					\overline{\bm{\alpha}} \\ \overline{\bm{\beta}}
				\end{pmatrix} \right\rangle_{\Gamma}   \right\} \geq C_{\mathbb{A}}\norm{(\bm{\alpha}, \bm{\beta})}^2_{\bm{\mathcal{H}}(\Gamma)},
			\end{equation} 
			where 
			$$\norm{(\bm{\alpha},\bm{\beta})}^2_{\bm{\mathcal{H}(\Gamma)}} \coloneqq \norm{\bm{\alpha}}_{\mathbf{H}^{-1/2}(\curl_{\Gamma},\Gamma)}^2+ \norm{\bm{\beta}}^2_{\mathbf{H}^{-1/2}(\dv_{\Gamma}, \Gamma)}, $$
			for all $\bm{\alpha}\in\mathbf{H}^{-1/2}(\curl_{\Gamma}, \Gamma)$ and $\bm{\beta}\in\mathbf{H}^{-1/2}(\dv_{\Gamma}, \Gamma).$
		\end{corollary}
		\begin{proof}
			The proof for a Generalized G\r{a}rding inequality for $\mathbf{M}^{-1}\mathbf{A}_{\kappa_0}\mathbf{M}$ follows the same approach as in Proposition \ref{GardingWeightedCald}, also usign the isomorphism $\mathbf{X}_{\Gamma}$. The result follows by noting that $\mathbb{A}$ is a linear combination of two operators that satisfy a Generalized G\r{a}rding inequality with same isomorphisms $\mathbf{X}_{\Gamma}$ (Proposition \ref{GardingWeightedCald}).
		\end{proof}

		\begin{proposition}[Generalized G\r{a}rding inequality for $ \mathbf{I}- \mathbf{\Lambda}^{\star} $] \label{IminusA}
			Let $ \mathbf{\Lambda}^{\star}, \ \star = \{e, m\} $ defined as in \eqref{eq:VolK} or \eqref{eq:VolA},
			\begin{align*}
				\mathbf{\Lambda}^e \neu &\coloneqq -\kappa_1^2 \mathbf{N}_{\Omega_i,\kappa_1}(p_e\neu) + \grad \mathsf{N}_{\Omega_i,\kappa_1} (\bm{\tau}_e \cdot \neu),\\ 
				\mathbf{\Lambda}^m \nev &\coloneqq -\kappa_1^2 \mathbf{N}_{\Omega_i,\kappa_1}(p_m\nev) + \grad \mathsf{N}_{\Omega_i,\kappa_1} (\bm{\tau}_m\cdot \nev).
			\end{align*}
			Then, there exist a compact operator $ \Theta_{\star} : \mathbf{H}(\Curl, \Omega_i)\rightarrow \mathbf{H}(\Curl, \Omega_i)$, an isomorphism $ X_{\star} : \mathbf{H}(\Curl, \Omega_i)\rightarrow (\mathbf{H}(\Curl, \Omega_i))'$ and a constant $ C_{\star} > 0 $ such that
			\begin{equation}\label{eq:TcoerciveU}
				\mathrm{Re}\left\{ \langle  (\mathbf{I} - \mathbf{\Lambda}^{\star})\bm{u}, X_{\star}\overline{\bm{u}}\rangle_{\Omega_i} + ( \Theta_{\star}\bm{u}, \overline{\bm{u}})_{\mathbf{H}(\Curl,\Omega_i)}  \right\} \geq C_{\star}\norm{\bm{u}}_{\mathbf{H}(\Curl, \Omega_i)}^2
			\end{equation}
			holds for all $ \bm{u}\in \mathbf{H}(\Curl, \Omega_i). $
		\end{proposition}
		\begin{proof}
			First, we study the duality pairing $ \langle \bm{u}, \bm{w}\rangle_{\Omega_i}$, with $\bm{w}\in (\mathbf{H}(\Curl, \Omega_i))'.$ Note that a simple choice is $X_{\star}$ such that
			\begin{equation}
				\langle \neu, X_{\star} \overline{\nev}\rangle_{\Omega_i} \coloneqq (\neu, \overbar{\nev})_{\mathbf{H}(\Curl, \Omega_i)} = (\neu, \overbar{\nev})_{\Omega_i} + (\Curl \neu, \Curl\overbar{\nev})_{\Omega_i},
			\end{equation}
			for all $\neu,\nev\in\mathbf{H}(\Curl, \Omega_i),$
			i.e. the inverse Riesz isomorphism.\\
			In that case, $\norm{X_{\star}} = 1$ and
			\begin{equation}
				\langle \bm{u}, X_{\star}\overline{\neu}\rangle_{\Omega_i}  = \norm{\neu}_{\mathbf{H}(\Curl, \Omega_i)}^2.
			\end{equation}
			From Proposition \ref{compactness}, we know that $\mathbf{\Lambda}^{\star}$ is compact in $\mathbf{H}(\Curl, \Omega_i).$ Therefore, choosing $\Theta_{\star} = \mathbf{\Lambda}^{\star}$ leads to \eqref{eq:TcoerciveU}.
		\end{proof}

		\subsection{STF-VIEs: Uniqueness of solutions}\label{sec:unique}
		The results from this section require an assumption on the material properties $\varepsilon  $ and $ \mu. $ 
		\begin{tcolorbox}[colback=lightgray!15!white, colframe=black, sharp corners, colframe=lightgray!15!white]
			\begin{assumption}\label{ConstCoeffs}
				We assume that the material properties $ \varepsilon $ and $\mu  $ are constant on the interface $ \Gamma, $ i.e.
				\begin{equation}
					\varepsilon(\nex)\equiv \varepsilon_1, \quad \mu(\nex) \equiv \mu_1, \quad \text{ for all }\nex\in \Gamma.
				\end{equation}
			\end{assumption}
		\end{tcolorbox}
		\begin{proposition}\label{unique}
			Under Assumption \ref{ConstCoeffs}, there exists a unique solution to Problem \ref{eq:Problem}.
		\end{proposition}
		\begin{proof}
			Let us assume that we have a solution $ \neu \in \mathbf{H}(\Curl, \Omega_i), \ \tilde{\nev} \in \mathbf{H}(\Curl, \Omega_i), \  (\bm{\alpha}, \bm{\beta})\in\mathbf{H}^{-1/2}(\curl_{\Gamma}, \Gamma)\times \mathbf{H}^{-1/2}(\dv_{\Gamma}, \Gamma)$ such that
			\begin{subequations}
				\begin{align}
					\left(\mathbf{M}^{-1}\mathbf{A}_{\kappa_0}\mathbf{M} + \mathbf{A}_{\kappa_1}^{\tilde{\varepsilon}, \tilde{\mu}} \right) \begin{pmatrix}
						\bm{\alpha} \\ \bm{\beta}
					\end{pmatrix}  + \mathbf{J}^e \neu  + \mathbf{J}^m \tilde{\nev}&= 0,\label{eq:STF1-proof}\\
					\bm{\mathcal{D}}_{\kappa_1}(\bm{\alpha}) 	-  \bm{\mathcal{T}}^{\tilde{\varepsilon}, \tilde{\mu}}_{\kappa_1}(\bm{\beta}) +\neu  -\mathbf{\Lambda}^e\neu - \mathbf{\Xi}^m\tilde{\nev} &= 0, \label{eq:STFU-proof}\\
					\kappa_1^2\bm{\mathcal{T}}^{\tilde{\mu}, \tilde{\varepsilon}}_{\kappa_1}(\mathsf{R}\bm{\alpha}) + \bm{\mathcal{D}}_{\kappa_1}(\mathsf{R}\bm{\beta}) +\tilde{\nev}  -\mathbf{\Lambda}^m\tilde{\nev} - \mathbf{\Xi}^e\neu &= 0, \label{eq:STFV-proof}
				\end{align}
			\end{subequations}
			in $\mathbf{H}^{-1/2}(\curl_{\Gamma}, \Gamma)\times \mathbf{H}^{-1/2}(\dv_{\Gamma}, \Gamma) \times \mathbf{H}(\Curl, \Omega_i)\times \mathbf{H}(\Curl, \Omega_i).$ \\
			Because we assume $ \varepsilon(\nex)\equiv \varepsilon_1 > 0 $ and $ \mu(\nex) \equiv \mu_1 >0$ for all $ \nex \in \Gamma,$ we can rewrite \eqref{eq:STF1-proof}--\eqref{eq:STFV-proof} as
			\begin{subequations}
				\begin{align}\label{eq:STF1-proof2}
					\left(\mathbf{M}^{-1}\mathbf{A}_{\kappa_0}\mathbf{M} + \mathbf{A}_{\kappa_1} \right) \begin{pmatrix}
						\bm{\alpha} \\ \bm{\beta}
					\end{pmatrix}  + \mathbf{J}^e \neu  + \mathbf{J}^m \tilde{\nev}&= 0,\\ \label{eq:STFU-proof2}
					\bm{\mathcal{D}}_{\kappa_1}(\bm{\alpha}) 	-  \bm{\mathcal{T}}_{\kappa_1}(\bm{\beta}) +\neu  -\mathbf{\Lambda}^e\neu - \mathbf{\Xi}^m\tilde{\nev} &= 0,  \\
					\kappa_1^2\bm{\mathcal{T}}_{\kappa_1}(\mathsf{R}\bm{\alpha}) + \bm{\mathcal{D}}_{\kappa_1}(\mathsf{R}\bm{\beta}) +\tilde{\nev}  -\mathbf{\Lambda}^m\tilde{\nev} - \mathbf{\Xi}^e\neu &= 0, \label{eq:STFV-proof2}
				\end{align}
			\end{subequations}
			The proof is divided into five parts.
			\begin{enumerate}
				\item From \eqref{eq:STFU-proof2} and \eqref{eq:STFV-proof2}, we show that $\neu$ and $\nev$ ``almost" satisfy Maxwell equations in $\Omega$. We have an extra term, which is the gradient of a volume potential.
				\item We show that the extra term from Part 1 is zero, and therefore $\neu$ and $\nev$ satisfy Maxwell equations in $\Omega_i.$
				\item Using Maxwell layer potentials (see \eqref{eq:MaxwellSL} and \eqref{eq:MaxwellDL}) and $\bm{\alpha}$ and $\bm{\beta},$ we define $\neu_0$ and $\nev_0$ such that they satisfy Maxwell equations in $\Omega_o.$ Then, we compute the jumps $\gammat^+\neu_0 - \gammat^-\neu$ and $\gammatau^+\nev_0 - \gammatau^-\nev$. Using \eqref{eq:STF1-proof2}, we show that the jumps are zero.
				\item We conclude that $\neu,\nev,\neu_0$ and $\nev_0$ define solutions for the Maxwell transmission problem with no sources. It follows that all of them are zero. From \eqref{eq:STF1-proof2}, we conclude that $\bm{\alpha}$ and $\bm{\beta}$ are also zero.
			\end{enumerate}
			\begin{proofpart}
				We take the curl of \eqref{eq:STFU-proof2}. Using \eqref{eq:curlSL}, \eqref{eq:curlDL} we get 
				\begin{equation}\label{eq:proof_curlB}
					\kappa_1^2\bm{\mathcal{T}}_{\kappa_1}(\mathsf{R}\bm{\alpha}) + \bm{\mathcal{D}}_{\kappa_1}(\mathsf{R}\bm{\beta}) + \Curl \neu + \kappa_1^2 \Curl \mathbf{N}_{\Omega_i, \kappa_1}(p_e \neu) + \highlightr{cyan}{\Curl\mathbf{N}_{\Omega_i, \kappa_1}(\Curl(p_m \tilde{\nev}))}= 0,
				\end{equation}
				which by \thighlightr{cyan}{integration by parts \eqref{eq:ibpNewton}} and \thighlightr{cyan}{Assumption \ref{ConstCoeffs}}, can be rewritten as
				\begin{equation}\label{eq:proof_curlC}
					\highlightr{green}{\kappa_1^2\bm{\mathcal{T}}_{\kappa_1}(\mathsf{R}\bm{\alpha}) + \bm{\mathcal{D}}_{\kappa_1}(\mathsf{R}\bm{\beta})} + \Curl \neu + \highlightr{green}{\kappa_1^2 \Curl \mathbf{N}_{\Omega_i, \kappa_1}(p_e \neu)} + \highlightr{cyan}{\Curl^2\mathbf{N}_{\Omega_i, \kappa_1}(p_m \tilde{\nev})}= 0.
				\end{equation}
				Similarly, \eqref{eq:STFV-proof2} can be rewritten as
				\begin{equation}\label{eq:proof_V}
					\highlightr{green}{\kappa_1^2\bm{\mathcal{T}}_{\kappa_1}(\mathsf{R}\bm{\alpha}) + \bm{\mathcal{D}}_{\kappa_1}(\mathsf{R}\bm{\beta})} + \tilde{\nev} +  \kappa_1^2 \mathbf{N}_{\Omega_i, \kappa_1}(p_m \tilde{\nev}) -  \grad \mathsf{N}_{\Omega_i,\kappa_1}(\bm{\tau}_m \cdot \tilde{\nev}) + \highlightr{green}{\kappa_1^2\Curl\mathbf{N}_{\Omega_i, \kappa_1}(p_e\neu)} = 0.
				\end{equation}
				\thighlightr{green}{Substracting  \eqref{eq:proof_V} from \eqref{eq:proof_curlC}} we obtain
				\begin{align}\label{eq:proof_Eq1}
					\Curl \neu -  \tilde{\nev} +  \highlightr{pink}{\Curl^2(\mathbf{N}_{\Omega_i, \kappa_1}(p_m \tilde{\nev})) -  \kappa_1^2 \mathbf{N}_{\Omega_i, \kappa_1}(p_m \tilde{\nev}) }+  \grad (\mathsf{N}_{\Omega_i, \kappa_1}(\bm{\tau}_m \cdot \tilde{\nev})) = 0.
				\end{align}
				Note that, as $\bm{w} = \mathbf{N}_{\Omega_i, \kappa_1}\bm{f}$ defines a solution for the (vector) Helmholtz equation 
				$$-\bm{\Delta}\bm{w} - \kappa_1^2\bm{w} = \bm{f} \quad \text{ in } \Omega_i,$$ and $$-\bm{\Delta} = \Curl^2 - \grad\dv,$$ we get
				\begin{subequations}\label{eq:proof_Eq2}
					\begin{align}
						\highlightr{pink}{\Curl^2\mathbf{N}_{\Omega_i, \kappa_1}(p_m \tilde{\nev}) -  \kappa_1^2 \mathbf{N}_{\Omega_i, \kappa_1}(p_m \tilde{\nev})} &= \grad( \dv (\mathbf{N}_{\Omega_i, \kappa_1}(p_m \tilde{\nev}) ))+ p_m\tilde{\nev} \\ \label{eq:ibpDivNewton}
						&= \grad (\mathbf{N}_{\Omega_i, \kappa_1}(\dv(p_m \tilde{\nev})) )+ p_m\tilde{\nev},
					\end{align}
				\end{subequations}
				where \eqref{eq:ibpDivNewton} is obtained by the integration by parts
				\begin{equation}
					\dv(\mathbf{N}_{\Omega_i, \kappa_1}(\bm{f})) = \mathbf{N}_{\Omega_i, \kappa_1}(\dv\bm{f}) - \mathsf{S}_{\kappa_1}(\gamman^-\bm{f}), 
				\end{equation}
				and Assumption \ref{ConstCoeffs}.\\
				From \eqref{eq:proof_Eq1} and \eqref{eq:proof_Eq2} we write 
				\begin{equation}
					\Curl \neu - \tilde{\nev} + p_m\tilde{\nev} + \grad (\mathsf{N}_{\Omega_i, \kappa_1} (\dv(p_m\tilde{\nev}) + \bm{\tau}_m\cdot \tilde{\nev} )) = 0.
				\end{equation}
				Recall that $ p_m(\nex) = 1 - \tfrac{\mu(\nex)}{\mu_1}, $ so we can write
				\begin{equation}\label{eq:proofBeforeTedious}
					\Curl \neu - \tfrac{\mu}{\mu_1}\tilde{\nev}+ \grad (\mathsf{N}_{\Omega_i, \kappa_1} (\dv(p_m\tilde{\nev}) + \bm{\tau}_m\cdot \tilde{\nev} )) = 0.
				\end{equation}
			\end{proofpart}
			\begin{proofpart}
				In this part, we show that the third term in \eqref{eq:proofBeforeTedious} is zero. First, we do the following computation
				\begin{align}\nonumber
					\dv(p_m\tilde{\nev}) + \bm{\tau}_m\cdot \tilde{\nev} &= \dv(p_m\tilde{\nev}) + \tfrac{\grad\mu}{\mu}\cdot \tilde{\nev} \\\nonumber
					&= \dv\tilde{\nev} -\dv(\tfrac{\mu}{\mu_1} \tilde{\nev})  + \tfrac{\grad\mu}{\mu}\cdot \tilde{\nev}\\\nonumber
					&= \dv\tilde{\nev} -\tfrac{1}{\mu_1}(\grad \mu \cdot \tilde{\nev} + \mu \dv(\tilde{\nev}))  + \tfrac{\grad\mu}{\mu}\cdot \tilde{\nev}\\\nonumber
					&= \dv\tilde{\nev} + (-\tfrac{1}{\mu_1} + \tfrac{1}{\mu})\grad \mu \cdot \tilde{\nev} - \tfrac{\mu}{\mu_1} \dv(\tilde{\nev})\\\label{eq:tedious}
					&= \dv\tilde{\nev} + p_m\tfrac{1}{\mu}\grad \mu \cdot \tilde{\nev} - \tfrac{\mu}{\mu_1} \dv(\tilde{\nev})\\\nonumber
					&= p_m\tfrac{1}{\mu}\grad \mu \cdot \tilde{\nev} + p_m\dv(\tilde{\nev})\\\nonumber
					&= p_m\tfrac{1}{\mu}(\grad \mu \cdot \tilde{\nev} + \mu\dv(\tilde{\nev})) = p_m\tfrac{1}{\mu}\dv(\mu\tilde{\nev})
				\end{align}
				From \eqref{eq:tedious}, we can write \eqref{eq:proofBeforeTedious} as
				\begin{equation}\label{eq:proofEqPreDiv}
					\Curl \neu - \frac{\mu}{\mu_1} \tilde{\nev} = -\grad(\mathsf{N}_{\Omega_i,\kappa_1}(p_m \tfrac{1}{\mu}\dv(\mu\tilde{\nev}))).
				\end{equation}
				We take the divergence of \eqref{eq:proofEqPreDiv} and get
				\begin{align}\label{eq:proofDiv}
					-\frac{1}{\mu_1}\dv(\mu\tilde{\nev}) = -\Delta \mathsf{N}_{\Omega_i,\kappa_1}(p_m \tfrac{1}{\mu}\dv(\mu\tilde{\nev})) =\kappa_1^2\mathsf{N}_{\Omega_i,\kappa_1}(p_m \tfrac{1}{\mu}\dv(\mu\tilde{\nev})) + \highlightr{yellow}{p_m \tfrac{1}{\mu}}\dv(\mu\tilde{\nev}),
				\end{align}
				where we used that $w= \mathsf{N}_{\Omega_i, \kappa_1}f$ defines a solution for the (scalar) Helmholtz equation $-\Delta w - \kappa_1^2w = f$.
				Rearranging terms in \eqref{eq:proofDiv} we get
				\begin{equation}
					\left(-\frac{1}{\mu_1} \highlightr{yellow}{- \frac{1}{\mu} +  \frac{1}{\mu} \frac{\mu}{\mu_1}}\right) \dv(\mu \tilde{\nev}) = \kappa_1^2\mathsf{N}_{\Omega_i,\kappa_1}(p_m \tfrac{1}{\mu}\dv(\mu\tilde{\nev})). 
				\end{equation}
				Writing $\eta \coloneqq \tfrac{1}{\mu}\dv(\mu\tilde{\nev})$, \eqref{eq:proofDiv} becomes the Lippmann-Schwinger equation with zero right-hand side \cite[Section~8.2]{colton2012inverse}
				\begin{equation}\label{eq:proofLS}
					\eta + \kappa_1^2 \mathsf{N}_{\Omega_i,\kappa_1}(p_m \eta) = 0.
				\end{equation}
				This is an equivalent formulation to an homogeneous Helmholtz transmission problem (see \cite[Lemma~7]{costabel2010volume}, \cite[Theorem~8.3]{colton2012inverse}).\\
				This problem is known to have a unique solution as long as a unique continuation principle holds \cite[Section~8.3]{colton2012inverse}, which is the case for $p_m\in C^1(\overbar{\Omega}_i).$ \cite[Theorem~8.6]{colton2012inverse}.\\ The homogeneous problem has only the trivial solution, and we know $$\eta = \tfrac{1}{\mu}\dv(\mu\tilde{\nev}) \equiv 0.$$
				It follows that $\neu$ and $\tilde{\nev}$ satisfy
				\begin{equation}\label{eq:max1}
					\Curl \neu - i\omega\mu (\tfrac{1}{i\omega\mu_1}\tilde{\nev}) = 0 \qquad \text{ in }\Omega_i.
				\end{equation}
				We denote $\nev \coloneqq \tfrac{1}{i\omega\mu_1}\tilde{\nev}$. Similar computations show that 
				\begin{equation}\label{eq:max2}
					\Curl \nev - i\omega\varepsilon \neu = 0 \qquad \text{ in }\Omega_i.
				\end{equation}
			\end{proofpart}
			\begin{proofpart}
				Now, we define an exterior field
				\begin{subequations}
					\begin{align}\label{eq:potentials_exterior_proof}
						\neu_0 &= \bm{\mathcal{T}}_{\kappa_0}(\tfrac{\mu_0}{\mu_1} \bm{\beta}) + \bm{\mathcal{D}}_{\kappa_0}(\bm{\alpha}), &\text{ in }\mathbb{R}^3 \setminus\overbar{\Omega}_i,\\
						\tilde{\nev}_0 &= \bm{\mathcal{D}}_{\kappa_0}(\tfrac{\mu_0}{\mu_1} \mathsf{R}\bm{\beta}) + \kappa_0^2\bm{\mathcal{T}}_{\kappa_0}(\mathsf{R}\bm{\alpha}), &\text{ in }\mathbb{R}^3 \setminus\overbar{\Omega}_i,
					\end{align}
				\end{subequations}
				with $\tilde{\nev}_0= i\omega\mu_0\nev_0$
				that satisfies
				\begin{subequations}\label{eq:Maxwell-exterior-proof}
					\begin{align}
						\Curl \neu_0 - i\omega\mu_0 \nev_0 = 0 \quad \text{ in }\mathbb{R}^3 \setminus\overbar{\Omega}_i,\\
						\Curl \nev_0 + i\omega\varepsilon_0 \neu_0 = 0 \quad \text{ in }\mathbb{R}^3 \setminus\overbar{\Omega}_i.
					\end{align}
				\end{subequations}
			\end{proofpart}
			Taking traces on \eqref{eq:potentials_exterior_proof} we get
			\begin{equation}\label{eq:exterior}
				\begin{pmatrix}
					\gammat^+ \neu_0 \\ \gammatau^+ \tilde{\nev}_0
				\end{pmatrix} = \left(\frac{1}{2}\mathbf{I} - \mathbf{A}_{\kappa_0}\right)\mathbf{M}\begin{pmatrix}
					\bm{\alpha} \\ \bm{\beta}
				\end{pmatrix}.
			\end{equation}
			Taking traces on \eqref{eq:STFU-proof2} and \eqref{eq:STFV-proof2} we obtain
			\begin{equation}\label{eq:interior}
				\begin{pmatrix}
					\gammat^- \neu \\ \gammatau^-\tilde{\nev} 
				\end{pmatrix} =  \left(\frac{1}{2}\mathbf{I} + \mathbf{A}_{\kappa_1}\right)\begin{pmatrix}
					\bm{\alpha} \\ \bm{\beta}
				\end{pmatrix} + \mathbf{J}^e\neu + \mathbf{J}^m\tilde{\nev}  
			\end{equation}
			Combining  \eqref{eq:exterior} and \eqref{eq:interior}, from \eqref{eq:STF1-proof2} we conclude that 
			\begin{equation}\label{eq:proof-trans}
				\begin{pmatrix}
					\gammat^+ \neu_0 \\ \gammatau^+ \nev_0
				\end{pmatrix} - \begin{pmatrix}
					\gammat^- \neu \\ \gammatau^-\nev 
				\end{pmatrix} = 0.
			\end{equation}
			\begin{proofpart}
				We know that $\neu, \nev$ satisfy \eqref{eq:max1} and \eqref{eq:max2}. We also know that $\neu_0$ and $\nev_0$ satisfy \eqref{eq:Maxwell-exterior-proof}. Moreover, the transmission conditions \eqref{eq:proof-trans} hold.
				Therefore, 
				\begin{equation*}
					\bm{U}(\nex) \coloneqq \left\{ \begin{array}{ll}
						\neu_0(\nex), & \nex\in \mathbb{R}^3 \setminus \overbar{\Omega}_i,\\
						\neu(\nex), & \nex \in \Omega_i,
					\end{array}\right. \qquad 
					\bm{V}(\nex) \coloneqq \left\{ \begin{array}{ll}
						\nev_0(\nex), & \nex\in \mathbb{R}^3 \setminus \overbar{\Omega}_i,\\
						\nev(\nex), & \nex \in \Omega_i,
					\end{array}\right.
				\end{equation*}
				are solutions of the homogeneous Maxwell transmission problem. It follows that $\bm{U} \equiv 0, \ \bm{V}\equiv 0$ and $\neu \equiv 0, \ \nev\equiv 0.$
				We conclude from \eqref{eq:STF1-proof2} that 
				\begin{equation}\label{eq:injectiveSTF}
					\left(\mathbf{M}^{-1}\mathbf{A}_{\kappa_0}\mathbf{M} + \mathbf{A}_{\kappa_1} \right) \begin{pmatrix}
						\bm{\alpha} \\ \bm{\beta}
					\end{pmatrix} = 0,
				\end{equation}
				which is known to be an invertible operator \cite[Theorem~12]{buffa2003galerkin}. Therefore, $\bm{\alpha} \equiv 0$ and $\bm{\beta} \equiv 0$, which concludes the proof.
			\end{proofpart}
		\end{proof}
		
		\begin{remark}
			The assumption of constant coefficients over the boundary $\Gamma$ is essential in two parts of the proof: (1) for obtaining homogeneous right-hand side in \eqref{eq:proofLS} and therefore a divergence free field; (2) to ensure injectivity of the single-trace equation in \eqref{eq:injectiveSTF}. This is similar to what was observed in the Helmholtz transmission problem \cite[Section~3.3]{labarca2023}.
		\end{remark}
		
		\begin{theorem}[Well-Posedness of Problem \ref{prob:varSTF}]\label{th:well-posed} Under Assumptions \ref{assumption} and \ref{ConstCoeffs}, there exists a unique solution 
			\begin{equation*}
				(\bm{\alpha}^{\star}, \bm{\beta}^{\star}, \neu^{\star}, \nev^{\star}) \in \bm{\mathcal{H}}(\Gamma)\times \bm{\mathcal{H}}(\Omega_i)
			\end{equation*}
			to Problem \ref{prob:varSTF}, which satisfies
			\begin{equation*}
				\norm{\bm{\alpha}^{\star}}_{\mathbf{H}^{-1/2}(\curl_{\Gamma}, \Gamma)} +
				\norm{\bm{\beta}^{\star}}_{\mathbf{H}^{-1/2}(\dv_{\Gamma}, \Gamma)} +
				\norm{\neu^{\star}}_{\mathbf{H}(\Curl, \Omega_i)} +
				\norm{\nev^{\star}}_{\mathbf{H}(\Curl, \Omega_i)}  \leq C\norm{\underline{\bm{g}}}_{\bm{\mathcal{H}}(\Gamma)}.
			\end{equation*}
		\end{theorem}
		\begin{proof}
			The proof follows Proposition \ref{th:well_posed} and the framework of Appendix \ref{BlockOps}. In particular, we have
			\begin{itemize}
				\item Compactness results for $\mathbf{\Xi}^e, \mathbf{J}^e$ and $\mathbf{\Xi}^m,\mathbf{J}^m$, from Corollary \ref{compactness}.
				\item $T$-coercivity (or generalized G\r{a}rding inequality) for $$\mathbb{A} \coloneqq  \mathbf{M}^{-1}\mathbf{A}_{\kappa_0}\mathbf{M} + \mathbf{A}_{\kappa_1}^{\tilde{\varepsilon}, \tilde{\mu}},$$ from Corollary \ref{th:Tcoercive-weighted-STF}
				\item $T$-coercivity for $\mathbf{I}-\mathbf{\Lambda}^{\star}$, from Proposition \ref{IminusA}.
				\item Uniqueness of solutions, from Proposition \ref{unique}.
			\end{itemize}
			As the assumptions of Proposition \ref{th:well_posed} hold, we obtain well-posedness of Problem \ref{prob:varSTF}.
		\end{proof}
		\section{Galerkin Discretization}\label{sec:Galerkin}
		\subsection{Finite Element and Boundary Element Spaces}
		Let $\cdef{\{\mathcal{T}_h\}_{h>0}}$ be a globally quasi-uniform and shape-regular family of simplicial meshes of $\Omega_i$ (see \cite[Section~9]{steinbach2007numerical}). Let $\cdef{\{\Sigma_h\}_{h>0}}$ be the induced family of meshes on $\Gamma$: $\Sigma_h = \mathcal{T}_h |_{\Gamma}$. We choose finite element spaces:
		\begin{itemize}
			\item $\cdef{N_h} \coloneqq N_h(\mathcal{T}_h)\subset \mathbf{H}(\Curl, \Omega_i)$ of lowest order N\'ed\'elec edge elements (in the volume) \cite{nedelec1986new},\cite[Section~3]{hiptmair2002finite},\cite{bossavit1988rationale}.
			\item $\cdef{E_h} \coloneqq E_h(\Sigma_h) \subset \mathbf{H}^{-1/2}(\curl_{\Gamma}, \Gamma)$ of lowest order surface edge elements \cite[Section~2.2]{borm2004fast}.
			\item $\cdef{W_h} \coloneqq W_h(\Sigma_h) \subset \mathbf{H}^{-1/2}(\dv_{\Gamma}, \Gamma)$ of lowest order rotated surface edge elements, also known as RWG (Rao-Wilton-Glisson) boundary elements in computational engineering \cite{rao1982electromagnetic}.
		\end{itemize}
		We will denote $\cdef{N_h^{\star}}$ a conforming subspace of the dual space of $\mathbf{H}(\Curl, \Omega_i).$ 
		
		\begin{remark}\label{remark}
			It is important to note that, contrary to what is a standard choice in the literature on volume integral equations \cite{botha2006solving,markkanen2020new}, using $N_h^{\star} = N_h$ does not lead to a stable discretization of the duality pairing. We briefly describe why this is not the case in Appendix \ref{sec:L2-proj}. As it happens with the duality product in the trace space \cite{buffa2007dual}, a good approach might be the use of a dual barycentric finite element complex, i.e. the use of face elements on a dual mesh as a subspace of $\mathbf{H}(\Curl, \Omega_i)'.$ These claims, although intuitive, remain as an open problem. The generalization of a dual barycentric complex has been studied in different contexts \cite{christiansen2008construction}. \\
			
		\end{remark}
		
		\subsection{Asymptotic Quasi-Optimality}\label{sec:quasi-optimal}
		In order to obtain a final result on the discretization of Problem \ref{prob:varSTF}, we need a discrete version of Proposition \ref{IminusA}. As mentioned in Remark \ref{remark}, this is related to a stable discrete duality pairing in $\mathbf{H}(\Curl, \Omega_i).$ \\
		Our goal is to have a conforming discretization of $\mathbf{H}(\Curl, \Omega_i)'$ such that the following holds.
		\begin{assumption}[Discrete inf-sup Condition for $\mathsf{d}$]\label{discreteInfSup} There exists $c_{\mathsf{d}} >0$ such that
			\begin{equation*}
				c_{\mathsf{d}} \leq \inf\limits_{0\neq \neu_h \in N_h}\sup\limits_{0\neq \bm{w}_h\in N^{\star}_h} \dfrac{\langle \bm{w}_h, \neu_h\rangle_{\Omega_i}}{\norm{\neu_h}_{\mathbf{H}(\Curl, \Omega_i)} \norm{\bm{w}_h}_{\mathbf{H}(\Curl, \Omega_i)'}} \quad \text{ for all }h > 0.
			\end{equation*}
		\end{assumption}
		We have to include this assumption in order to arrive at the following main result on the quasi-optimality of Galerkin solutions for \ref{prob:varSTF}:
		\begin{theorem} Provided that Assumptions \ref{ConstCoeffs} and \ref{discreteInfSup} hold, there are $h_0 >0$ and a constant $c_{\mathsf{qo}} > 0$ independent of $h$ such that there exists a unique Galerkin solution $(\bm{\alpha}_h^{\star}, \bm{\beta}_h^{\star}, \neu_h^{\star}, \nev_h^{\star}) \in E_h\times W_h \times N_h \times N_h$ of Problem \ref{prob:varSTF} for all $h< h_0$. The solution satisfies
			\begin{align*}
				\norm{(\bm{\alpha}^{\star}, \bm{\beta}^{\star}, \neu^{\star}, \nev^{\star}) - (\bm{\alpha}_h^{\star}, \bm{\beta}_h^{\star}, \neu_h^{\star}, \nev_h^{\star})} \leq c_{\mathsf{qo}} 
				\inf\limits_{\substack{(\bm{\lambda}_h, \bm{\eta}_h)\in E_h\times W_h, \\ (\bm{w}_h, \bm{q}_h)\in N_h\times N_h }}
				\norm{(\bm{\alpha}^{\star}, \bm{\beta}^{\star}, \neu^{\star}, \nev^{\star}) - (\bm{\lambda}_h, \bm{\eta}_h, \bm{w}_h, \nev_h)}
			\end{align*}
		\end{theorem}
		\begin{proof}
			The proof is based on the result from Propositions \ref{th:inf-sup-t0-weakened} and \ref{th:abstract_qo}. In particular, we need $T_h$-coercivity result (see \cite[Theorem~2]{ciarlet2012}) for the bilinear form
			\begin{align*}
				\mathsf{m}((\bm{\alpha}, \bm{\beta}, \neu, \nev), (\bm{\zeta}, \bm{\xi}, \bm{w}, \bm{q})) \coloneqq{}& \mathsf{a}((\bm{\alpha}, \bm{\beta}), (\bm{\zeta}, \bm{\xi})) + \mathsf{b}((\neu, \nev), (\bm{\zeta}, \bm{\xi}))\\
				+{}&\mathsf{c}((\bm{\alpha}, \bm{\beta}), (\bm{w}, \bm{q})) + \mathsf{d}((\neu, \nev), (\bm{w}, \bm{q}))
			\end{align*} 
			from Problem \ref{prob:varSTF}. According to Proposition \ref{th:inf-sup-t0-weakened} we need to verify the
			\begin{itemize}
				\item $T_h$-coercivity for $\mathsf{a}$. This result follows by noticing that the regular components in the stable regular decomposition from \eqref{eq:decomp-alpha} and \eqref{eq:decomp-beta} are in the domain of local linear interpolation operators \cite[Lemma~16]{buffa2003galerkin}. Therefore, $T$-coercivity translates to $T_h$-coercivity simply by local interpolation \cite[Section~9]{buffa2003galerkin}. 
				\item $T_h$-coercivity for $\mathsf{d}$. This property is supplied by Assumption \ref{discreteInfSup}.
			\end{itemize}
			As $T_h$-coercivity is equivalent to $h$-uniform inf-sup stability (see \cite[Theorem~2]{ciarlet2012}), quasi-optimality follows from $\mathsf{m}$ being $h$-uniform inf-sup stable up to compact perturbations.
		\end{proof}
		
		\subsection{Numerical Experiments}\label{sec:numerics}
		We show numerical experiments to validate our formulation. We compare our results with highly-resolved solution $(\neu^{\star}_h, \nev^{\star}_h)$ obtained from a FEM-BEM coupling, also known as the Johnson-N\'ed\'elec coupling (see \cite{johnson1980coupling}). We study convergence of solutions with respect to the $\mathbf{L}^2-$norm
		\begin{equation}\label{eq:error-norms}
			\text{error}_{\mathbf{L}^2} \coloneqq \dfrac{\norm{\neu^{\star}_h - \neu_h}_{\mathbf{L}^2(\Omega_i)}}{\norm{\neu^{\star}_h}_{\mathbf{L}^2(\Omega_i)}}, \quad \text{error}_{\mathbf{L}^2\times\mathbf{L}^2} \coloneqq \dfrac{\sqrt{\norm{\neu^{\star}_h - \neu_h}_{\mathbf{L}^2(\Omega_i)}^2 + \norm{\nev^{\star}_h - \nev_h}_{\mathbf{L}^2(\Omega_i)}^2  } }{\sqrt{\norm{\neu^{\star}_h}_{\mathbf{L}^2(\Omega_i)}^2 + \norm{\nev^{\star}_h}_{\mathbf{L}^2(\Omega_i)}^2 }}
		\end{equation}
		where $(\neu_h, \nev_h)$ is a Galerkin solution of Problem \ref{prob:varSTF}. In the case of FEM-BEM coupling, we compute $\nev_h^{\star}=\tfrac{1}{i\omega\mu}\Curl(\neu_h^{\star})$.\\
		
		In all of our experiments, we use $N_h(\mathcal{T}_h)$ as a finite element space for the dual of $\mathbf{H}(\Curl, \Omega_i).$ Note that, as mentioned in Remark \ref{remark}, this may not lead to a stable discretization of the duality pairing in $\mathbf{H}(\Curl, \Omega_i).$
		
		The implementation was carried out in C\texttt{++}, by extending the BemTool\footnote{\url{https://github.com/xclaeys/BemTool}} library for BEM computations to the case of VIEs. Numerical integration of singular integrals is computed in terms of a Duffy transformation \cite{feist2023efficient,feist2023fractional,munger2022efficient} and tensorized Gauss quadrature rules. Matrix compression with $\mathcal{H}-$matrices is done with the Castor library \cite{Aussal2022}, a C\texttt{++} header-only library for linear algebra computations. We have made our code available in a Github repository \footnote{\url{https://github.com/ijlabarca/CoupledBVIE}}.
		
		\subsection{Scattering at a dielectric cube}\label{sec:test_cube}
		We study the electromagnetic scattering problem at a unit cube 
		\begin{equation*}
			\Omega_i\coloneqq \{  \nex = (x,y,z)\in \mathbb{R}^3 : 0\leq x,y,z\leq 1 \}.
		\end{equation*}
		Material properties are given by
		\begin{align*}
			\varepsilon(\nex) = \left\{\begin{array}{ll}
				2 + 4xyz(1-x)(1-y)(1-z), &\text{for }\nex \in \Omega_i,\\
				1, &\text{for }\nex \in \mathbb{R}^3 \setminus \overbar{\Omega}_i.
			\end{array}  \right.\\
		\end{align*}
		and $\mu(\nex)\equiv 1$ in $\mathbb{R}^3.$ Note that material properties are constant at the boundary $\Gamma$. \\
		The incident wave is given by
		\begin{equation*}
			\neu^{\text{inc}}(\nex) = \bm{e}_0 \exp(i\kappa_0 \nex \cdot \nex_0),
		\end{equation*}
		where $\kappa_0 = 1, \ \nex_0 = (0, 1, 0)$ and $\bm{e}_0 = (1, 0, 0).$\\
		The meshes used for our computations are described in Table \ref{tab:meshCube}. The reference solution is obtained by FEM-BEM coupling computed on the finest mesh. Convergence results are shown in Figure \ref{fig:test_cube}. We observe $\mathcal{O}(h)$ convergence, which is the best we can expect for this setting, because the approximation spaces merely contain the full space of piecewise-constant functions. Apparently a potential violation of Assumption \ref{discreteInfSup} does not affect convergence in the $\mathbf{L}^2$--norm in this case.
		\begin{figure}[ht!]
			\begin{minipage}[b]{.45\linewidth}
				\centering
				\begin{tabular}{|c|c|c|c|}
					\hline
					\multicolumn{4}{|c|}{\textbf{Meshes}} \\
					\hline
					\textbf{Elements} & \textbf{Nodes}& \textbf{Edges}& \textbf{Mesh size} \\
					\hline
					24 & 14 & 49 & 1/2\\
					\hline
					192 & 63 &302 & 1/4\\
					\hline
					1536 & 365 &2092 & 1/8\\
					\hline
					12288 & 2457 &15512 & 1/16\\
					\hline
					98304 & 17969 & 119344& 1/32\\
					\hline
				\end{tabular}
				\captionof{table}{Meshes used in Section \ref{sec:test_cube}, generated by uniform regular refinement 
					{\label{tab:meshCube}}}
			\end{minipage}\hfill
			\begin{minipage}[b]{.4\linewidth}
				\centering
				\includegraphics[width=0.7\linewidth]{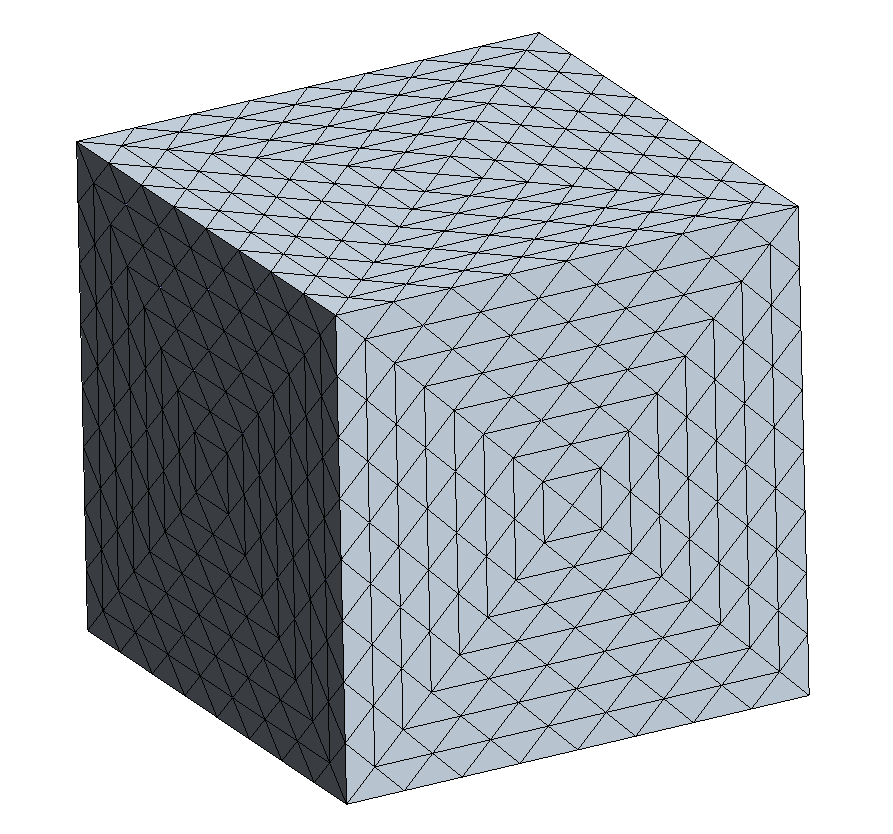}	
				\label{fig:meshCube}
				\captionof{figure}{Mesh with 12288 elements.}
			\end{minipage}\hfill
		\end{figure}
		\FloatBarrier

		\begin{figure}[ht!]
			\centering 
			\includegraphics[width=0.5\linewidth]{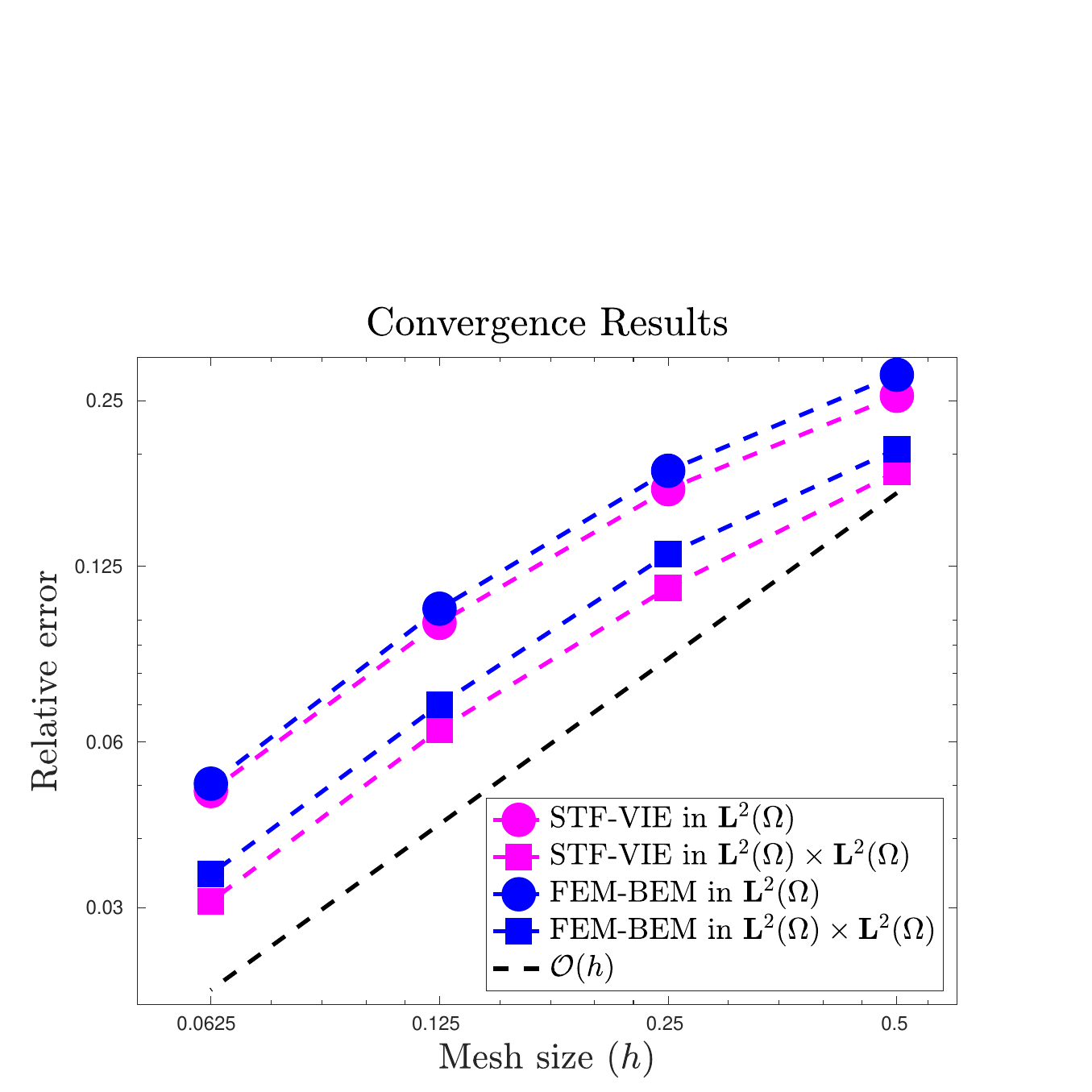}	
			\caption{Scattering at a cube: problem of Section \ref{sec:test_cube}. Error norms \eqref{eq:error-norms} as functions of $ h $.}
			\label{fig:test_cube}
		\end{figure}
		\FloatBarrier

		\subsection{Scattering at a tetrahedron}\label{sec:test_tetra}
		Now we study the problem with $\Omega$ being the tetrahedron with vertices $(0,0,0), (1, 0, 0), (0, 1, 0), (0, 0, 1).$
		
		\begin{figure}
			\begin{minipage}[b]{.45\linewidth}
				\centering
				\begin{tabular}{|c|c|c|l|}
					\hline
					\multicolumn{4}{|c|}{\textbf{Meshes}} \\
					\hline
					\textbf{Elements} & \textbf{Nodes}& \textbf{Edges}& \textbf{Mesh size} \\
					\hline
					4 & 7 & 15 & 0.346681\\
					\hline
					32 & 22 &73 & 0.173340\\
					\hline
					256 & 95 &430 & 0.0866702\\
					\hline
					2048 & 525 &2892 & 0.0433351\\
					\hline
					16384 & 3417 & 21080& 0.0216676\\
					\hline
				\end{tabular}   
				\captionof{table}{Meshes used in Section \ref{sec:test_tetra}, generated by uniform regular refinement.}
			\end{minipage}\hfill
			\begin{minipage}[b]{.4\linewidth}
				\centering
				\includegraphics[width=0.7\linewidth]{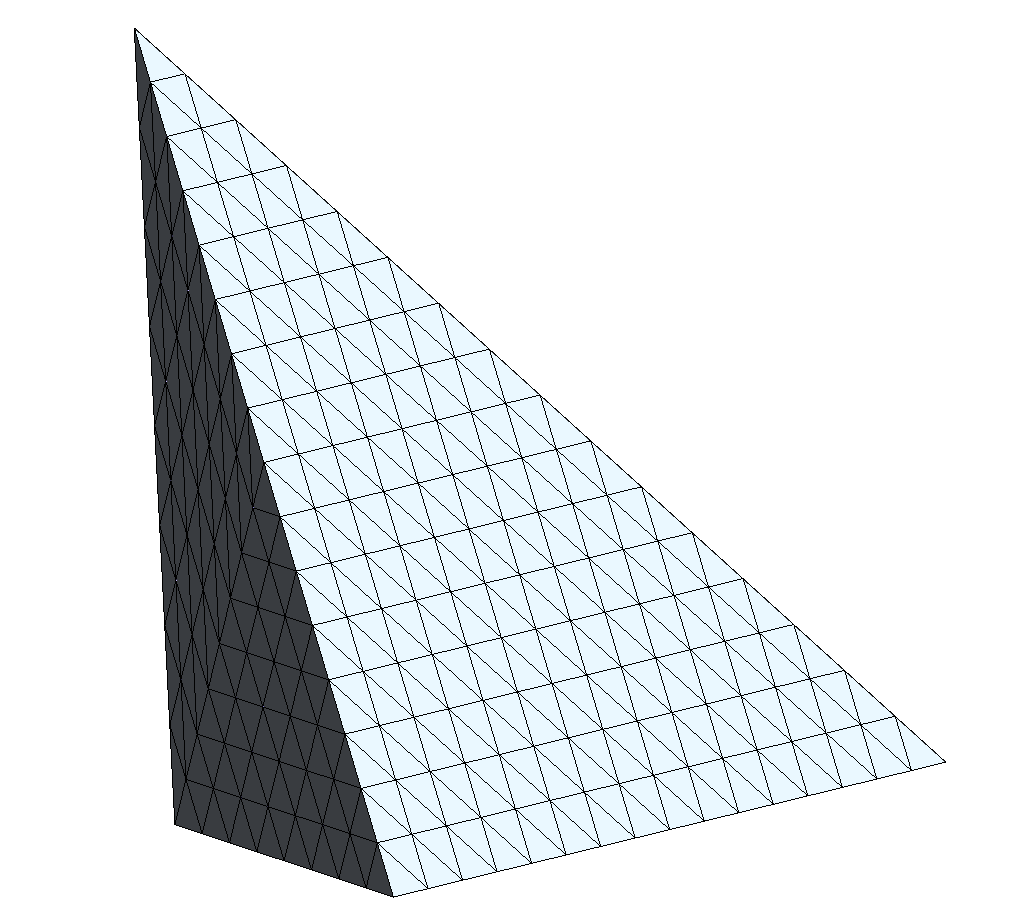}	
				\label{fig:meshTetra}
				\captionof{figure}{Mesh with 2048 elements. }
			\end{minipage}\hfill
		\end{figure}
		Material properties are given by
		\begin{align*}
			\varepsilon(\nex) = \left\{\begin{array}{ll}
				2 + 4xyz(1-x)(1-y)(1-z), &\text{for }\nex \in \Omega_i,\\
				1, &\text{for }\nex \in \mathbb{R}^3 \setminus \overbar{\Omega}_i.
			\end{array}  \right.\\
			\mu(\nex) = \left\{\begin{array}{ll}
				2 + 4xyz(1-x)(1-y)(1-z), &\text{for }\nex \in \Omega_i,\\
				1, &\text{for }\nex \in \mathbb{R}^3 \setminus \overbar{\Omega}_i.
			\end{array}  \right.\\
		\end{align*}
		Note that in this case, material properties are not homogeneous over the whole bondary $\Gamma.$ Convergence results are shown in Figure \ref{fig:test_tetra}. Again, we observe $\mathcal{O}(h)$ convergence, although this case does not satisfy the assumptions of Proposition \ref{unique}, nor Assumption \ref{discreteInfSup}.
		
		\begin{figure}[ht!] 
			\centering
			\includegraphics[width=0.5\linewidth]{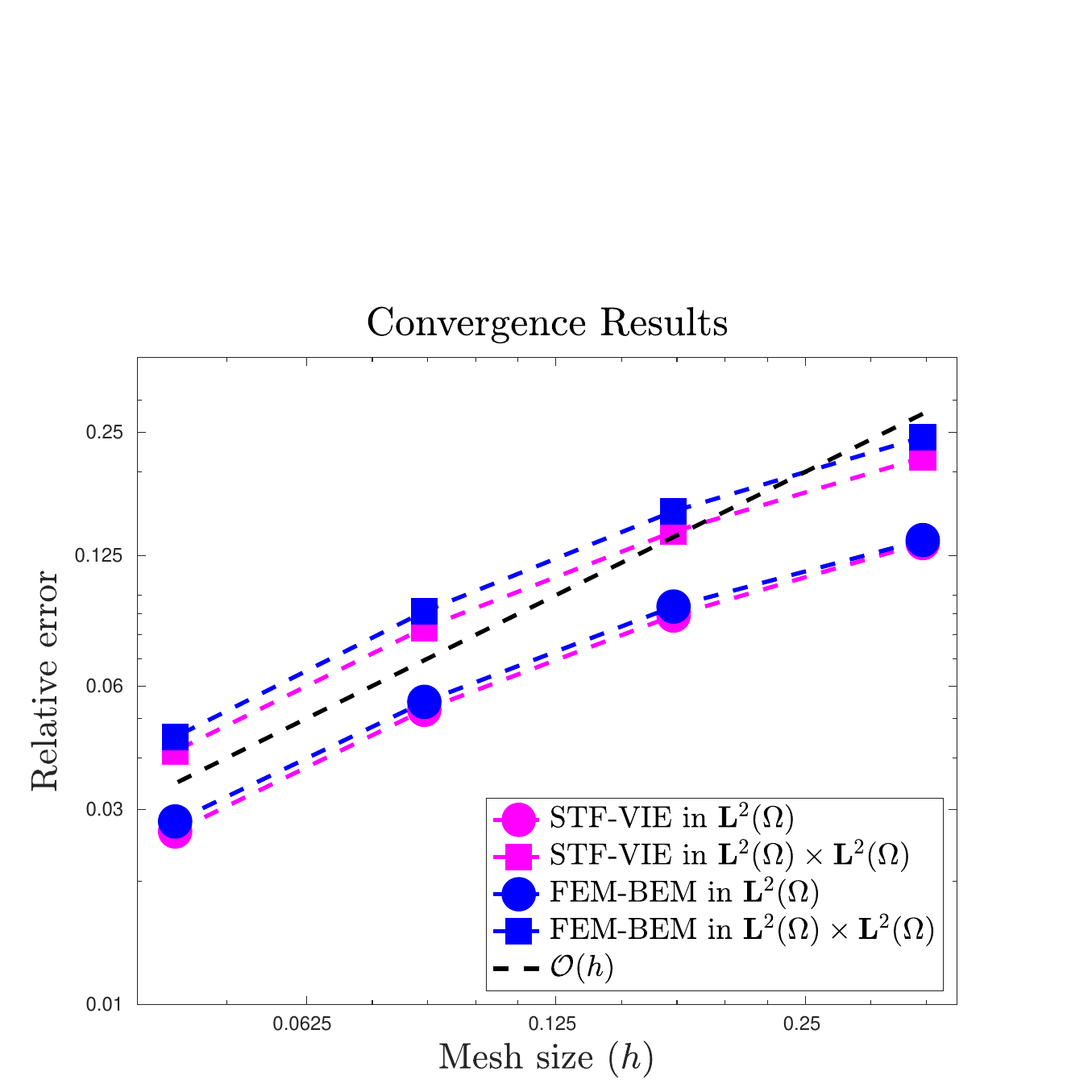}	
			\caption{Scattering at a Tetrahedron: problem of Section \ref{sec:test_tetra}. Error norms \eqref{eq:error-norms} as functions of $ h $.}
			\label{fig:test_tetra}
		\end{figure}
		
		\FloatBarrier

		\il{
		\subsection{Scattering at a dielectric Fichera cube}\label{sec:test_fichera}
		We study the electromagnetic scattering problem at a unit cube 
		\begin{equation*}
			\Omega_i\coloneqq [-1, 1]^3 \setminus [0, 1]^3.
		\end{equation*}
		Material properties are given by
		\begin{align*}
			\varepsilon(\nex) = \left\{\begin{array}{ll}
				2 + 4xyz(1-x)(1-y)(1-z), &\text{for }\nex \in \Omega_i,\\
				1, &\text{for }\nex \in \mathbb{R}^3 \setminus \overbar{\Omega}_i.
			\end{array}  \right.\\
		\end{align*}
		and $\mu(\nex)\equiv 1$ in $\mathbb{R}^3.$ Also in this case the material properties are not constant at the boundary $\Gamma$. \\
		Convergence results are shown in Figure \ref{fig:test_fichera}. We observe a reduced order of convergence due to the singular behavior of Maxwell solutions in this particular geometry \cite{costabel1999singularities}.
		\begin{figure}[ht!]
			\begin{minipage}[b]{.45\linewidth}
				\centering
				\begin{tabular}{|c|c|c|c|}
					\hline
					\multicolumn{4}{|c|}{\textbf{Meshes}} \\
					\hline
					\textbf{Elements} & \textbf{Nodes}& \textbf{Edges}& \textbf{Mesh size} \\
					\hline
					130 & 51 & 228 & 1/2\\
					\hline
					569 & 189 &925 & 1/3\\
					\hline
					1040 & 279 &1510 & 1/4\\
					\hline
					8320 & 1789 &10876 & 1/8\\
					\hline
					66560 & 12665 & 82296& 1/16\\
					\hline
				\end{tabular}
				\captionof{table}{Meshes used in Section \ref{sec:test_fichera}, generated by uniform regular refinement 
					{\label{tab:meshFichera}}}
			\end{minipage}\hfill
			\begin{minipage}[b]{.4\linewidth}
				\centering
				\includegraphics[width=0.7\linewidth]{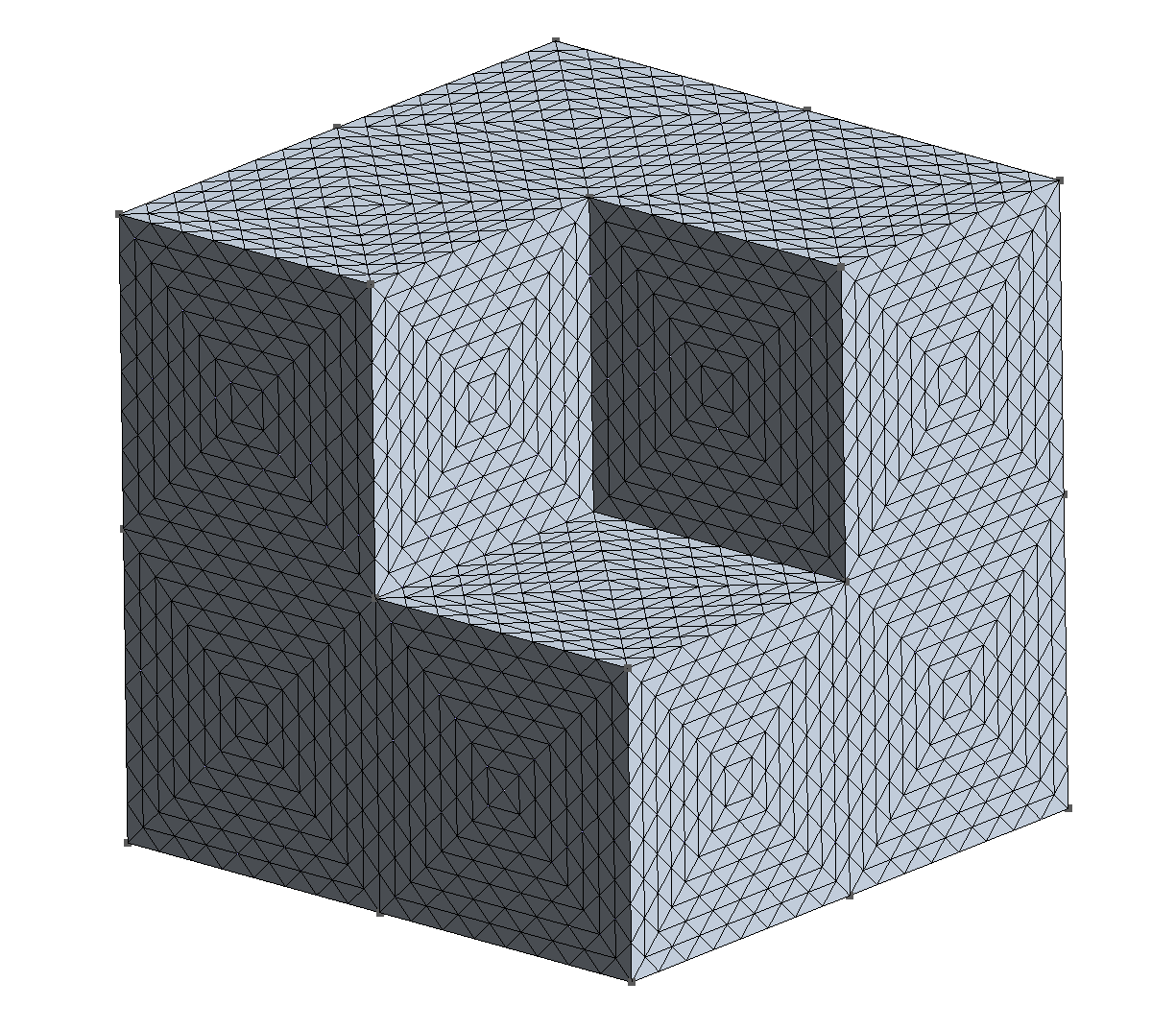}	
				\label{fig:meshFichera}
				\captionof{figure}{Mesh with 66560 elements.}
			\end{minipage}\hfill
		\end{figure}
		\FloatBarrier

		\begin{figure}[ht!]
			\centering 
			\includegraphics[width=0.5\linewidth]{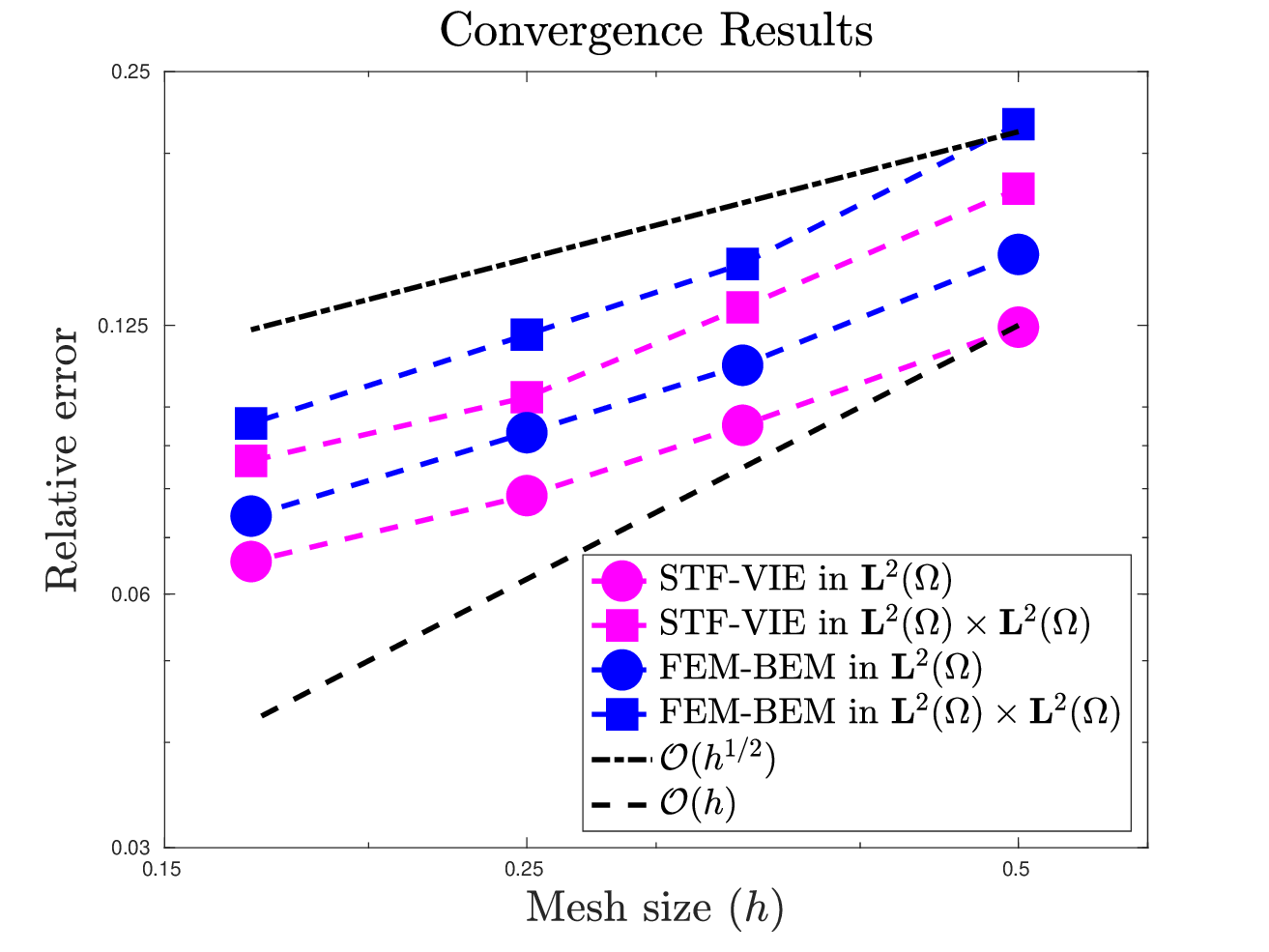}	
			\caption{Scattering at a Fichera cube: problem of Section \ref{sec:test_fichera}. Error norms \eqref{eq:error-norms} as functions of $ h $.}
			\label{fig:test_fichera}
		\end{figure}
		\FloatBarrier

		\subsection{Scattering at a dielectric cube: Empirical study of impact of frequency}\label{sec:test_freq}
		We study the electromagnetic scattering problem at a unit cube 
		\begin{equation*}
			\Omega_i\coloneqq \{  \nex = (x,y,z)\in \mathbb{R}^3 : 0\leq x,y,z\leq 1 \}.
		\end{equation*}
		Material properties are given by
		\begin{align*}
			\varepsilon(\nex) = \left\{\begin{array}{ll}
				2 + 2xyz(1-x)(1-y)(1-z), &\text{for }\nex \in \Omega_i,\\
				1, &\text{for }\nex \in \mathbb{R}^3 \setminus \overbar{\Omega}_i.
			\end{array}  \right.\\
		\end{align*}
		and
		\begin{align*}
	\mu(\nex) = \left\{\begin{array}{ll}
		2, &\text{for }\nex \in \Omega_i,\\
		1, &\text{for }\nex \in \mathbb{R}^3 \setminus \overbar{\Omega}_i.
	\end{array}  \right.\\
\end{align*}	
We study several frequency-domain problems for different values of $\omega > 0$, such that 
	
	\begin{align*}
		\kappa(\nex) = \left\{\begin{array}{ll}
			\omega\sqrt{\varepsilon(\nex)\mu(\nex)}, &\text{for }\nex \in \Omega_i,\\
			\omega\sqrt{\varepsilon_0\mu_0}, &\text{for }\nex \in \mathbb{R}^3 \setminus \overbar{\Omega}_i.
		\end{array}  \right.\\
	\end{align*}
	We define a weighted-norm in $\mathbf{L}^2(\Omega_i)\times\mathbf{L}^2(\Omega_i)$ given by 
	\begin{equation}
		\norm{(\neu, \nev)}_{\mathbf{H}_{\kappa_0}(\Omega_i)} ^2 \coloneqq \norm{\neu}^2_{\mathbf{L}^2(\Omega_i)} + \dfrac{1}{\kappa_0^2}\norm{\nev}^2_{\mathbf{L}^2(\Omega_i)} 
	\end{equation}
	for all $(\neu,\nev)\in \mathbf{L}^2(\Omega_i)\times \mathbf{L}^2(\Omega_i)$. Then, we measure relative errors as
	\begin{equation}\label{eq:error-norms-kappa}
		 \text{error}_{\mathbf{H}_{\kappa_0}} \coloneqq \dfrac{\sqrt{\norm{\neu^{\star}_h - \neu_h}_{\mathbf{L}^2(\Omega_i)}^2 + \frac{1}{\kappa_0^2}\norm{\nev^{\star}_h - \nev_h}_{\mathbf{L}^2(\Omega_i)}^2  } }{\sqrt{\norm{\neu^{\star}_h}_{\mathbf{L}^2(\Omega_i)}^2 +  \frac{1}{\kappa_0^2}\norm{\nev^{\star}_h}_{\mathbf{L}^2(\Omega_i)}^2 }}
	\end{equation}

	to compare results for different frequencies.\\
	For each problem, we use a uniform mesh such that the product $\kappa_0 h \approx 0.9$. As expected for the FEM-BEM coupling, the number of degrees of freedom needs to increase at a higher rate in order to keep the same level of accuracy, due to the pollution effect at higher frequencies. We observe that STF-VIEs have a better behavior in this context, suggesting better frequency-dependent constants for the inverse operator.
	
	The analysis of such behavior remains open for volume and boundary-volume integral equations, although existing results for the combined field integral equation \cite{galkowski2023does} point to the absense of pollution effect in certain integral equation formulations.
	
		\begin{figure}[h!]
				\centering
				\includegraphics[width=0.55\linewidth]{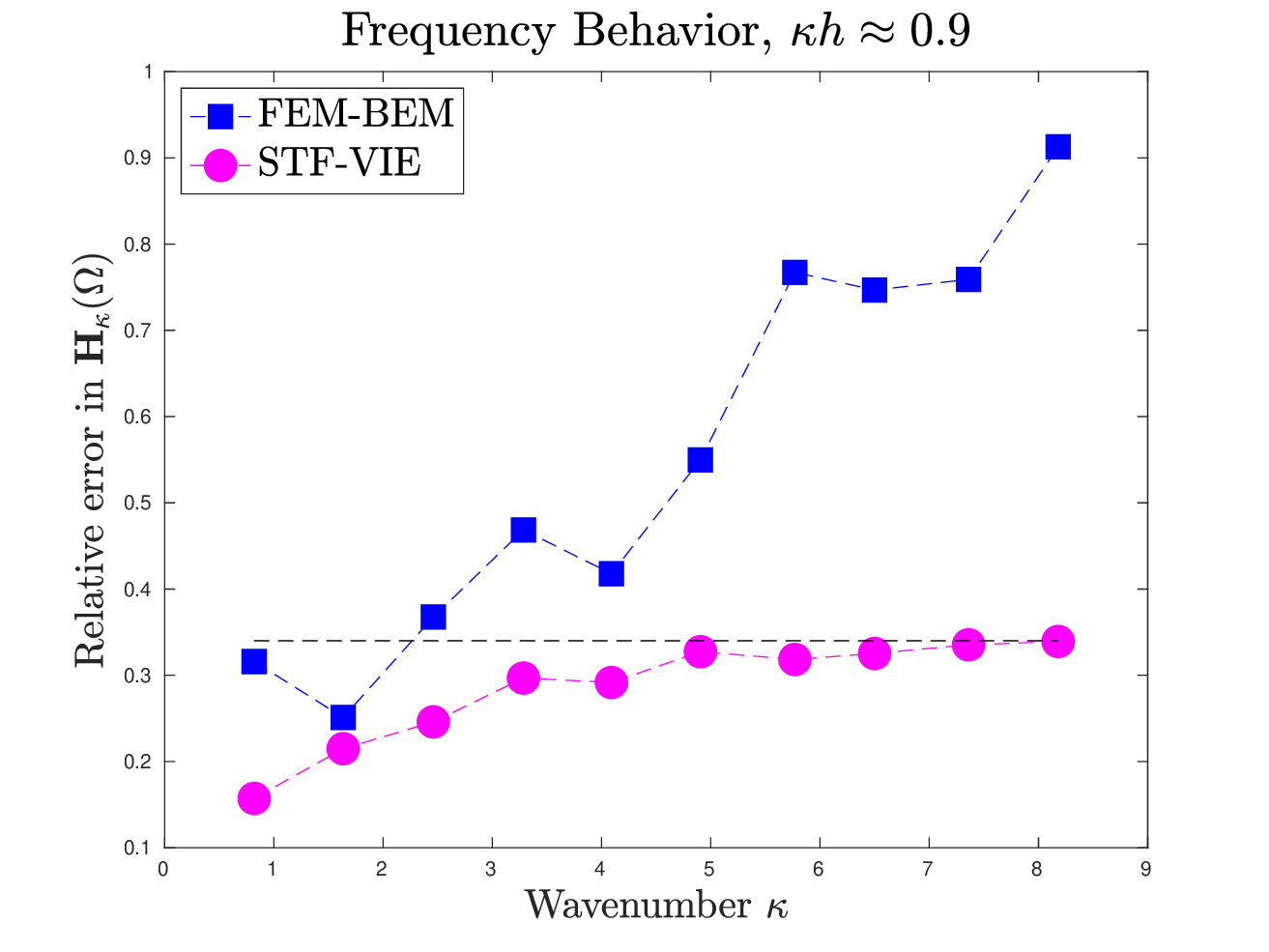}	
				\label{fig:maxwell-freq}
				\captionof{figure}{Wavenumber comparison between FEM-BEM and STF-VIEs in Section \ref{sec:test_fichera}. Errors are computed as in \eqref{eq:error-norms-kappa}.}
		\end{figure}
		\FloatBarrier
		
		}
		
		\section{Conclusion}
		We presented a new formulation coupling boundary and volume integral equations. Under assumptions on the material properties, we are able to show well-posedness of continuous and discrete settings. Uniqueness of solutions in a general setting remains an open problem.
		Our numerical experiments show optimal convergence of Galerkin discretizations \il{in the absense of geometric singularities.} The use of a conforming subspace of the dual of $\mathbf{H}(\Curl, \Omega_i)$ that ensures a stable discretization remains as an open problem. \\
		
		\il{Whether STF-VIE offers a competitive alternative to the $h$-version of FEM-BEM coupling depends on
			\begin{enumerate}[label=(\roman*)]
				\item a favorable setting with the inhomogeneities \emph{supported in a small part} of $\Omega_i$,
				\item the frequency,
				\item an efficient implementation harnessing matrix compression.
			\end{enumerate}
			Concerning (ii) we have some evidence that STF-VIEs are \emph{immune to the notorious pollution
				effect} haunting $h$-FEM for wave propagation problems for higher frequencies. Yet, studying
			the dependence of the discretization error of the new STF-VIEs on frequency is a
			challenging research topic beyond the scope of this paper.
			Concerning (iii) in this work, we employed ACA based $\mathcal{H}$–matrix compression, although the
			use of $\mathcal{H}^2$–matrix compression is even more necessary for VIEs to become relevant. In
			high-frequency situations, \emph{directional $\mathcal{H}^2$–matrix compression} \cite{bebendorf2015wideband,borm2017directional} offers a promising path
			towards efficient implementation.
			.\\ }

		\section*{Funding}
		This work was supported by the Swiss National Science Foundation under grant SNF200021 184848/1 “Novel BEM for Electromagnetics”.



\bibliographystyle{elsarticle-num} 
\bibliography{References}


%
%
%

\appendix

\section{Block Operators}\label{BlockOps}
In this section we present results from \cite[Appendix~A]{labarca2023} that cover a particular case of block operators. We show what is required to obtain inf-sup conditions in the continuous and discrete setting. The theoretical results from this appendix are used to establish well-posedness of the variational STF-VIE problem in Sections \ref{sec:analysis} and \ref{sec:Galerkin}.
\subsection{Fredholm Equation}
Let $ X, \Pi $ be Hilbert spaces and $ X', \Pi' $ their duals. Consider the operators
\begin{align*}
	&A : X \rightarrow X', &B : \Pi \rightarrow X', \\
	&C : X \rightarrow \Pi', &D : \Pi \rightarrow \Pi ',
\end{align*}
all of them bounded linear operators. We study the block operator equation
\begin{align}\label{eq:abstract_system}
	\begin{pmatrix}
		A & B\\ C & D
	\end{pmatrix}\begin{pmatrix}
		u \\ p  
	\end{pmatrix} = \begin{pmatrix}
		f \\ 0
	\end{pmatrix}, \ f\in X'.
\end{align}
\begin{assumption}\label{th:assumption}
	The operator 
	$$ \mathbf{T} = \begin{pmatrix}
		A & B\\ C & D
	\end{pmatrix} : X\times \Pi \rightarrow X' \times \Pi'$$
	is injective. Moreover, $ A $ and $ D $ are coercive operators. $ B $ is a compact operator.
\end{assumption}
\begin{proposition}[{\cite[Proposition~A.2]{labarca2023}}]\label{th:well_posed}
	Under assumption \ref{th:assumption}, there exists a unique solution $ (u^{\star}, p^{\star}) \in X\times \Pi $ to the system in \eqref{eq:abstract_system}. Moreover, the solution satisfies
	\begin{equation*}
		\norm{u^{\star}}_X + \norm{p^{\star}}_{\Pi} \leq C \norm{f}_{X'}.
	\end{equation*}
\end{proposition}

\subsection{Galerkin Discretization}
Next, we consider the Galerkin discretization of \eqref{eq:abstract_system}. Choose finite dimensional subspaces $ X_h\subset X$ and $\Pi_h \subset \Pi. $ We study the following variational problem: find $ (u_h, p_h)\in X_h \times \Pi_h $ such that 
\begin{align*}
	\langle Au_h, v_h \rangle_X + \langle Bp_h, v_h\rangle_{X} &= \langle f, v_h \rangle, &\text{for all }v_h\in X_h,\\
	\langle Cu_h, q_h \rangle_{\Pi} + \langle Dp_h, q_h\rangle_{\Pi} &= 0, &\text{for all }q_h\in \Pi_h,
\end{align*}
which can be rewritten as
\begin{equation}\label{eq:discrete_galerkin}
	\mathsf{t}\left((u_h, p_h), (v_h, q_h)\right) = \langle f, v_h \rangle, \quad \text{for all }v_h\in X_h, \ q_h \in \Pi_h,
\end{equation}
where 
\begin{equation}
	\mathsf{t}\left((u_h, p_h), (v_h, q_h)\right) = \langle Au_h, v_h \rangle_X + \langle Bp_h, v_h\rangle_{X} + \langle Cu_h, q_h \rangle_{\Pi} + \langle Dp_h, q_h\rangle_{\Pi}.
\end{equation}
\begin{proposition}[Inf-sup condition, {\cite[Proposition~A.3]{labarca2023}}]\label{th:inf-sup-t0}
	Let $A_0 : X\rightarrow X'$ and $D_0:\Pi\rightarrow\Pi'$ be elliptic operators, and $C : X\rightarrow \Pi'$ a bounded operator. The bilinear form $ \mathsf{t}_0: (X\times \Pi) \times (X\times \Pi)\rightarrow \mathbb{C} $ given by
	\begin{equation*}
		\mathsf{t}_0((u, p), (v, q)) = \langle A_0u, v \rangle_X + \langle Cu, q \rangle_{\Pi} + \langle D_0p, q\rangle_{\Pi},
	\end{equation*}
	satisfies the $ h-$uniform discrete inf-sup condition
	\begin{equation}
		c_1^{\mathsf{t}_0} \leq 	\inf\limits_{\substack{0\neq(u_h, p_h) \in X_h\times \Pi_h}}\sup\limits_{\substack{0\neq(v_h, q_h) \in X_h\times \Pi_h}} \dfrac{\mathrm{Re} \{\mathsf{t}_0((u_h, p_h), (v_h, q_h))\}}{\norm{(u_h, p_h)}_{X\times \Pi} \norm{(v_h, q_h)}_{X\times \Pi}}, \quad \text{ for all }h>0.
	\end{equation}
\end{proposition}
For the sake of simplicity we have stated Proposition \ref{th:inf-sup-t0} assuming elliptic operators $ A_0 $ and $ D_0. $ However, in Section \ref{sec:coercive} we face the situation that $ D_0 $ merely satisfies an inf-sup condition. This case is addressed by the following extended version of Proposition \ref{th:inf-sup-t0}.

\begin{proposition}[inf-sup condition, weakened assumptions, {\cite[Proposition~A.4]{labarca2023}}]\label{th:inf-sup-t0-weakened}
	In the setting of Section \ref{BlockOps}, let $ \widetilde{\Pi} $ be another Hilbert space and $ \widetilde{\Pi}_h \subset \widetilde{\Pi}$ a finite dimensional subspace. Let $ C: X \rightarrow \widetilde{\Pi}' $ be bounded and let $ D_0: \Pi \rightarrow \widetilde{\Pi}' $ be a bounded operator that satisfies an $ h-$uniform discrete inf-sup condition
	\begin{equation}
		c_1^{\mathsf{d}_0} \leq 	\inf\limits_{\substack{0\neq p_h \in \Pi_h}}\sup\limits_{\substack{0\neq q_h \in \widetilde{\Pi}_h}} \dfrac{\mathrm{Re} \{\langle D_0p_h, q_h\rangle_{\widetilde{\Pi}}\}}{\norm{p_h}_{\Pi} \norm{q_h}_{\widetilde{\Pi}}} \quad \text{ for all }h>0.
	\end{equation}
	Then, the bilinear form $ \mathsf{t}_0: (X\times \Pi) \times (X\times \widetilde{\Pi})\rightarrow \mathbb{C} $ given by
	\begin{equation*}
		\mathsf{t}_0((u, p), (v, q)) = \langle A_0u, v \rangle_X + \langle Cu, q \rangle_{\widetilde{\Pi}} + \langle D_0p, q\rangle_{\widetilde{\Pi}}
	\end{equation*}
	satisfies the $ h-$uniform discrete inf-sup condition
	\begin{equation}
		c_1^{\mathsf{t}_0} \leq 	\inf\limits_{\substack{0\neq(u_h, p_h) \in X_h\times \Pi_h}}\sup\limits_{\substack{0\neq(v_h, q_h) \in X_h\times \widetilde{\Pi}_h}} \dfrac{\mathrm{Re} \{\mathsf{t}_0((u_h, p_h), (v_h, q_h))\}}{\norm{(u_h, p_h)}_{X\times \Pi} \norm{(v_h, q_h)}_{X\times \widetilde{\Pi}}}, \quad \text{ for all }h>0.
	\end{equation}
\end{proposition}
%

\begin{proposition}[{\cite[Proposition~A.2.6]{labarca2024}}]\label{th:t-inf-sup-coercive}
	Let $V$ and $W$ be Hilbert spaces, $\{V_h\}_{h>0}$ and $\{W_h\}_{h>0}$ asymptotically dense families of finite dimensional subspaces of $V$ and $H$ respectively. Consider a bounded sesquilinear form $\mathsf{t} : V \times W\rightarrow \mathbb{C}$ such that $\mathsf{t}= \mathsf{t}_0 + \mathsf{t}_{\mathbf{K}}.$ We assume the following 
	\begin{enumerate}
		\item The operator $\mathbf{A}: V\rightarrow W'$ induced by the sesquilinear form $\mathsf{t}$ is injective.
		\item The operator $\mathbf{K}: V\rightarrow W'$ induced by the sesquilinear form $\mathsf{t}_{\mathbf{K}}$ is compact.
		\item The sesquilinear form $\mathsf{t}_0$ satisfies an inf-sup condition on $V\times W$.
		\item The sesquilinear form $\mathsf{t}_0$ satisfies an $h$--uniform discrete inf-sup condition on $V_h\times W_h$.
	\end{enumerate}
	Then, there exist $h_0 >0$ and $c_{\mathsf{t}} > 0$ such that
	\begin{equation}
		0< c_{\mathsf{t}} \leq \inf\limits_{0\neq v_h\in V_h}\sup\limits_{0\neq w_h\in W_h}\dfrac{\mathrm{Re}\{\mathsf{t}(v_h, w_h)\}}{\norm{v_h}_{V}\norm{w_h}_W}, \quad \text{ for all }h< h_0.
	\end{equation}
\end{proposition}
\begin{proof}
	We recall that $h$--uniform inf-sup conditions are equivalent to $T_h$--coercivity (see \cite[Theorem~2]{ciarlet2012}): let $\{T_h\}_{h>0}$ be the family of bounded linear operators $T_h : V_h \rightarrow W_h$ such that
	\begin{equation}
		\norm{T_h} \leq C \quad \text{ for all }h>0,
	\end{equation}
	and 
	\begin{equation}
		\text{Re}	\{\mathsf{t}_0(v_h, T_h v_h)\} \geq c_{\mathsf{t}_0}' \norm{v_h}^2_{V} \quad \text{ for all }v_h\in V_h,
	\end{equation}
	where $ c_{\mathsf{t}_0}' > 0$ is independent of $h$.\\
	
	We define an operator $\mathbf{X}: V_h\rightarrow W$ such that given $v_h\in V_h$,  
	\begin{equation}\label{eq:wh0}
		\mathsf{t}(q, \mathbf{X}v_h) = -\mathsf{t}_{\mathbf{K}}(q, T_hv_h) \quad \text{for all }q\in V,
	\end{equation}
	which means that $\mathbf{X} = -(\mathbf{A}')^{-1} \mathbf{K}'T_h$. This operator is well defined since $\mathbf{A}$ is invertible due to Fredholm alternative and injectivity. Moreover, $\mathbf{X}$ is a compact operator, since $\mathbf{K}$ is compact.
	We choose conveniently
	\begin{equation}\label{eq:inf-sup-candidates0-full}
		w^{\star}_h= 
		T_h v_h + P_h\mathbf{X}v_h, 
	\end{equation}
	where $P_h: W\rightarrow W_h$ is the $W$--orthogonal projection.
	Now, we compute
	\begin{align}\label{eq:t-inf-sup-compact}
		\begin{aligned}
			\mathsf{t}(v_h, w_h^{\star}) &= \mathsf{t}(v_h, T_hv_h) + \mathsf{t}(v_h, P_h\mathbf{X}v_h) \\
			&=\mathsf{t}(v_h, T_hv_h) + \mathsf{t}(v_h, \mathbf{X}v_h) + \mathsf{t}(v_h, (P_h-\mathsf{Id})\mathbf{X}v_h)\\
			&=\mathsf{t}(v_h, T_hv_h) - \mathsf{t}_{\mathbf{K}}(v_h, T_hv_h) + \mathsf{t}(v_h, (P_h-\mathsf{Id})\mathbf{X}v_h)\\
			&=\mathsf{t}_0(v_h, T_hv_h) + \mathsf{t}(v_h, (P_h-\mathsf{Id})\mathbf{X}v_h)\\
		\end{aligned}
	\end{align}
	From \eqref{eq:t-inf-sup-compact} we obtain
	\begin{equation}\begin{aligned}
			|\mathsf{t}(v_h, w_h^{\star})|&\geq |\mathsf{t}_0(v_h, T_hv_h)| - |\mathsf{t}(v_h,(P_h-\mathsf{Id})\mathbf{X}v_h)|\\
			&\geq c_{\mathsf{t}_0}\norm{v_h}_V^2 - \norm{\mathbf{A}}\norm{v_h}_V^2 \norm{(P_h-\mathsf{Id})\mathbf{X}},
		\end{aligned}
	\end{equation}
	where $ \norm{(P_h-\mathsf{Id})\mathbf{X}}\rightarrow 0$ uniformly as $h\rightarrow 0,$ due to $\mathbf{X}$ being a compact operator. Therefore, there exists $h_0>0$ such that
	\begin{equation}
		|\mathsf{t}(v_h, w_h^{\star})| = |\{\mathsf{t}(v_h, (T_h+P_h\mathbf{X})v_h )|\geq \tfrac{1}{2}c_{\mathsf{t}_0}\norm{v_h}_V^2
	\end{equation}
	This corresponds to $T_h$--coercivity with a family of operators $\{\tilde{T}_h\}_{h<h_0}$, where
	$$\tilde{T}_h \coloneqq T_h + P_h\mathbf{X}, \quad \norm{\tilde{T}_h}\leq \norm{T_h} + \norm{P_h}\norm{\mathbf{X}} \leq C',$$
	with $C'>0$ independent of $h.$
	This result is equivalent to an $h$--uniform inf-sup condition for $\mathsf{t}$, for all $h< h_0$ (see \cite[Theorem~2]{ciarlet2012}).
\end{proof}

\begin{proposition}[Asymptotic quasi-optimality,{\cite[Proposition~A.5]{labarca2023}}]\label{th:abstract_qo}
	Provided that Assumption \ref{th:assumption} holds, there is $ h_0 >0$ and a constant $ c_{\mathsf{qo}}>0 $ independent of $ h $ such that there exists a unique Galerkin solution $ (u_h, p_h) \in X_h \times \Pi_h$ of \eqref{eq:discrete_galerkin} for all $ h < h_0 $. The solution satisfies
	\begin{equation}
		\norm{(u, p) - (u_h, p_h)}_{X\times\Pi} \leq c_{\mathsf{qo}}  \inf\limits_{(\eta_h, \tau_h)\in X_h\times\Pi_h} \norm{(u, p) - (\eta_h, \tau_h)}_{X\times\Pi}.
	\end{equation}
\end{proposition}
%
\section{Equivalent norms in trace spaces}\label{sec:norm-eq}

The coercivity results in Section \ref{sec:coercive} depend on a norm equivalence in trace spaces. This will be important for the subsequent analysis.
\begin{lemma}\label{lemma:eq_norms}
	Let $ \chi \in C^1(\overbar{\Omega}_i) $ be such that $$ 0 < \chi_{\text{min}} < \chi(\nex) < \chi_{\text{max}} $$ for all $ \nex\in \overbar{\Omega}_i$. Then 
	\begin{equation}
		c_{1, \chi}\norm{\varphi}_{H^{1/2}(\Gamma)} \leq \norm{\chi^{1/2}\varphi}_{H^{1/2}(\Gamma)} \leq c_{2, \chi}\norm{\varphi}_{H^{1/2}(\Gamma)}, \quad \text{ for all }\varphi\in H^{1/2}(\Gamma),
	\end{equation}
	with constants $c_{1,\chi}, c_{2,\chi}$ depending on $\chi$ and $\Gamma.$\\
	By duality, the result also holds for $ H^{-1/2}(\Gamma) $
	\begin{equation}
		c_{1, \chi}\norm{\psi}_{H^{-1/2}(\Gamma)} \leq \norm{\chi^{1/2}\psi}_{H^{-1/2}(\Gamma)} \leq c_{2, \chi}\norm{\psi}_{H^{-1/2}(\Gamma)}, \quad \text{ for all }\psi\in H^{-1/2}(\Gamma).
	\end{equation}
	The result also extends component-wise to the vectorial case, to $\mathbf{H}^{1/2}(\Gamma)$ and its dual $\mathbf{H}^{-1/2}(\Gamma).$
\end{lemma}
\begin{proof}
	We start by recalling that for any $\varphi\in H^{1/2}(\Gamma), $ we can write \cite[Section~2.5]{steinbach2007numerical}
	\begin{equation}\label{eq:sobolev-slobodeki}
		\norm{\varphi}_{H^{1/2}(\Gamma)}^2 = \norm{\varphi}^2_{L^2(\Gamma)} + \displaystyle \int\limits_{\Gamma}\displaystyle\int\limits_{\Gamma}\dfrac{|\varphi(\nex) - \varphi(\ney)|^2}{|\nex-\ney|^{3}}\dsy\dsx. 
	\end{equation}
	Then, we compute
	\begin{align*}
		\norm{\chi^{1/2}\varphi}_{H^{1/2}(\Gamma)}^2 &= \norm{\chi^{1/2}\varphi}_{L^2(\Gamma)}^2 +   \displaystyle \int\limits_{\Gamma}\displaystyle\int\limits_{\Gamma}\dfrac{|\chi^{1/2}(\nex)\varphi(\nex) - \chi^{1/2}(\ney)\varphi(\ney)|^2}{|\nex-\ney|^{3}}\dsy\dsx.
	\end{align*}
	We denote 
	\begin{equation}\label{eq:Ivarphi}
		I_{\varphi}^{(1)} \coloneqq \displaystyle \int\limits_{\Gamma}\displaystyle\int\limits_{\Gamma}\dfrac{|\chi^{1/2}(\nex)\varphi(\nex) - \chi^{1/2}(\ney)\varphi(\ney)|^2}{|\nex-\ney|^{3}}\dsy\dsx.
	\end{equation}
	By adding zero,
	\begin{align}\label{eq:AddZero}\begin{aligned}
			|\chi^{1/2}(\nex)\varphi(\nex) - \chi^{1/2}(\ney)\varphi(\ney)| &= |\chi^{1/2}(\nex)\varphi(\nex) - \chi^{1/2}(\ney)\varphi(\nex) \\
			&{}+\chi^{1/2}(\ney)\varphi(\nex) - \chi^{1/2}(\ney)\varphi(\ney)|\\
			&\leq |\chi^{1/2}(\ney)| |\varphi(\nex) - \varphi(\ney)| \\
			&{}+ |\varphi(\nex)| |\chi^{1/2}(\nex) - \chi^{1/2}(\ney)|
		\end{aligned}
	\end{align}
	Since $\chi\in C^1(\overbar{\Omega}_i)$ and positively bounded from below, we know that $\chi^{1/2}\in C^1(\overbar{\Omega}_i).$ Therefore, since $\Omega$ is bounded,
	\begin{equation}\label{eq:LipschitzBound}
		|\chi^{1/2}(\nex) - \chi^{1/2}(\ney)| \leq C_{\chi} |\nex - \ney|, \quad \text{ for all }\nex,\ney\in \Gamma.
	\end{equation}
	Combining \eqref{eq:AddZero} and \eqref{eq:LipschitzBound} into \eqref{eq:Ivarphi}, we obtain
	\begin{equation}\label{eq:I1}
		I_{\varphi}^{(1)} \leq C \left(\chi_{\text{max}} |\varphi|_{H^{1/2}(\Gamma)}^2 + C^2_{\chi} I^{(2)}_{\varphi}\right),
	\end{equation}
	where \begin{subequations}
		\begin{align}\label{eq:I2}
			I^{(2)}_{\varphi} &\coloneqq \displaystyle\int\limits_{\Gamma}\displaystyle\int\limits_{\Gamma}  \frac{|\varphi(\nex)|^2}{|\nex-\ney|}\dsy\dsx \\ \label{eq:rm1-proof}
			&\leq C \int\limits_{\Gamma}|\varphi(\nex)|^2 \left(\int\limits_{\Gamma} \dfrac{1}{|\nex-\ney|} \dsy\right)\dsx  \\
			&\leq C' \norm{\varphi}^2_{L^2(\Gamma)}.
		\end{align}
	\end{subequations}
	due to the integral in \eqref{eq:rm1-proof} being finite, since $\Gamma$ is compact and Lipschitz.\\
	From \eqref{eq:I1} and \eqref{eq:I2} we conclude that there exists a constant $c_{2,\chi}>0$ such that 
	\begin{equation}\label{eq:upper}
		\norm{\chi^{1/2}\varphi}_{H^{1/2}(\Gamma)} \leq c_{2,\chi}\norm{\varphi}_{H^{1/2}(\Gamma)}.
	\end{equation}
	Using \eqref{eq:upper} with $\chi' = \chi^{-1}$ and $\varphi' = \chi^{1/2}\varphi$, we obtain
	\begin{equation}\label{eq:H12}
		c_{1,\chi}\norm{\varphi}_{H^{1/2}(\Gamma)} \leq \norm{\chi^{1/2}\varphi}_{H^{1/2}(\Gamma)} \leq c_{2,\chi}\norm{\varphi}_{H^{1/2}(\Gamma)}.
	\end{equation}
	The proof for $\psi\in H^{-1/2}(\Gamma)$ follows a duality argument. Note that
	\begin{subequations}
		\begin{align}
			\norm{\chi^{1/2}\psi}_{H^{-1/2}(\Gamma)} &= \sup\limits_{\varphi\in H^{1/2}(\Gamma)\setminus\{0\}} \dfrac{\langle \chi^{1/2}\psi, \varphi\rangle_{\Gamma}}{\norm{\varphi}_{H^{1/2}(\Gamma)}} \\
			&=\sup\limits_{\varphi\in H^{1/2}(\Gamma)\setminus\{0\}} \dfrac{\langle \psi, \chi^{1/2}\varphi\rangle_{\Gamma}}{\norm{\varphi}_{H^{1/2}(\Gamma)}} \\
			&\leq \sup\limits_{\varphi\in H^{1/2}(\Gamma)\setminus\{0\}} \norm{\psi}_{H^{-1/2}(\Gamma)}\dfrac{\norm{\chi^{1/2}\varphi}_{H^{1/2}(\Gamma)}}{\norm{\varphi}_{H^{1/2}(\Gamma)}} \\
			&= c_{2, \chi}\norm{\psi}_{H^{-1/2}(\Gamma)}.
		\end{align}
	\end{subequations}
	Repeating the argument with $\chi' = \chi^{-1}$ and $\psi' = \chi^{1/2}\psi,$ we conclude
	\begin{equation*}
		c_{1,\chi}\norm{\psi}_{H^{-1/2}(\Gamma)} \leq \norm{\chi^{1/2}\psi}_{H^{-1/2}(\Gamma)} \leq c_{2,\chi}\norm{\psi}_{H^{-1/2}(\Gamma)}.
	\end{equation*}
\end{proof}
\section{$\mathbf{L}^2$--Projection in $\mathbf{H}(\Curl, \Omega_i)$}\label{sec:L2-proj}

In the scalar case, there is a $h$-uniform discrete inf-sup condition for the dual product between $H^1(\Omega_i)$ and $\widetilde{H}^{-1}(\Omega_i)$, discretized with the finite dimensional space of piecewise-linear continuous functions $P^1_h$ \cite{steinbach2002}. The result is based on the $H^1$-stability of the $L^2$-projection $Q_h : L^2(\Omega_i) \rightarrow P^1_h$ defined as
\begin{equation}
	\langle Q_hu, v_h\rangle_{\Omega_i} = \langle u, v_h\rangle,\quad \text{ for all }v_h\in P^1_h,\ u\in L^2(\Omega_i).
\end{equation}
We know $Q_h$ satisfies (see \cite[Theorem~4.1]{steinbach2002},\cite[Theorem~3]{karkulik20132d},\cite[Section~3]{bramble1991some})
\begin{equation}\label{eq:Qh}
	\norm{Q_hu}_{H^1(\Omega_i)} \leq c_{Q} \norm{u}_{H^1(\Omega_i)}\quad \text{ for all }u\in H^1(\Omega_i).
\end{equation}
The result in \eqref{eq:Qh} is proven by using a quasi-interpolation operator $\widetilde{I}_h$, known to be stable in $H^1(\Omega_i)$ and for which some approximation properties can be shown \cite[Section~3]{steinbach2002}:
\begin{align}
	\norm{\widetilde{I}_h u}_{H^1(\Omega_i)} &\leq c_{\tilde{I}}\norm{u}_{H^1(\Omega_i)},\\
	\norm{u-\widetilde{I}_hu}_{L^2(\Omega_i)} &\leq c_{I\!I} h |u|_{H^1(\Omega_i)},
\end{align}
for all $u\in H^1(\Omega_i)$.\\
Assuming a quasi-uniform and shape-regular family of meshes, it follows that
\begin{align*}
	\norm{Q_h u}_{H^1(\Omega_i)} &\leq \norm{Q_h u - \widetilde{I}_h u + \widetilde{I}_h u}_{H^1(\Omega_i)} \\
	&\leq \norm{(Q_h  - \widetilde{I}_h) u}_{H^1(\Omega_i)} + \norm{\widetilde{I}_h u}_{H^1(\Omega_i)} \\
	&\leq \tfrac{1}{h}\norm{(Q_h  - \widetilde{I}_h) u}_{L^2(\Omega_i)} + \norm{\widetilde{I}_h u}_{H^1(\Omega_i)} \\
	&\leq \tfrac{1}{h}\norm{u - Q_hu}_{L^2(\Omega_i)} +\tfrac{1}{h}\norm{u - \widetilde{I}_hu}_{L^2(\Omega_i)}+ c_{\tilde{I}}\norm{u}_{H^1(\Omega_i)}\\ \label{eq:CrucialH1}
	&\leq C|u|_{H^1(\Omega_i)} + c_{I\!I}|u|_{H^1(\Omega_i)}+ c_{\tilde{I}}\norm{u}_{H^1(\Omega_i)} \ \leq c_{Q}\norm{u}_{H^1(\Omega_i)}
\end{align*}
The fundamental step in this proof is: being able to bound the $L^2$-error with the $H^1$-seminorm. \\
Is it possible to have a similar result in $\mathbf{H}(\Curl, \Omega_i)$ with its seminorm? The answer is no. Consider $\bm{w}\in \mathbf{H}(\Curl0, \Omega_i)$ and $\mathbf{Q}_h : \mathbf{H}(\Curl,\Omega_i)\rightarrow N_h \subset \mathbf{H}(\Curl, \Omega_i)$ the standard $\mathbf{L}^2$-projection into N\'ed\'elec edge elements. Then, such a result requires
\begin{equation}
	\norm{\bm{w} - \mathbf{Q}_h \bm{w}}_{\mathbf{L}^2(\Omega_i)} \leq C|\bm{w}|_{\mathbf{H}(\Curl,\Omega_i)} = 0,
\end{equation}
which can only be true for constants or polynomials in $N_h$, but as we know, $\mathbf{H}(\Curl0, \Omega_i)$ is an infinite dimensional subspace of $\mathbf{H}(\Curl,\Omega_i).$ Therefore, such a proof is not valid for the standard $\mathbf{L}^2$-projection $\mathbf{Q}_h$.\\
Some numerical evidence of this issue and its implications will be shown in Appendix \ref{sec:L2-projection}.

\subsection{Numerical experiment}\label{sec:L2-projection}
In this section, we study the convergence in the $\mathbf{H}(\Curl, \Omega_i)$ norm of the $\mathbf{L}^2$-projection $\mathbf{Q}_h$, defined as
\begin{equation}
	\langle \mathbf{Q}_h \neu, \nev_h\rangle_{\Omega_i} = \langle \neu, \nev_h \rangle_{\Omega_i} \quad \text{ for all }\nev_h\in N_h, \ \neu\in \mathbf{L}^2(\Omega_i),
\end{equation}
where $N_h = N_h(\mathcal{T}_h)$ is the finite dimensional space of N\'ed\'elec edge functions in a tetrahedral mesh $\mathcal{T}_h$ of $\Omega_i.$\\
In particular, we consider 
\begin{equation*}
	\Omega_i\coloneqq \{  \nex = (x,y,z)\in \mathbb{R}^3 : 0\leq x,y,z\leq 1 \},
\end{equation*}
and 
\begin{equation*}
	\neu(\nex) = \neu^{\star}(\nex)\coloneqq \bm{e}_0 \exp(i\kappa_0 \nex \cdot \nex_0),
\end{equation*}
where $\kappa_0 = 2, \ \nex_0 = (0, 1, 0)$ and $\bm{e}_0 = (1, 0, 0).$ We can observe in Figure \ref{fig:L2-projection} the errors of the projections $\mathbf{Q}_h\neu^{\star}$ in the $\mathbf{L}^2(\Omega_i)$ and $\mathbf{H}(\Curl, \Omega_i)$ norms. \\

\begin{figure}[ht!]
	\centering 
	\includegraphics[width=0.5\linewidth]{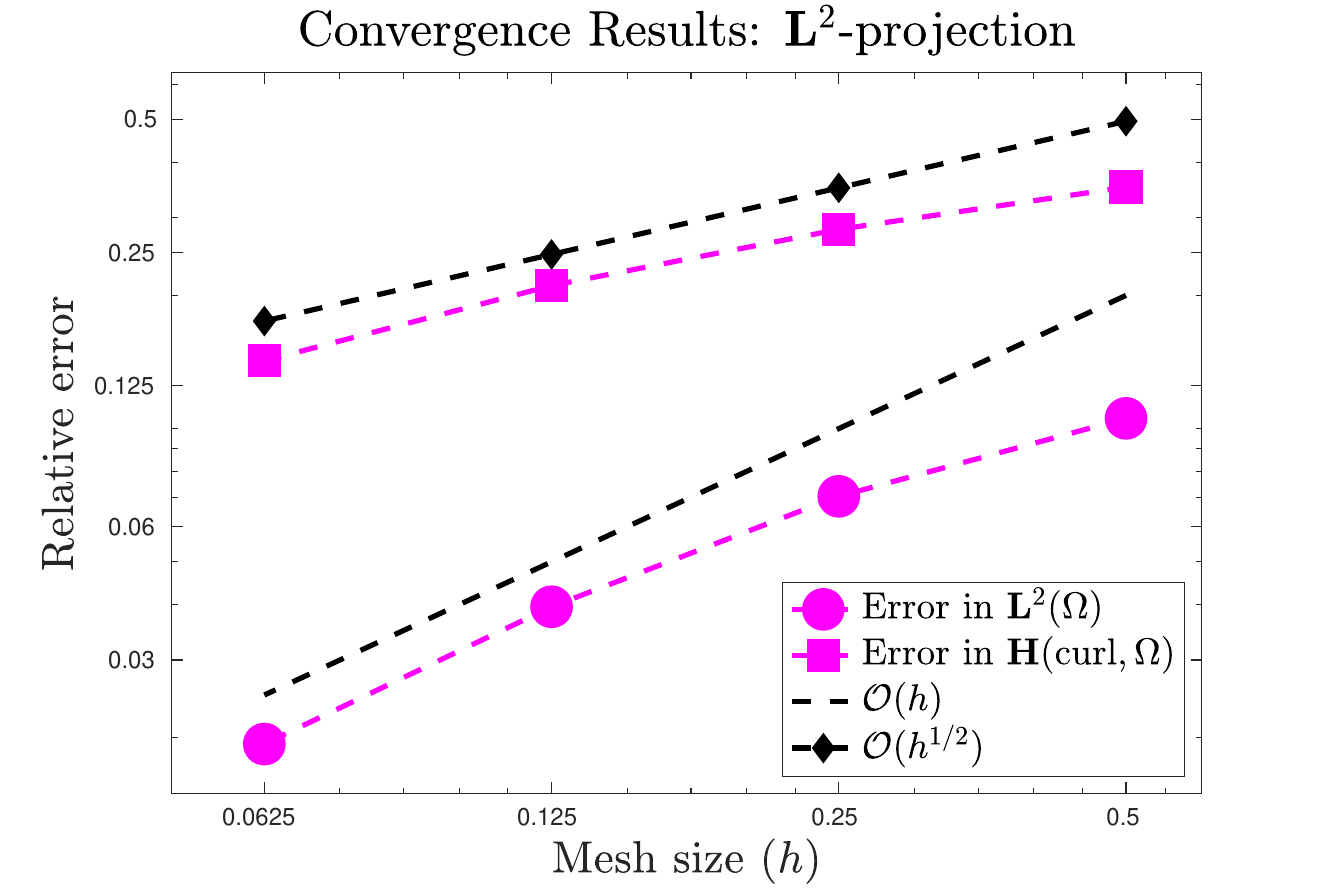}	
	\caption{$\mathbf{L}^2$-projection $\mathbf{Q}_h$, Section \ref{sec:L2-projection}. Error norms \eqref{eq:error-norms} as functions of $ h $.}
	\label{fig:L2-projection}
\end{figure}
It is a well known result that the approximation error for $\mathbf{L}^2(\Omega_i)$ and $\mathbf{H}(\Curl, \Omega_i)$ in formulations that are stable in $\mathbf{H}(\Curl,\Omega_i)$ has the same convergence rate, due to the interpolation estimates being the same \cite[Remark~10]{hiptmair2002finite}.\\

Assume $\norm{\mathbf{Q}_h \neu}_{\mathbf{H}(\Curl,\Omega_i)}\leq C_S \norm{\neu}_{\mathbf{H}(\Curl,\Omega_i)}$. Then it follows
\begin{align}\label{eq:uminusQhu}
	\begin{aligned}
		\norm{\neu - \mathbf{Q}_h\neu}_{\mathbf{H}(\Curl,\Omega_i)} &= \norm{\neu - \mathbf{Q}_h\neu - \nev_h + \mathbf{Q}_h\nev_h}_{\mathbf{H}(\Curl,\Omega_i)} \\
		&= \norm{(\mathbf{I} - \mathbf{Q}_h)(\neu-\nev_h)}_{\mathbf{H}(\Curl,\Omega_i)} \\
		&\leq (1+C_S)\norm{\neu-\nev_h}_{\mathbf{H}(\Curl,\Omega_i)}
	\end{aligned}
\end{align}
for all $\nev_h\in N_h.$ From \eqref{eq:uminusQhu} we obtain that for smooth vector fields $\neu\in\mathbf{H}^1(\Curl,\Omega_i)$ it holds
\begin{equation*}
	\norm{\neu - \mathbf{Q}_h\neu}_{\mathbf{H}(\Curl,\Omega_i)} \leq (1+C_S)\inf\limits_{\nev_h\in N_h}\norm{\neu - \nev_h}_{\mathbf{H}(\Curl,\Omega_i)} = \mathcal{O}(h)
\end{equation*}
on shape-regular and quasi-uniform families of meshes.

As we observe in Figure \ref{fig:L2-projection}, there is a reduced order of convergence of the $\mathbf{L}^2$-projection in the $\mathbf{H}(\Curl, \Omega_i)$ norm. We conclude that 
\begin{equation}
	\dfrac{\norm{\mathbf{Q}_h\neu}_{\mathbf{H}(\Curl,\Omega_i)}}{\norm{\neu}_{\mathbf{H}(\Curl,\Omega_i)}} \text{ is not bounded uniformly in }h.
\end{equation}

\end{document}